\documentclass[11pt]{article}
\usepackage{amsmath,amsthm,amssymb,amsfonts}
\usepackage{enumerate}
\usepackage{mathrsfs}
\usepackage{caption}
\usepackage[width=15cm,height=21cm]{geometry}
\usepackage{tikz}
\usepackage{tikz-cd}
\usetikzlibrary{shapes.geometric}

\numberwithin{equation}{section}
\numberwithin{figure}{section}

\newtheoremstyle{theoremstyle}
  {10pt}      
  {5pt}       
  {\itshape}  
  {}          
  {\bfseries} 
  {:}         
  {.5em}      
  {}          

\newtheoremstyle{examplestyle}
  {10pt}      
  {5pt}       
  {}          
  {}          
  {\bfseries} 
  {:}         
  {.5em}      
  {}          

\theoremstyle{theoremstyle}
\newtheorem{theorem}{Theorem}[section]
\newtheorem*{theorem*}{Theorem}
\newtheorem{lemma}[theorem]{Lemma}
\newtheorem{proposition}[theorem]{Proposition}
\newtheorem*{proposition*}{Proposition}
\newtheorem{corollary}[theorem]{Corollary}
\newtheorem*{corollary*}{Corollary}

\theoremstyle{examplestyle}

\newtheorem{definition}[theorem]{Definition}
\newtheorem{definition*}{Definition}
\newtheorem{remark}[theorem]{Remark}
\newtheorem{remark*}{Remark}

\let\smash=\wedge
\let\iso=\cong

\let\directsum=\oplus
\let\minus=\smallsetminus
\let\ul=\underline

\definecolor{sq3sq1color}{rgb}{0.5,0,0.5}
\definecolor{sq2color}{rgb}{0.1,0.7,0.1}
\colorlet{taucolor}{red}
\colorlet{partialsq2sq1color}{sq3sq1color!53!black}
\colorlet{incsq2sq1color}{sq3sq1color!67!green}
\colorlet{sq2rhosq1color}{taucolor!45!sq2color}
\colorlet{sq2prcolor}{sq2color!53!black}
\colorlet{incsq2color}{sq2color!67!yellow}
\colorlet{sq2partialcolor}{sq2color!42!blue}
\colorlet{tauprcolor}{taucolor!53!black}
\colorlet{taupartialcolor}{taucolor!42!yellow}

\newcommand{\unit}{\mathbf{1}}
\newcommand{\SH}{\mathbf{SH}}

\newcommand{\forget}{\mathrm{forget}}
\newcommand{\hyper}{\mathsf{h}}

\newcommand{\bideg}{{(\star)}}

\newcommand{\im}{\operatorname{im}}
\newcommand{\holim}{\operatornamewithlimits{holim}}
\newcommand{\Ext}{{\operatorname{Ext}}}
\newcommand{\Spec}{{\operatorname{Spec}}}
\newcommand{\colim}{\operatornamewithlimits{colim}}

\newcommand{\PP}{\mathbf{P}}
\newcommand{\MGL}{\mathbf{MGL}}

\newcommand{\MU}{\mathbf{MU}}
\newcommand{\BP}{\mathbf{BP}}

\newcommand{\KQ}{\mathbf{KQ}}
\newcommand{\kq}{\mathbf{kq}}
\newcommand{\ko}{\mathbf{ko}}
\newcommand{\KW}{\mathbf{KW}}
\newcommand{\kw}{\mathbf{kw}}

\newcommand{\KGL}{\mathbf{KGL}}
\newcommand{\kgl}{\mathbf{kgl}}
\newcommand{\ku}{\mathbf{ku}}

\newcommand{\E}{\mathsf{E}}
\newcommand{\M}{\mathbf{M}}

\newcommand{\Witt}{\mathbf{W}}
\newcommand{\fundid}{\mathbf{I}}

\newcommand{\dd}{\mathsf{d}}

\newcommand{\s}{\mathsf{s}}
\newcommand{\f}{\mathsf{f}}

\newcommand{\image}{\mathrm{Im}}

\newcommand{\Mil}{\mathsf{M}}
\newcommand{\MilWitt}{\mathsf{MW}}

\newcommand{\KMil}{\mathbf{K}^{\Mil}}

\newcommand{\kmil}{\mathbf{k}^{\Mil}}

\newcommand{\kerpart}{\mathbf{d}}
\newcommand{\KMW}{\mathbf{K}^{\MilWitt}}

\newcommand{\NN}{{\mathbb N}}
\newcommand{\ZZ}{{\mathbb Z}}
\newcommand{\QQ}{\mathbb{Q}}

\newcommand{\Z}{\mathbb{Z}}

\newcommand{\RR}{\mathbb{R}}
\newcommand{\CC}{\mathbb{C}}
\newcommand{\sphere}{\mathbb{S}}

\newcommand{\pr}{{\mathrm{pr}}}
\newcommand{\id}{{\mathrm{id}}}

\newcommand{\inc}{{\mathrm{inc}}}

\newcommand{\A}{\mathbf{A}}

\newcommand{\Char}{\mathsf{char}}

\newcommand{\MZ}{\mathbf{M}\mathbb{Z}}
\newcommand{\HZ}{\mathbf{H}\mathbb{Z}}

\newcommand{\Sq}{\mathsf{Sq}}

\newcommand{\Top}{\mathsf{top}}

\newcommand{\Sm}{\mathbf{Sm}}

\newcommand{\slicecomp}{\mathsf{sc}}

\newcommand{\Hom}{\operatorname{Hom}}

\title{{\bf The second stable homotopy groups \\ of motivic spheres}}
\author{Oliver R\"ondigs, Markus Spitzweck, Paul Arne {\O}stv{\ae}r}
\date{April 1, 2021} 
\begin{document}
\maketitle
\begin{abstract}
We compute the $2$-line of stable homotopy groups of motivic spheres over fields of characteristic not two in terms 
of motivic cohomology and hermitian $K$-groups.
\end{abstract}

{\footnotesize{\tableofcontents}}
\newpage

\section{Introduction}
\label{section:introduction}

One of the fundamental problems in motivic homotopy theory over a field $F$ is to calculate the homotopy groups 
$\pi_{t,w}\unit$ of the motivic sphere spectrum $\unit$.
Here the integers $t,w\in\ZZ$ refer to the topological degree and the weight, respectively.
In a precise sense, 
these groups are the universal motivic invariants because the motivic sphere is the unit for the tensor product 
on motivic spectra.
All the relations witnessed in $\pi_{\ast,\ast}\unit$ hold in every other theory representable in the 
stable motivic homotopy category, 
such as algebraic cobordism, algebraic and hermitian $K$-theory, motivic cohomology, and higher Witt theory.  
By work of Morel \cite{MorelICM2006}, 
the group $\pi_{s,t}\unit=0$ if $s<t$ and the so-called $0$-line $\bigoplus_{n\in\ZZ}\pi_{n,n}\unit$ is the 
Milnor-Witt $K$-theory $\KMW(F)$ of $F$, 
which incorporates a great deal of arithmetic information about the base field.
It turns out that the Grothendieck-Witt group of isomorphism classes of symmetric bilinear forms features 
very naturally in this calculation; 
a surprising fact since the construction of motivic homotopy theory does not involve quadratic forms.
Over the complex numbers, 
the Betti realization functor witnesses all the classical stable homotopy groups as the weight zero part of the 
corresponding groups of the motivic sphere \cite{levine.comparison}.
The key inputs enabling the connection between motivic homotopy groups and quadratic forms are the resolutions 
of the Bloch-Kato, and Milnor conjectures on Galois cohomology and quadratic forms in
\cite{Voevodsky.reduced},  \cite{ovv}, and \cite{MR2811603}.

In \cite{rso.oneline} we extended these calculations to the $1$-line $\bigoplus_{n\in\ZZ}\pi_{n+1,n}\unit$ 
of the stable motivic homotopy groups, 
for every field $F$ of characteristic unequal to two.
Our approach makes a systematic analysis of the slice spectral sequences for the sphere and hermitian $K$-theory 
\cite{voevodsky.open}, \cite{roendigs-oestvaer.hermitian}.
To "turn the pages" in the slice spectral sequence for the sphere, we calculate sufficiently many differentials 
to deduce that it collapses at its $E^{2}$-page,
at least as far as the $1$-line is concerned.
A case-by-case analysis achieves this, which combines Voevodsky's description of the motivic Steenrod algebra 
\cite{Voevodsky.reduced}, \cite{hko.positivecharacteristic}, 
with the observation that the slice differentials induce graded Milnor $K$-theory $\KMil(F)$-module 
maps \cite{milnor.k-quadratic}.

This paper deepens our understanding of the stable motivic homotopy groups by means of an explicit calculation of 
the $2$-line $\bigoplus_{n\in\ZZ}\pi_{n+2,n}\unit$ over fields of characteristic different from two.
To that end, 
we develop powerful computational techniques taking into account the Milnor-Witt $K$-theory module structure on 
successive stages of the slice filtration.
Owing to a comparison between the motivic sphere and very effective hermitian $K$-theory, 
we also strengthen the main result in \cite{rso.oneline}.
Moving up from the $1$- to the $2$-line increases the computational complexity, 
which is rooted in the problem of controlling mod-$n$ motivic cohomology groups as $n$ increases.

Suppose $S$ is a base scheme and $\Lambda$ is a localization of $\mathbb{Z}$ such that $(S,\Lambda)$ 
is a compatible pair in the sense of \cite[\S2.1]{rso.oneline},
e.g., 
$S$ is smooth over a field or a Dedekind domain of mixed characteristic such that every positive residue characteristic 
of $S$ is invertible in $\Lambda$.
Later we will restrict to the case of a field $F$ of exponential characteristic $c\geq 1$ and 
$\Lambda=\mathbb{Z}[\tfrac{1}{c}]$.
By Theorem 2.12 in \cite{rso.oneline} the slices of the $\Lambda$-local sphere spectrum are given by a 
finite wedge product of suspensions of motivic Eilenberg-MacLane spectra
\begin{equation}
\label{equation:sphereslices}
\s_{q}(\unit_{\Lambda})
\cong
\bigvee_{p\geq 0} \Sigma^{2q-p,q}\mathbf{M}(\Lambda \otimes\Ext_{\MU_{\ast}\MU}^{p,2q}(\MU_{\ast},\MU_{\ast})).
\end{equation}
Here, 
the extension groups are calculated in comodules over the Hopf algebroid for the cobordism spectrum $\MU$;
these form the $E^{2}$-page of the classical Adams-Novikov spectral sequence \cite{ravenel.green}.
Voevodsky's slice spectral sequence \cite{voevodsky.open} takes \eqref{equation:sphereslices} as its input
\begin{equation}
\label{equation:splicespectralsequence}
E^{1}_{t,q,w}=
\pi_{t,w}\s_{q}(\unit_{\Lambda})
\Longrightarrow
\pi_{t,w}\unit_{\Lambda}.
\end{equation}

To formulate our results, we compare the motivic sphere spectrum with hermitian $K$-theory $\KQ$, or rather its
very effective cover $\kq$. Over a perfect field, 
$\kq$ coincides with the effective cover $\f_{0}(\KQ_{\geq 0})$ with respect to the slice filtration by $\kq$. 
Here $\KQ_{\geq 0}$ is the 
connective cover of $\KQ$ with respect to the homotopy $t$-structure on the stable motivic homotopy category.
For a compatible pair $(S,\Lambda)$ as above, 
the slices of $\kq$ have the following form.
\begin{equation}
\label{equation:slices-kq} 
\s_{q}(\kq_{\Lambda})
\cong
\begin{cases}
\Sigma^{2q,q}\mathbf{M}\Lambda
\vee  
\bigvee_{0\leq i<\frac{q}{2}}\Sigma^{2i+q,q}\mathbf{M}\Lambda/2 & 0\leq q\equiv 0\bmod 2 \\
\bigvee_{0\leq i<\frac{q+1}{2}}\Sigma^{2i+q,q}\mathbf{M}\Lambda/2 & 0\leq q\equiv 1\bmod 2\\
\ast & q<0
\end{cases}
\end{equation}
The canonical map $\kq_{\Lambda}\to \KQ_{\Lambda}$ induces an inclusion on all slices.
The infinite wedge sum in the identification of $\s_{q}(\KQ_{\Lambda})$ in 
\cite[Theorem 4.18]{roendigs-oestvaer.hermitian}, \cite[Theorem 2.15]{rso.oneline} 
explains to some extent why we prefer to use $\kq_{\Lambda}$ in our comparison with $\unit_{\Lambda}$.
\vspace{0.1in}

A refinement of \cite[Theorem 5.6]{rso.oneline}, see Section \ref{section:1line}, yields the following calculation.
\begin{theorem}
\label{theorem:1line}
Let $F$ be a field of characteristic zero.
For every integer $n$,
the unit map $\unit\to \kq$ induces a short exact sequence
\[ 
0 
\rightarrow 
\KMil_{2-n}(F)/24
\rightarrow 
\pi_{n+1,n} \unit
\rightarrow 
\pi_{n+1,n}\kq
\rightarrow
0.
\]
\end{theorem}
In particular, this proves Morel's $\pi_{1}$-conjecture, 
i.e., 
there is a short exact sequence
\begin{equation}
\label{eq:first-stable-stem2} 
0 
\to 
\KMil_{2}(F)/24 
\to 
\pi_{1,0}\unit 
\to 
F^{\times}/2\oplus\Z/2
\to
0.
\end{equation}
With reference to the $\KMW(F)$-module structure, 
the second motivic Hopf map $\nu\in\pi_{3,2}\unit$ generates $\mathbf{K}^\M_{2}(F)/24$ 
and the topological Hopf map $\eta_{\Top}\in\pi_{1,0}\unit$ generates $F^{\times}/2\oplus\Z/2$.
These generators are subject to the relations $2\eta_\Top=0$  and 
$\eta^{2}\eta_{\Top}=6(1-\epsilon)\nu=12\nu$, implying
$24\nu=0$ (as predicted by a motivic analogue of the Adams conjecture in this degree).
Here $\eta\in\pi_{1,1}\unit$ and $\hyper=1-\epsilon\in\pi_{0,0}\unit$ are the first and zeroth motivic Hopf maps.
By specializing to the field of real numbers $\RR$, 
we can distinguish between the elements $\rho^{2}\eta_{\mathrm{top}}$ and $\rho^4\nu$ in the integral group 
$\pi_{-1,-2}\unit_{\RR}$, 
which cannot be witnessed after 2-completion according to \cite{dugger-isaksen.real}.
On the other hand we have $\rho^3\eta_{\mathrm{top}} = \rho^5\nu$ in $\pi_{-2,-3}\unit_{\RR}$.
\vspace{0.1in}

Theorem \ref{theorem:1line} holds more generally over any field $F$ of exponential characteristic $c\neq 2$.
For $\Lambda=\mathbb{Z}[\tfrac{1}{c}]$ there is an exact sequence
\begin{equation}
\label{equation:first-stable-stem2}
0 
\rightarrow 
\KMil_{2-n}(F)\otimes \Lambda/24
\rightarrow 
\pi_{n+1,n} \unit_{\Lambda}
\rightarrow 
\pi_{n+1,n}\kq_{\Lambda}
\rightarrow
0.
\end{equation}
It is unclear whether the restriction on $c$ is necessary.
To begin with, 
an identification of the motivic Steenrod algebra at the characteristic, 
extending \cite{Voevodsky.reduced} and \cite{hko.positivecharacteristic}, 
may be required for answering this question.
\vspace{0.1in}

In this paper we continue the work in \cite{rso.oneline} by calculating the $2$-line.
\begin{theorem}
\label{theorem:2line}
Let $F$ be a field of characteristic zero.
For every integer $n$,
the unit map $\unit\to \kq$ induces a short exact sequence
\[ 
0 
\rightarrow 
H^{1-n,2-n}(F)/24\directsum \KMil_{4-n}(F)/2
\rightarrow 
\pi_{n+2,n} \unit
\rightarrow 
\pi_{n+2,n}\kq.
\] 
In particular, 
$\pi_{n+2,n}\unit = 0$ for $n\geq 5$ and $\pi_{6,4}\unit\iso \ZZ/2$.
\end{theorem}
Here $H^{1-n,2-n}(F)$ refers to integral motivic cohomology.
Its mod-$24$ reduction maps canonically to $\pi_{n+2,n}\unit$, 
which contains a quotient of $h^{1-n,2-n}_{24}$.  
Theorem \ref{theorem:2line} holds more generally for fields with $\Lambda=\mathbb{Z}[\tfrac{1}{c}]$-coefficients, 
$c\neq 2$, 
as in \eqref{equation:first-stable-stem2}.
The $\KMW(F)$-module $\mathbf{K}^{\M}_{4-n}/2$ is generated by the motivic Hopf map $\nu^{2}$. 
On the $2$-line, we find the relation
\[ 
\rho^{2}\nu^{2} = \eta_{\mathrm{top}}\nu\in \pi_{4,2}\unit. 
\]

For $n=2$, 
the rightmost map in Theorem~\ref{theorem:2line} is not surjective since $\pi_{4,2}\unit$ is torsion and 
$\pi_{4,2}\kq$ is isomorphic to the zeroth symplectic Grothendieck group $\mathbf{K}^{\mathrm{Sp}}_{0}(F) \cong 2\ZZ$ of even integers. 
For $n=1$,
there is a split short exact sequence
\[ 
0 \to \mathbf{K}^{\M}_{3}(F)/2 \to \pi_{3,1}\unit \to \mu_{24} (F)\to 0. 
\]
Here $\pi_{3,1}\unit \to \pi_{3,1}\kq = H^{1,1}$ is given by the inclusion of the $24$-th roots of unity into $F^{\times}$. 
Since $\pi_{0,-1}\unit_\CC = 0$ and $\pi_{0,1}\unit_\CC=\pi_{2,1}\unit_\CC = \ZZ/2$, 
the product map from the $1$-line to the $2$-line is not surjective over the complex numbers, 
contrary to the topological situation.
\vspace{0.1in}

We carry out our calculations of stable motivic stems on $F$-points, 
but the results hold on the level of motivic homotopy sheaves since these are strictly $\mathbf{A}^1$-invariant 
and hence unramified sheaves on $\Sm_{F}$, 
see \cite[Theorem 6.2.7]{morel.connectivity}, \cite[Theorem 2.11]{morel.field}.  
\vspace{0.1in}

Throughout the paper, we employ the following notation.
\vspace{0.1in}

\begin{tabular}{l|l}
$F$, $S$ & field, finite dimensional separated Noetherian scheme \\
$\Sm_{S}$ & smooth schemes of finite type over $S$ \\
$S^{s,t} = S^{s-t+(t)}$ &  motivic $(s,t)=s-t+(t)$-sphere \\
$\Sigma^{s,t} = \Sigma^{s-t+(t)}$ & suspension with $S^{s,t}=S^{s-t+(t)}$ \\
$\SH(S)$  & motivic stable homotopy category over $S$ \\ 
$\E$, $\unit=S^{0,0}$ & generic motivic spectrum, the motivic sphere spectrum  \\
$R$, $\mathbf{M}R$ & ring, motivic Eilenberg-MacLane spectrum of $R$ \\
$\MGL$, $\MU$, $\BP$ & algebraic and complex cobordism, \\ 
& Brown-Peterson spectrum \\
$\KGL$, $\KQ$, $\KW$ & algebraic and hermitian $K$-theory, Witt-theory \\
$\f_{q}$, $\s_{q}$ & $q$th effective cover and slice functors \\
$\KMW$, $\KMil$, $\kmil\!=\KMil\!/2$ & Milnor-Witt $K$-theory, Milnor $K$-theory (modulo 2) \\
$H^{\ast,\ast}$, $h^{\ast,\ast}$, $h^{\ast,\ast}_{n}$ & integral, mod-$2$, 
mod-$n$ motivic cohomology groups  \\
$\partial^a_b\colon h_{a}^{s,t} \to h^{s+1,t}_{b}$ & connecting homomorphism, 
$a,b\in \NN\cup \{\infty\}$, $h_\infty = H$ \\
$ \inc^a_b,\pr^a_b\colon h^{s,t}_{a}\to h^{s,t}_{b}$ & inclusion, projection homomorphism, \\ 
& $a,b\in \NN\cup \{\infty\}$, $h_\infty = H$.
\end{tabular}
\vspace{0.05in}
\noindent

The ring $R$ is commutative and unital.
For legibility we suppress $\Lambda=\mathbb{Z}[\tfrac{1}{c}]$-coefficients in our notation.
Our standard conventions are 
$\mathbf{P}^1\simeq \mathbf{T}\simeq S^{2,1}$ and $\mathbf{A}^{1}\minus \{0\}\simeq S^{1,1}$ 
for the Tate object $\mathbf{T}=\mathbf{A}^{1}/\mathbf{A}^{1}\minus \{0\}$.
Let $S$ be a scheme. For a motivic spectrum $\E\in \SH(S)$
and integers $s,w\in \ZZ$,
let $\pi_{s,w} \E$ denote the abelian group 
$[\Sigma^{s,w}\unit,\E]$, where $\E$ is a motivic spectrum and
$\unit_F=\unit$ is the motivic sphere spectrum. The grading conventions
are such that the suspension functor $\Sigma^{2,1}=\Sigma^{1+(1)}$
is suspension with $\PP^1$, and $\Sigma^{1,0}=\Sigma^{1+(0)}=\Sigma^1=\Sigma$
is suspension with the simplicial circle. Set $\pi_{s+(w)}\E:=\pi_{s+w,w}\E$, 
and let
\[ \pi_{s+\bideg}\E := \bigoplus_{w\in \ZZ} \pi_{s+w,w}\E \]
denote the direct sum, considered as a $\ZZ$-graded module over the
$\ZZ$-graded ring $\pi_{0+\bideg}\unit$. The notation
$\pi_{s-\bideg}\E:=\pi_{s+(-\star)}\E$ will be used frequently, as
well as the abbreviation $\pi_{s}\E:=\pi_{s+\bideg}\E$.
Note that $\pi_{s}\E$ is canonically a (left) $\pi_{0}\unit$-module
for every $s\in \ZZ$. Over a field $F$, the graded ring
$\pi_{0}\unit_F\iso \KMW(F)$ is the Milnor-Witt $K$-theory of $F$ by Morel's theorem \cite{morelmotivicpi0}.
Hence $\pi_{s}\E$ is naturally a $\KMW$-module for every $s\in \ZZ$.
The strictly 
$\A^1$-invariant sheaf obtained as the associated Nisnevich sheaf
of $U\mapsto \pi_{s,w} \E_U$ for $U\in \Sm_F$ is denoted
$\underline{\pi}_{s,w}\E$, which gives rise to
the homotopy module $\underline{\pi}_{s+\bideg}\E$.
As mentioned already, our main results can be suitably
reinterpreted as statements regarding homotopy modules.
We write $\dd^{\E}_{1}(q)\colon\s_{q}(\E)\to\Sigma^{1,0}\s_{q+1}(\E)$ or simply $\dd^{\E}_{1}$ for the first slice 
differential of $\E$ \cite[\S2]{roendigs-oestvaer.hermitian}, \cite[\S7]{voevodsky.open}.
For $r\geq 1$, 
we write $d^{r}_{p,q,w}(\E)\colon E^{r}_{p,q,w}(\E)\to E^{r}_{p-1,q+r,w}(\E)$ or simply $d^{r}(\E)$ 
for the $r$th differential in the weight $w$ slice spectral sequence.

\section{Very effective hermitian $K$-theory}
\label{sec:very-effect-herm}

The path to computations for the sphere spectrum proceeds via a
comparison with the very effective cover $\kq\to \KQ$ of the motivic ring
spectrum representing hermitian $K$-theory that is classifying vector bundles equipped with quadratic forms. 
Over perfect fields of
characteristic not two, it can be identified as
$\kq=\f_{0}(\KQ_{\geq 0})$, the effective cover of the connective cover of 
$\KQ$ \cite{bachmann.generalized}, although it can be defined formally
over any base scheme for which $\KQ$ is defined 
\cite{SO}.\footnote{The Ph.D. thesis \cite{kumar} provides a motivic spectrum over $\Spec(\ZZ)$, and hence any base scheme, 
which pulls back to $\KQ$ over any scheme in which $2$ is invertible.}
Since $\unit$ is very effective, the unit map for $\KQ$ factors uniquely over $\kq$.
The computations for $\kq$ require an analysis of the
slice spectral sequence for $\kq$ extending the results in \cite{aro.kq}.
Its convergence will be necessary for applying these slice spectral sequence computations and hence is discussed first. 
When the base scheme $S$ has no points of residue field characteristic two, 
let $\tau$ denote the generator of $h^{0,1}\cong\mu_{2}(\mathcal{O}_{S})$ and $\rho$ denote the class of $-1$ in $h^{1,1}\cong \mathcal{O}^{\times}_{S}/2$.
We denote motivic Steenrod operations by $\Sq^{i}$ \cite[\S9]{Voevodsky.reduced}. Let 
$\kgl:=\f_{0}\KGL$ be the (very) effective cover of the motivic ring spectrum $\KGL$ classifying vector bundles. 
See \cite{aro.kq} for a proof of the next result.

\begin{theorem}
\label{thm:kq-wood}
Multiplication with the Hopf map $\eta$ yields a homotopy cofiber sequence
\[ 
\Sigma^{1,1}\kq 
\xrightarrow{\eta} 
\kq 
\xrightarrow{\forget}
\kgl 
\xrightarrow{\mathrm{hyper}}
\Sigma^{2,1}\kq.
\]
\end{theorem}

\begin{corollary}
\label{cor:eta-slice-comp-kq}
There is a canonical equivalence between the $\eta$-completion of $\kq$ and its slice completion
\[ 
\kq^{\wedge}_{\eta} 
\simeq 
\slicecomp(\kq). 
\]
\end{corollary}
\begin{proof}
By \cite[Lemma 3.13]{rso.oneline} and Theorem \ref{thm:kq-wood} it suffices to prove that $\kgl$ is slice complete. 
This follows via $\holim_q\f_{q}\kgl \simeq \holim_q \f_{q}\KGL \simeq \ast$
since $\KGL$ is slice complete by \cite{Voevodsky:motivicss}. 
\end{proof}

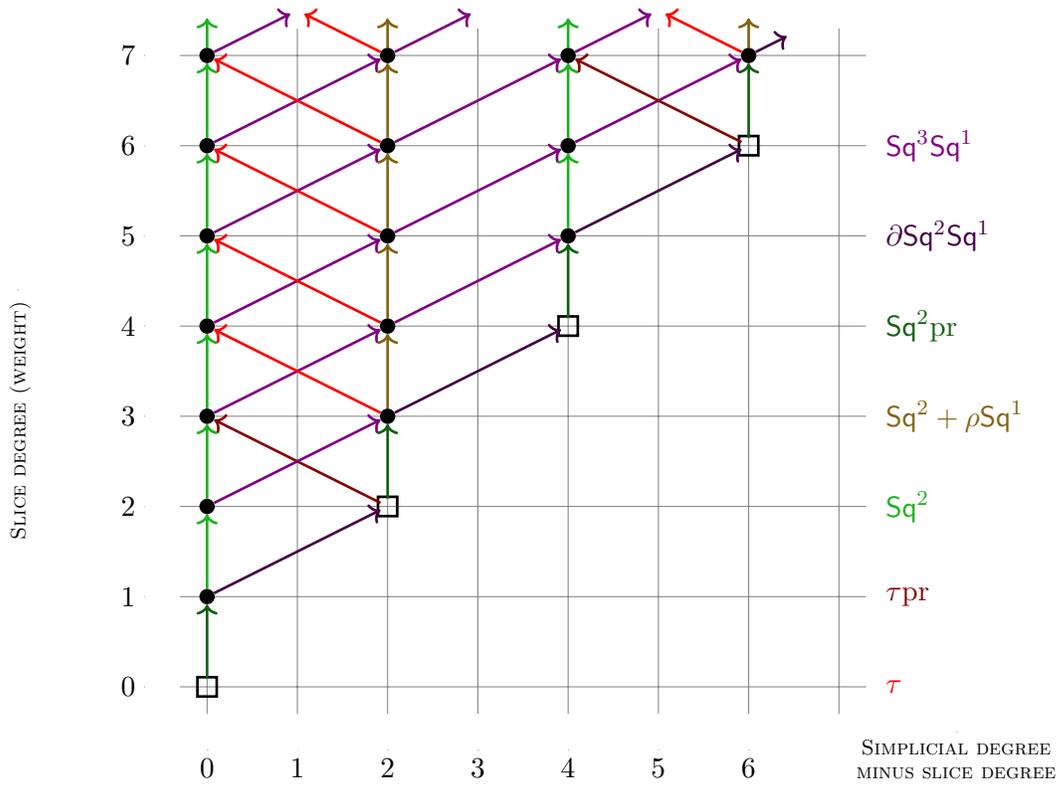
\begin{figure}
  \pgfsetshortenend{3pt}
  \pgfsetshortenstart{3pt}
  \begin{tikzpicture}[scale=1.2,line width=1pt]
    \draw[help lines] (-4.3,-2.3) grid (3.3,5.3);
    {\draw[fill]     
      (-4.7,-2) circle (0pt) node[left=-1pt] {$0$}
      ;}
    \node[rectangle, draw] at (-4,-2) {};
    {\draw[fill]     
      (-4,-1) circle (2pt) 
      (-4.7,-1) circle (0pt) node[left=-1pt] {$1$}
      ;}
    {\draw[fill]     
      (-4,0) circle (2pt) 
      (-4.7,0) circle (0pt) node[left=-1pt] {$2$}
      ;}
    \node[rectangle, draw] at (-2,0) {};
    {\draw[fill]     
      (-4,1) circle (2pt) (-2,1) circle (2pt) 
      (-4.7,1) circle (0pt) node[left=-1pt] {$3$}
      ;}
    {\draw[fill]     
      (-4,2) circle (2pt) (-2,2) circle (2pt) 
      (-4.7,2) circle (0pt) node[left=-1pt] {$4$}
      ;}
    \node[rectangle, draw] at (0,2) {};
    {\draw[fill]     
      (-4,3) circle (2pt) (-2,3) circle (2pt) (0,3) circle (2pt)  
      (-4.7,3) circle (0pt) node[left=-1pt] {$5$}
      ;}
    {\draw[fill]     
      (-4,4) circle (2pt) (-2,4) circle (2pt) (0,4) circle (2pt) 
      (-4.7,4) circle (0pt) node[left=-1pt] {$6$}
      ;}
    \node[rectangle, draw] at (2,4) {};

    {\draw[fill]     
      (-4,5) circle (2pt) (-2,5) circle (2pt) (0,5) circle (2pt) (2,5) circle (2pt) 
      (-4.7,5) circle (0pt) node[left=-1pt] {$7$}
      ;}
    {\draw[fill]     
      (-6.1,2.4) circle (0pt) node[left=-1pt,rotate=90] {\sc \scriptsize Slice degree (weight)}
      ;}

    {\draw[fill]     
    (-4,-2.7) circle (0pt) node[below=-1pt] {$0$}
    (-3,-2.7) circle (0pt) node[below=-1pt] {$1$}
    (-2,-2.7) circle (0pt) node[below=-1pt] {$2$}
    (-1,-2.7) circle (0pt) node[below=-1pt] {$3$}
    (0,-2.7) circle (0pt) node[below=-1pt] {$4$} 
    (1,-2.7) circle (0pt) node[below=-1pt] {$5$}
    (2,-2.7) circle (0pt) node[below=-1pt] {$6$}
    (4.3,-2.5) circle (0pt) node[below=-1pt] {\sc \scriptsize Simplicial degree}
    (4.3,-2.8) circle (0pt) node[below=-1pt] {\sc \scriptsize minus slice degree}
    ;}
  
  {\draw[sq2prcolor,->] 
    (-4,-2) -- (-4,-1)
    ;}
  {\draw[sq2prcolor,->] 
    (-2,0) -- (-2,1)
    ;}
  {\draw[sq2prcolor,->] 
    (0,2) -- (0,3)
    ;}
  {\draw[sq2prcolor,->] 
    (2,4) -- (2,5)
    ;}
   {\draw[sq2color,->] 
    (0,3) -- (0,4)
    ;}
  {\draw[sq2color,->] 
    (0,4) -- (0,5)
    ;}
    {\draw[partialsq2sq1color,->] 
    (-4,-1) -- (-2,0)
    ;}
  {\draw[partialsq2sq1color,->] 
    (-2,1) -- (0,2)
    ;}
  {\draw[partialsq2sq1color,->] 
    (0,3) -- (2,4)
    ;}
  {\draw[partialsq2sq1color,->] 
    (2,5) -- (2.5,5.25)
    ;}
  {\draw[tauprcolor,->] 
    (-2,0) -- (-4,1)
    ;}
  
  {\draw[sq2color,->] 
    (-4,0) -- (-4,1)
    ;}
  {\draw[sq2color,->] 
    (-4,-1) -- (-4,0)
    ;}
  {\draw[sq2color,->] 
    (-4,1) -- (-4,2)
    ;}
  {\draw[sq2color,->] 
    (-4,2) -- (-4,3)
    ;}
  {\draw[sq2color,->] 
    (-4,3) -- (-4,4)
    ;}
  {\draw[sq2color,->] 
    (-4,4) -- (-4,5)
    ;}
  {\draw[sq2color,->] 
    (-4,5) -- (-4,5.5)
    ;}
  {\draw[sq2color,->] 
    (0,5) -- (0,5.5)
    ;}
  
  {\draw[sq3sq1color,->] 
    (-4,0) -- (-2,1)
    ;}
  {\draw[sq3sq1color,->] 
    (-4,1) -- (-2,2)
    ;}
  {\draw[sq3sq1color,->] 
    (-4,2) -- (-2,3)
    ;}
  {\draw[sq3sq1color,->] 
    (-4,3) -- (-2,4)
    ;}
  {\draw[sq3sq1color,->] 
    (-4,4) -- (-2,5)
    ;}
  {\draw[sq3sq1color,->] 
    (-4,5) -- (-3,5.5)
    ;}
  {\draw[sq3sq1color,->] 
    (-2,2) -- (0,3)
    ;}
  {\draw[sq3sq1color,->] 
    (-2,3) -- (0,4)
    ;}
  {\draw[sq3sq1color,->] 
    (-2,4) -- (0,5)
    ;}
  {\draw[sq3sq1color,->] 
    (-2,5) -- (-1,5.5)
    ;}
  {\draw[sq3sq1color,->] 
    (0,4) -- (2,5)
    ;}
  {\draw[sq3sq1color,->] 
    (0,5) -- (1,5.5)
    ;}
  
  {\draw[taucolor,->] 
    (-2,1) -- (-4,2)
    ;}
  {\draw[taucolor,->] 
    (-2,2) -- (-4,3)
    ;}
  {\draw[taucolor,->] 
    (-2,3) -- (-4,4)
    ;}
  {\draw[taucolor,->] 
    (-2,4) -- (-4,5)
    ;}
  {\draw[taucolor,->] 
    (-2,5) -- (-3,5.5)
    ;}
  {\draw[tauprcolor,->] 
    (2,4) -- (0,5)
    ;}
  {\draw[taucolor,->] 
    (2,5) -- (1,5.5)
    ;}
      
  {\draw[sq2rhosq1color,->] 
    (-2,1) -- (-2,2)
    ;}
  {\draw[sq2rhosq1color,->] 
    (-2,2) -- (-2,3)
    ;}
  {\draw[sq2rhosq1color,->] 
    (-2,3) -- (-2,4)
    ;}
  {\draw[sq2rhosq1color,->] 
    (-2,4) -- (-2,5)
    ;}
  {\draw[sq2rhosq1color,->] 
    (-2,5) -- (-2,5.5)
    ;}
  {\draw[sq2rhosq1color,->] 
    (2,5) -- (2,5.5)
    ;}
  
    {\draw[sq3sq1color] 
    (3.4,4) circle (0pt) node[right] {$\Sq^3\Sq^1$}
    ;}
  {\draw[partialsq2sq1color] 
    (3.4,3) circle (0pt) node[right] {$\partial \Sq^{2}\Sq^1$}
    ;}
  {\draw[sq2prcolor] 
    (3.4,2) circle (0pt) node[right] {$\Sq^{2} \pr$}
    ;}
  {\draw[sq2rhosq1color] 
    (3.4,1) circle (0pt) node[right] {$\Sq^{2}+\rho\Sq^1$}
    ;}
  {\draw[sq2color] 
    (3.4,0) circle (0pt) node[right] {$\Sq^{2}$}
    ;}
  {\draw[tauprcolor] 
    (3.4,-1) circle (0pt) node[right] {$\tau \pr$}
    ;}
  {\draw[taucolor] 
    (3.4,-2) circle (0pt) node[right] {$\tau$}
    ;}
\end{tikzpicture}
\caption{The first slice differential for $\kq$.}
\label{fig:d1-kq}
\end{figure}

The Ph.D. thesis \cite{lopez}
provides an $\E_\infty$-ring structure on $\KQ$, which
by \cite{grso} lifts to $\kq$ for formal reasons.
The graded slices of $\kq$ then admit the multiplicative description
\begin{equation}
\label{equation:gradedhermitianktheoryslices}
\s_{\ast}(\kq) 
\iso 
\MZ[\eta,\sqrt{\alpha}]/(2\eta=0,\eta^{2}\xrightarrow{\delta}\sqrt{\alpha}). 
\end{equation}
Here $\vert \eta\vert=(1,1)$ and $\vert \sqrt{\alpha}\vert=(4,2)$.
Figure~\ref{fig:d1-kq} displays the slice $\dd^{1}$-differential for $\kq$.
This follows from \eqref{equation:slices-kq} and the analogous result for $\KQ$.

\begin{definition}
\label{defn:kw}
Let $\kw:=\kq[\tfrac{1}{\eta}]$ be shorthand for the $\eta$-inversion of $\kq$.
\end{definition}

\begin{figure}
  \pgfsetshortenend{3pt}
  \pgfsetshortenstart{3pt}
  \begin{tikzpicture}[scale=1.2,line width=1pt]
    \draw[help lines] (-4.3,-2.3) grid (3.3,5.3);
    \foreach \i in {-2,...,5} {\node[label=left:$\i$] at (-4.4,\i) {};}
    \foreach \i in {0,...,6} {\node[label=below:$\i$] at (\i-4,-2.3) {};}
    
    \node[left=-1pt,rotate=90] at (-6.1,2.4) {\sc \scriptsize Slice degree (weight)};
    
    \node[below=-1pt] at (4.3,-2.5) {\sc \scriptsize Simplicial degree};
    \node[below=-1pt] at (4.3,-2.8) {\sc \scriptsize minus slice degree};
    
    \node[color=sq3sq1color,right] at (3.4,4)  {$\Sq^3\Sq^1$};
    \node[sq2rhosq1color,right] at (3.4,2) {$\Sq^{2}+\rho\Sq^1$};
    \node[sq2color,right] at (3.4,0) {$\Sq^{2}$};
    \node[taucolor,right] at (3.4,-2) {$\tau$};
    
    \foreach \i in {-2,...,5} 
    {\foreach \j in {-4,-2,0,2} 
      { \draw[fill] (\j,\i) circle (2pt);}}

    \foreach \i in {-2,...,4} 
    {\foreach \j in {-2,2}
      {\draw[taucolor,->] 
        (\j,\i) -- (\j-2,\i+1)
        ;}}

    \foreach \i in {-2,...,4} 
    {\foreach \j in {-4,0}
      {\draw[sq2color,->] 
        (\j,\i) -- (\j,\i+1)
        ;}}

    \foreach \i in {-2,...,4} 
    {\foreach \j in {-2,2}
      {\draw[sq2rhosq1color,->] 
        (\j,\i) -- (\j,\i+1)
        ;}}

    \foreach \i in {-2,...,4} 
    {\foreach \j in {-4,-2,0}
      {\draw[sq3sq1color,->] 
        (\j,\i) -- (\j+2,\i+1)
        ;}}

    \foreach \i in {-2,...,4} 
    {\draw[sq3sq1color,->] 
        (2,\i) -- (3,\i+0.5)
        ;}
    {\draw[sq3sq1color,->] 
      (-4,5) -- (-3,5.5)
      ;}
    {\draw[sq3sq1color,->] 
      (-2,5) -- (-1,5.5)
      ;}
    {\draw[sq3sq1color,->] 
      (0,5) -- (1,5.5)
      ;}

    {\draw[sq2color,->] 
      (-4,5) -- (-4,5.5)
      ;}
    {\draw[sq2color,->] 
      (0,5) -- (0,5.5)
      ;}

    {\draw[sq2rhosq1color,->] 
      (-2,5) -- (-2,5.5)
      ;}
    {\draw[sq2rhosq1color,->] 
      (2,5) -- (2,5.5)
      ;}
    
    {\draw[taucolor,->] 
      (-2,5) -- (-3,5.5)
      ;}
    {\draw[taucolor,->] 
      (2,5) -- (1,5.5)
      ;}
  \end{tikzpicture}
\caption{The first slice differential for $\kw$.}
\label{fig:d1-kw}
\end{figure}

Figure \ref{fig:d1-kw} displays the slices together with the slice $\dd^{1}$-differentials for $\kw$.

\begin{proposition}
\label{prop:kw-defn-agree}
The canonical map $\kq \to \KQ\to\KQ[\tfrac{1}{\eta}]=\KW$ yields an identification $\kw \simeq\KW_{\geq 0}$ between $\kw$ and the connective cover of $\KW$.
\end{proposition}
\begin{proof}
There is an induced map $\kw\to \KW_{\geq 0}$ since $\kw$ is a homotopy colimit of connective motivic spectra ($\kq$ is connective by definition).
We show that it is a $\ul{\pi}_{t,w}$-isomorphism. 
The homotopy sheaves are trivial for $t<w$.
When $t\geq w$, then $\ul{\pi}_{t,w}\kw$ is isomorphic to 
\[ 
\colim_n \ul{\pi}_{t+n,w+n}\kq 
\iso 
\colim_n\ul{\pi}_{t+n,w+n}\KQ 
\iso
\ul{\pi}_{t,w}\KW \]
via the canoncial map, which concludes the proof.
\end{proof}

As a consequence, the following proposition provides that
the slice spectral sequence for $\kq$ computes
the correct data in suitable degrees.

\begin{proposition}
\label{prop:eff-kq-conv}
The canonical map $\kq \to \slicecomp(\kq)$ induces an isomorphism
\[ 
\pi_{k}\kq 
\overset{\cong}{\rightarrow} 
\pi_{k}\slicecomp (\kq) 
\]
of $\KMW$-modules for all $k\in \NN$ with $k\not\equiv 0,3\bmod 4$.
\end{proposition}
\begin{proof}
Since $\kw=\kq[\tfrac{1}{\eta}]$, 
$\eta$ acts invertibly on $\s_{\ast}(\kw)$ and the columns in the slice spectral sequences for $\kq$ and $\kw$ agree outside a finite range. 
The calculation of the $E^{2}=E^\infty$-page for $\KW$ in \cite[Theorem 6.3]{roendigs-oestvaer.hermitian} implies that $\pi_{k}\slicecomp(\kq)[\tfrac{1}{\eta}] = 0$ if $k\not\equiv 0\bmod 4$.
Corollary~\ref{cor:eta-slice-comp-kq} shows $\slicecomp(\kq)\simeq {\unit}^\wedge_\eta$.
It remains to apply the $\eta$-arithmetic square for $\kq$ and recall that $\kq[\tfrac{1}{\eta}] =\kw$ is the connective cover of $\KW$,
see Proposition~\ref{prop:kw-defn-agree}.
The result follows since $\pi_{k}\kw = 0$ for $k<0$ or $k\not\equiv 0\bmod 4$.
\end{proof}

Having clarified the range in which the slice spectral sequence for $\kq$ actually
computes $\kq$, it is now time to present these computations, starting
with the following summary.

\begin{theorem}
\label{theoremE2kq}
The nontrivial rows in the $1$st and $2$nd columns of the $E^{2}$-page of the $-n$th slice spectral sequence for $\kq$ are given as:
\begin{center}
\begin{tabular}{lll}
\hline
$q$ & $E^{2}_{-n+1,q,-n}(\kq)$                      & $E^{2}_{-n+2,q,-n}(\kq)$ \\ \hline
$3$ & $0$                                                     & $h^{n+1,n+3}/\bigl(\Sq^{2}h^{n-1,n+2}+\tau\pr^\infty_2 H^{n+1,n+2}\bigr)$ \\
$2$ & $h^{n+1,n+2}/\Sq^{2}h^{n-1,n+1}$     & $2 H^{n+2,n+2} \directsum \bigl(h^{n,n+2}/\Sq^{2}h^{n-2,n+1}\bigr)$ \\
$1$ & $h^{n,n+1}/\Sq^{2}\pr^\infty_2 H^{n-2,n}$        & $\ker (h^{n-1,n+1}\xrightarrow{\Sq^{2}}h^{n+1,n+2})$ \\
$0$ & $H^{n-1,n}$                                         & $\ker (H^{n-2,n}\xrightarrow{\Sq^{2}\pr}h^{n,n+1})$ 
\end{tabular}
\end{center}
\end{theorem}
The proof of Theorem \ref{theoremE2kq} is based on the identification of the slices of $\kq$ in \eqref{equation:slices-kq}. 
Various arguments are required to conclude $E^{2}=E^\infty$ for all of these tridegrees. First, we need to record how $\unit\to \kq$ affects slices.

\begin{lemma}
\label{lem:unit-kq-s0}
On $0$-slices, the unit map $\unit \to \kq$ induces the identity
\[ 
\s_{0}(\unit) 
= 
\MZ 
\xrightarrow{1}
\MZ 
= 
\s_{0}(\kq).
\]
\end{lemma} 
\begin{proof}
The zero slice functor preserves the ring structure;
see also \cite[Lemma 2.29]{rso.oneline}.
\end{proof}

\begin{lemma}
\label{lem:unit-kq-s1}
On $1$-slices, the unit map $\unit \to \kq$ induces the identity
\[ 
\s_{1}(\unit) 
= 
\Sigma^{1,1}\MZ/2\{\alpha_{1}\} 
\xrightarrow{1}
\Sigma^{1,1}\MZ/2 
= 
\s_{1}(\kq).
\]
\end{lemma} 
\begin{proof}
This follows from Lemma~\ref{lem:unit-kq-s0} and the multiplicative structure since the graded slice functor preserves the ring structure, 
see also \cite[Lemma 2.29]{rso.oneline}.
\end{proof}

\begin{lemma}
\label{lem:unit-kq-s2}
On $2$-slices the unit map $\unit \to \kq$ induces
\[ 
\s_{2}(\unit) 
= 
\Sigma^{2,2}\MZ/2\{\alpha_{1}^{2}\} \vee \Sigma^{3,2}\MZ/12\{\alpha_{2/2}\} 
\xrightarrow{
\begin{pmatrix} 
1 & 0 \\ 
0 & \partial^{12}_{\infty} 
\end{pmatrix}
}
\Sigma^{2,2}\MZ/2 \vee \Sigma^{4,2}\MZ 
= 
\s_{2}(\kq).
\]
\end{lemma} 
\begin{proof}
The first column follows from the multiplicative structure or \cite[Lemma 2.29]{rso.oneline}.
By \cite[Lemma 2.30]{rso.oneline} the second diagonal entry has the form $n\cdot \partial^{12}_{\infty}$ for $n$ an odd integer. 
We show that $n$ is not divisible by $3$. 
The commutative diagram  
\begin{center}
\begin{tikzcd}
\unit \ar[r] \ar[d] & \MGL \ar[d] \\
\kq \ar[r,"\forget"] & \kgl
\end{tikzcd}
\end{center}
of maps of motivic ring spectra induces a commutative diagram
\begin{center}
\begin{tikzcd}
\s_2\unit \ar[r] \ar[d] & \s_2\MGL \ar[d] \\
\s_2 \kq \ar[r,"\s_2\forget"] & \s_2\kgl
\end{tikzcd}
\end{center}
on slices. The Wood cofiber sequence from Theorem~\ref{thm:kq-wood} implies that 
$\s_2\forget$ is the canonical projection, as is
the map $\s_2\MGL\to \s_2\kgl$. Hence
$n$ can be identified from the map on $2$-slices induced by the unit map $\unit \to \MGL$. 
This computation follows from the proof of \eqref{equation:sphereslices},
and shows that $n$ is not divisible by $3$. Thus $n$ is a unit in $\ZZ/12\ZZ$.
The result follows by applying a suitable isomorphism to $\s_{2}(\unit)$.
\end{proof}

\begin{lemma}
\label{lem:unit-kq-sq}
For $q\geq 3$ 
the induced map $\s_{q}(\unit \to \kq)$ is the identity on the summands generated by $\alpha_{1}^{q}$ and $\alpha_{3}\alpha_{1}^{q-3}$,
i.e., 
there is a commutative diagram with vertical inclusions
\begin{center} 
\begin{tikzcd}
\Sigma^{q,q}\MZ/2 \vee \Sigma^{q+2,q}\MZ/2 
\ar[r,"\id"] \ar[d] & 
\Sigma^{q,q}\MZ/2 \vee \Sigma^{q+2,q}\MZ/2  \ar[d] \\
\s_{q}(\unit) \ar[r] & \s_{q}(\kq).
\end{tikzcd}
\end{center}
In particular, the unit map $\unit\to \kq$ induces the identity on $\s_{3}$.
\end{lemma} 
\begin{proof}
The argument in \cite[Lemma 2.30]{rso.oneline} for $\unit \to \KQ$ applies also to $\unit \to \kq$. 
\end{proof}

\begin{corollary}
\label{cor:unit-kq-e1-iso}
For all $n\in \ZZ$ the unit map $\unit\to \kq$ induces a surjection on $E^1_{-n+1,2,-n}$ and $E^1_{-n+2,4,-n}$, and for all $m, k\in \ZZ$ an isomorphism on:
\begin{align*}
& E^1_{-n,m,-n}        \\
& E^1_{-n+1,m,-n}   &  m\neq 2 \\
&E^1_{-n+2,m,-n}   &   m\neq 2,4 \\
&E^1_{-n+k,m,-n}   &   m\leq 1.
\end{align*}
\end{corollary}
\begin{proof}
This follows from Lemmas~\ref{lem:unit-kq-s0},~\ref{lem:unit-kq-s1},~\ref{lem:unit-kq-s2}, and~\ref{lem:unit-kq-sq}.
\end{proof}

\begin{lemma}
\label{lem:e1-kq-0}
The groups $E^{1}_{-n,m,-n}(\kq)$ and $E^{1}_{-n+1,m,-n}(\kq)$ consist of permanent cycles.
\end{lemma}
\begin{proof}
This follows from Corollary~\ref{cor:unit-kq-e1-iso} and the corresponding result for $\unit$, 
or alternatively from the computation of the $E^{2}=E^{\infty}$ page for $\kw\simeq\KW_{\geq 0}$ and the fact that $\pi_{-n,-n}\kq\cong\pi_{-n,-n}\kw\cong \Witt(F)$ is the Witt ring for $n<0$.
\end{proof}

\begin{lemma}
\label{lem:e2-kq-1>2}
The group $E^{2}_{-n+1,m,-n}(\kq)$ is trivial for $m\geq 3$.
\end{lemma}
\begin{proof}
We show the $d^{1}$-differential entering $E^1_{-n+1,m,-n}(\kq)=h^{n+m-1,n+m}$ is surjective.
For $m\geq 4$ the differential is given by 
\begin{align*}
E^1_{-n+2,m-1,-n}(\kq) 
= 
h^{n+m-3,n+m-1}\oplus h^{n+m-1,n+m-1} 
& 
\to 
h^{n+m-1,n+m} = E^1_{-n+1,m,-n}(\kq)\\
(b_{2},b_{0}) 
&
\mapsto 
\Sq^{2}b_{2}+\tau b_{0}.
\end{align*}
The claim follows since $\tau\colon h^{n+m-1,n+m-1}\to h^{n+m-1,n+m}$ is surjective.
For $m=3$, 
\begin{align*}
E^1_{-n+2,2,-n}(\kq) = h^{n,n+2}\oplus H^{n+2,n+2} 
& 
\to 
h^{n+2,n+3} = E^1_{-n+1,3,-n}(\kq)\\
(b_{2},B_{0}) 
&
\mapsto 
\Sq^{2}b_{2}+\tau \pr^\infty_2 B_{0}
\end{align*}
is surjective, since $\tau$ and $\pr^\infty_2 \colon H^{n+2,n+2}\to h^{n+2,n+2}$ are both surjective maps.
\end{proof}

\begin{lemma}
\label{lem:e2-kq-1<3}
The possibly nontrivial groups in the first column $E^{2}_{-n+1,m,-n}(\kq)$ are:
\begin{align*}
E^{2}_{-n+1,0,-n}(\kq) & \cong H^{n-1,n} \\
E^{2}_{-n+1,1,-n}(\kq) & \cong h^{n,n+1}/\Sq^{2}\pr^\infty_2 H^{n-2,n} \\
E^{2}_{-n+1,2,-n}(\kq) & \cong h^{n+1,n+2}/\Sq^{2} h^{n-1,n+1} 
\end{align*}
\end{lemma}
\begin{proof}
This follows from Lemma \ref{lem:e2-kq-1>2} and the determination of the slice $d^{1}$-differential.
\end{proof}

\begin{lemma}
\label{lem:e2-kq-2>3}
The group $E^{2}_{-n+2,m,-n}(\kq)$ is trivial for $m\geq 4$.
\end{lemma}
\begin{proof}
  The argument is analogous to the part of the proof of Lemma~\ref{lem:e2-kq-1>2}
  for $m\geq 4$.
\end{proof}

\begin{lemma}
\label{lem:e2-kq-23}
The group $E^{2}_{-n+2,3,-n}(\kq)$ is isomorphic to 
\[
h^{n+1,n+3}/\tau^{2}\rho^{2}h^{n-1,n-1}+\tau\pr^{\infty}_{2} H^{n+1,n+2}
\] 
generated by $\alpha_{1}^3$.
\end{lemma}
\begin{proof}
The kernel of the $d^{1}$-differential
\[ 
h^{n+1,n+3}\oplus h^{n+3,n+3} \to h^{n+3,n+4}, 
\quad   
(b_{2},b_{0}) 
\mapsto 
\Sq^{2}b_{2}+\tau b_{0} 
\]
is isomorphic to $h^{n+1,n+3}$ via the map $c_{2}\mapsto (c_{2},\tau^{-1}\Sq^{2}c_{2})$. 
The image of the  $d^{1}$-differential
\[  
h^{n-1,n+2}\oplus H^{n+1,n+2} 
\to 
h^{n+1,n+3}\oplus h^{n+3,n+3}, 
\quad
(a_{3},A_{1}) 
\mapsto 
(\Sq^{2}a_{3}+\tau\pr^{\infty}_{2} A_{1}, \Sq^3\Sq^1a_{3})
\]
is generated by $\tau^{2}\rho^{2}h^{n-1,n-1}$ and $\tau \pr^{\infty}_{2} H^{n+1,n+2}=\tau\ker(\partial^{2}_{\infty}\colon h^{n+1,n+2}\to H^{n+2,n+2})$.
\end{proof}

\begin{lemma}
\label{lem:e2-kq-22}
The group $E^{2}_{-n+2,2,-n}(\kq)$ is isomorphic to the direct sum
\[ h^{n,n+2}/\Sq^2h^{n-2,n+1}\directsum 2H^{n+2,n+2}.\]
\end{lemma}

\begin{proof}
  The kernel of $d^{1}$ maps to $2H^{n+2,n+2}$ via the surjection
  \begin{align*}
    \psi_n
    \colon 
    \{(b_{2},b_{0})\in h^{n,n+2}\oplus H^{n+2,n+2}  
    \colon
    \Sq^{2}b_{2}=\tau\pr^\infty_2 b_{0}\} 
    & 
      \to 
      \ker(H^{n+2,n+2}\xrightarrow{\pr^\infty_2} h^{n+2,n+2}) \\
    (b_{2},b_{0}) & \mapsto \partial^{2}_{\infty}(\tau^{-1}b_{2})\cdot \{-1\} + b_{0}.
  \end{align*}
  Here $\psi_n(0,b_{0})= b_{0}$ for all $b_{0}\in \ker(H^{n+2,n+2}\xrightarrow{\pr} h^{n+2,n+2})$, 
  and $\psi_n$ is well-defined since
  \begin{align*} 
    \pr^{\infty}_{2}(\partial^{2}_{\infty}(\tau^{-1}b_{2})\cdot \{-1\} + b_{0})
    & =
      \pr^{\infty}_{2}(\partial^{2}_{\infty}(\tau^{-1}b_{2}))\cdot \pr^{\infty}_{2}(\{-1\}) + \pr^{\infty}_{2}b_{0} \\
    & = 
      \Sq^1(\tau^{-1}b_{2})\cdot\rho + \pr^{\infty}_{2}b_{0} \\
    & = 
      \tau^{-1}\Sq^{2}b_{2} + \pr^{\infty}_{2}b_{0} \\
    & = 
      0.
  \end{align*}
  We can identify $\ker(\psi_n)$ with $h^{n,n+2}$ via the split injection $a_{2}\mapsto (a_{2},\partial^{2}_{\infty}(\tau^{-1}a_{2})\cdot \{-1\})$. 
  Hence the kernel of the exiting $d^{1}$ is isomorphic to the direct sum
  \[ 
    h^{n,n+2}\directsum \ker(H^{n+2,n+2}\xrightarrow{\pr} h^{n+2,n+2}). 
  \]
  The entering $d^{1}$ has image generated by $\Sq^{2}h^{n-2,n+1}$ in $\ker(\psi_n)$, 
  while its image in the target of $\psi_n$ is trivial since 
  \[ 
    \psi_n\circ (\Sq^{2},\partial^{2}_{\infty}\Sq^{2}\Sq^1)(h^{n-2,n+1}) 
    = 0.
  \]
  This follows from the calculation 
  \begin{align*}
    \psi_n\bigl(\Sq^{2}(\tau^3 c_{0}),\partial^{2}_{\infty} \Sq^{2}\Sq^{1}(\tau^3c_{0})\bigr)
    & = \psi_n\bigl(\tau^{2}\rho^{2} c_{0},\partial^{2}_{\infty} \tau\rho^3c_{0}\bigr) \\
    & = \partial^{2}_{\infty}(\tau\rho^{2}c_{0})\cdot \{-1\} + \partial^{2}_{\infty} \tau\rho^3c_{0} \\
    & = C_{0}\cdot \{-1\}^4 + C_{0}\cdot \{-1\}^4 \\
    & = 0,
  \end{align*}
  where $C_{0}\in H^{n-2,n-2}$ satisfies $\pr^{\infty}_{2}C_{0}=c_{0}$.
  Hence there is a split short exact sequence
  \[ 
    0
    \to 
    h^{n,n+2}/\Sq^{2}h^{n-2,n+1} 
    \to 
    E^{2}_{-n+2,2,-n}(\kq) 
    \to 
    2H^{n+2,n+2} 
    \to 
    0, 
  \]
  which concludes the proof.
\end{proof}

\begin{lemma}\label{lem:e2-kq-21}
  The group $E^{2}_{-n+2,1,-n}(\kq)$ coincides with
  the kernel of $\Sq^2\colon h^{n-1,n+1}\to h^{n+1,n+2}$.
\end{lemma}

\begin{proof}
  This follows from Corollary~\ref{cor:sq2prinfty2iszero}.
\end{proof}

\begin{lemma}\label{lem:Sq2pr42iszero}
  The composition $\Sq^2\pr^4_2\colon h^{n-3,n}_{4}\to h^{n-1,n+1}$ is zero.
\end{lemma}

\begin{proof}
  An element $x\in h^{n-3,n}_{4}$ maps via $\pr^4_2$ to an 
  element $\tau^3y$, where $y\in h^{n-3,n-3}$. Since
  $\Sq^1\pr^4_2=0$, the equality 
  \[ 0=\Sq^1(\pr^4_2(x))=\Sq^1(\tau^3y) = \tau^2\rho y \]
  follows, whence 
  $\Sq^2(\pr^4_2(x)) = \Sq^2(\tau^3y) = \tau^2\rho^2 y = 
  \rho\cdot \tau^2 \rho y = \rho\cdot 0 =0$. 
\end{proof}

\begin{corollary}\label{cor:sq2prinfty2iszero}
  The composition $\Sq^2\pr^\infty_2\colon H^{n-3,n}\to h^{n-1,n+1}$ is zero.
\end{corollary}

\begin{proof}
  This follows, because $\pr^\infty_2=\pr^4_2\pr^\infty_4$
  and already $\Sq^2\pr^4_2$ is zero by Lemma~\ref{lem:Sq2pr42iszero}.
\end{proof}

\begin{lemma}\label{lem:e2-kq-20}
  The group $E^{2}_{-n+2,0,-n}(\kq)$ coincides with
  the kernel of $\Sq^2\pr^\infty_2\colon H^{n-2,n}\to h^{n,n+1}$.
\end{lemma}

\begin{proof}
Use the description of the slice $d^1$-differential.
\end{proof}

\begin{lemma}
\label{lem:e2-kq-3>3}
The group $E^{2}_{-n+3,m,-n}(\kq)$ is trivial for $m\geq 4$.
\end{lemma}

\begin{proof}
  The argument is analogous to the part of the proof of Lemma~\ref{lem:e2-kq-1>2}
  for $m\geq 4$.
\end{proof}

\begin{lemma}
\label{lem:e2-kq-33}
The group $E^{2}_{-n+3,3,-n}(\kq)$ is isomorphic to 
\[
h^{n,n+2}/\pr^{\infty}_{2} H^{n,n+2}\iso {}_{2}H^{n+1,n+2}
\] 
generated by $\alpha_{1}^3$.
\end{lemma}
\begin{proof}
The kernel of the $d^{1}$-differential
\[ 
h^{n,n+3}\oplus h^{n+2,n+3} \to h^{n+2,n+4}, 
\quad   
(b_{3},b_{1}) 
\mapsto 
\Sq^{2}b_{3}+\tau b_{1} 
\]
is isomorphic to $h^{n,n+3}$ via the map $c_{3}\mapsto (c_{3},\tau^{-1}\Sq^{2}c_{3})$. 
The image of the  $d^{1}$-differential
\begin{align*}
 h^{n-2,n+2}\oplus  H^{n,n+2} 
&\longrightarrow 
h^{n,n+3}\oplus h^{n+2,n+3}, 
\\
(a_{4},A_{2}) 
&\mapsto 
(\Sq^{2}a_{4}+\tau\pr^{\infty}_{2} A_{2}, \Sq^3\Sq^1a_{4})=(\tau\pr^\infty_{2}A_{2},0)
\end{align*}
is generated by $\tau \pr^{\infty}_{2} H^{n,n+2}$. The result follows, where
the last isomorphism in the statement is induced by the short exact sequence
\[ 0 \to H^{n,n+2}/2H^{n,n+2}\to h^{n,n+2}\to {}_2H^{n+1,n+2}\to 0. \]
\end{proof}

\begin{lemma}
\label{lem:e2-kq-32}
The group $E^{2}_{-n+3,2,-n}(\kq)$ is isomorphic to the kernel of
\[ d^1\colon h^{n-1,n+2}\directsum H^{n+1,n+2},\, (b_3,B_1)\mapsto (\tau\pr^\infty_2B_1+\Sq^2b_3,\Sq^3\Sq^1b_3).\]
\end{lemma}

\begin{proof}
This follows from the description of the slice $d^1$-differential since the incoming differential is zero. 
\end{proof}

\begin{lemma}\label{lem:e2-kq-31}
  The group $E^{2}_{-n+3,1,-n}(\kq)$ is isomorphic to
  the kernel of $\Sq^2\colon h^{n-2,n+1}\to h^{n,n+2}$.
\end{lemma}

\begin{proof}
 This follows from the description of the slice $d^1$-differential since the incoming differential is zero. 
\end{proof}

\begin{lemma}\label{lem:e2-kq-30}
  The group $E^{2}_{-n+3,0,-n}(\kq)$ is equal to
  $H^{n-3,n}$.
\end{lemma}

\begin{proof}
  This follows from the description of the slice $d^1$-differential
  and Corollary~\ref{cor:sq2prinfty2iszero}.
\end{proof}

The above concludes our identification of the first four nontrivial columns
of the second page of the slice spectral sequence for $\kq$.
To determine the first three nontrivial columns on the third page, it remains to identify the second slice differentials in the relevant range.

\begin{lemma}\label{lem:kq-perm-cycle-1}
Every element in $E^{2}_{-n+1,m,-n}(\kq)$ is a permanent cycle.
\end{lemma}
\begin{proof}
This follows from Lemma~\ref{lem:e1-kq-0}.
\end{proof}

Lemma~\ref{lem:kq-perm-cycle-1} shows that $E^{1}_{-n,m,-n}(\kq)=E^{\infty}_{-n,m,-n}(\kq)$ for all $m,n\in\ZZ$. Hence the first potentially nonzero second slice
differential for $\kq$ is:
\begin{equation}\label{eq:d2-kq-20-12}
  d^2\colon \ker \bigl(\Sq^2\pr^\infty_2\colon H^{n-2,n}\to h^{n,n+1}\bigr)
  \to h^{n+1,n+2}/\Sq^2h^{n-1,n+1}
\end{equation}
Instead of referring to \cite[Lemma 4.15]{rso.oneline}, which implies that
the differential~(\ref{eq:d2-kq-20-12}) is zero, an alternative argument
based on \cite{rondigs.moore} will be employed. The advantage is the
identification of $\KMW$-module structures. For this purpose,
consider the slice filtration
\[ \dotsm \to \f_2\kq\to \f_1\kq \to \f_0\kq=\kq \]
and the induced filtration
\[ \dotsm \subset \f_2\pi_m\kq \subset \f_1\pi_m\kq \subset \pi_m\kq \]
on $\KMW$-modules; here $\f_n\pi_m\kq$ is the image of
$\pi_m\f_n\kq\to \pi_m\kq$. Proceed now to the first stage
of the identification of the $\KMW$-modules $\pi_1\kq$ and
$\pi_{2}\kq$. If $A$ is a graded $\KMW$-module and $n,d$ are integers, then $A(n)$ is the graded $\KMW$-module with $A(n)_d=A_{n+d}$.

\begin{lemma}\label{lem:pi-f3kq}
  There are isomorphisms $\pi_1\f_3\kq\iso \kmil(2)$ generated
  in $\pi_{1+(2)}\f_3\kq\iso \kmil_0$ and
  $\pi_2\f_3\kq\iso \kmil(1)$ generated
  in $\pi_{2+(1)}\f_3\kq\iso \kmil_0$.
\end{lemma}

\begin{proof}
  Slice convergence for $\kq$ implies slice convergence for $\f_3\kq$,
  as in Corollary~\ref{cor:eta-slice-comp-kq}.
  Similar to Lemma~\ref{lem:e2-kq-1>2} and~\ref{lem:e2-kq-3>3}, the
  columns of the second page of the slice spectral sequence for $\f_3\kq$
  corresponding to $\pi_1$ and $\pi_2$ are concentrated in the lowest
  possible slice, namely $3$. There is no room for higher differentials,
  whence $\pi_{1}\f_3\kq\iso \pi_1\s_3\kq\iso h^{\star+2,\star+3}$.
  The computation of the slice $d^1$ differential
  $\pi_2\s_3\kq\to \pi_1\s_4\kq$ provides as kernel
  $\pi_{2}\f_3\kq\iso h^{\star+1,\star+3}$.
\end{proof}

\begin{lemma}\label{lem:pi-f2kq}
  There is an isomorphism $\pi_1\f_2\kq\iso \kmil(1)$ generated
  in $\pi_{1+(1)}\f_2\kq\iso \kmil_0$.
  The map $\pi_1\f_3\kq\to \pi_1\f_2\kq$ is zero.
\end{lemma}

\begin{proof}
  Slice convergence for $\kq$ implies slice convergence for $\f_2\kq$,
  as in Corollary~\ref{cor:eta-slice-comp-kq}.
  The column of the second page of the slice spectral sequence for $\f_2\kq$
  corresponding to $\pi_1$  is concentrated in the lowest
  possible slice, namely $2$. There is no room for higher differentials,
  whence $\pi_{1}\f_2\kq\iso \pi_1\s_2\kq\iso h^{\star+1,\star+2}$.
  The long exact sequence on homotopy modules induced by the cofiber sequence
  \[ \f_3\kq \to \f_2\kq\to \s_2\kq \to \Sigma\f_3\kq \]
  provides the last statement. 
\end{proof}

The following statement is essentially \cite[Lemma 2.3]{rondigs.moore}.

\begin{lemma}\label{lem:pi-f1kq}
  There is an isomorphism $\pi_1\f_1\kq\iso \KMW/(2,\eta^2)$ generated
  by the image of the topological Hopf map $\eta_\Top$ in $\pi_{1+(0)}\f_1\kq$.
  The image of $\pi_1\f_2\kq\to \pi_1\f_1\kq$ is the submodule generated
  by $\eta\eta_\Top$, which is isomorphic to $\kmil(1)/\rho^2\kmil(-1)$.
\end{lemma}

\begin{proof}
  Slice convergence for $\kq$ implies slice convergence for $\f_1\kq$,
  as in Corollary~\ref{cor:eta-slice-comp-kq}.
  The column of the second page of the slice spectral sequence for $\f_1\kq$
  corresponding to $\pi_1$  is concentrated in slices $1$ and
  $2$. The slice $d^1$ differential
  $\Sq^2\colon \s_1\kq \to \Sigma\f_2\kq\to \Sigma\s_2\kq$
  induces $\Sq^2\colon \kmil(-1)\iso \pi_2\s_1\kq \to \pi_1\f_2\kq\iso
  \pi_1\s_2\kq\iso \kmil(1)$,
  whence the image in $\pi_1\f_2\kq$ coincides with $\rho^2\kmil(-1)$.
  There is no room for higher differentials,
  Thus the long exact sequence
  \[ \dotsm \to \pi_2\s_1\kq \to \pi_1\f_2\kq\to \pi_1\f_1\kq \to \pi_1\s_1\kq \to 0\]
  induced by the cofiber sequence
  \[ \f_2\kq\to \f_1\kq\to \s_1\kq\to \Sigma\f_2\kq \]
  induces a short exact sequence
  \begin{equation}\label{eq:ses-pi1-f1kq}
    0 \to \kmil(1)/\rho^2\kmil(-1)\to \pi_{1}\f_1\kq \to \pi_1\s_1\kq\iso \kmil\to 0
  \end{equation}
  of $\KMW$-modules. 
  The target of the extension~(\ref{eq:ses-pi1-f1kq}) is generated
  in $\pi_{1+(0)}\s_1\kq$ by the image of the topological Hopf map $\eta_\Top$,
  as (\'etale or complex) realization implies.
  The multiplicative structure of the slices implies that
  the image of $\eta\eta_\Top$ in $\pi_{1+(1)}\f_1\kq$ is nontrivial,
  hence coincides with the (unique) generator of $\kmil(1)/\rho^2\kmil(-1)$.
  In particular, the map $\KMW\to \pi_1\f_1\kq$ sending $1$ to the
  image of $\eta_\Top$ is surjective. As in the proof
  of~\cite[Lemma 2.3]{rondigs.moore}, one concludes the desired isomorphisms.
\end{proof}

Under the isomorphism given in Lemma~\ref{lem:pi-f1kq},
the short exact sequence~(\ref{eq:ses-pi1-f1kq}) corresponds
to the short exact sequence
\[ 0 \to \eta\KMW/(2,\eta^2) \to \KMW/(2,\eta^2)\to \KMW/(2,\eta)=\kmil \to 0\]
of $\KMW$-modules.
Hence $\eta\KMW/(2,\eta^2)\iso \KMW/(2,\eta,[-1]^2)\iso \kmil(1)/\rho^2\kmil(-1)$.
This furnishes an identification
\begin{equation}\label{eq:eta-torsion-pi1kq}
  {}_{\eta}\bigl(\KMW/(2,\eta^2)\bigr)\iso \kmil(1)/\rho^2\kmil(-1)\directsum \rho^2(\kmil(-2))
\end{equation}
of $\KMW$-modules. The summands are generated by
$\eta\eta_\Top$ and $[-1]^2\eta_\Top$, respectively.
This will be used in the proof of the following statement.

\begin{lemma}\label{lem:d2-pi1kq-zero}
  The slice $d^2$ differential
  $E^2_{-n+2,0,-n}(\kq)\to E^2_{-n+1,2,-n}(\kq)$ is zero.
\end{lemma}

\begin{proof}
  The cofiber sequence
  \[ \f_1\kq\to \kq\to \s_0\kq\to \Sigma\f_1\kq \]
  induces a long exact sequence of homotopy modules containing the
  portion
  \[
    \dotsm\to \pi_{2}\s_0\kq \to \pi_1\f_1\kq\to \pi_1\kq\to \pi_1\s_0\kq \to 0.\]
  Since $\s_0\kq\iso \MZ$, on which $\eta$ acts as zero,
  the homotopy module $\pi_2\s_0\kq$ is a
  $\KMil$-module. Hence the $\KMW$-module homomorphism
  $\pi_{2}\s_0\kq\to \pi_1\f_1\kq$ factors as
  \[ \pi_{2}\s_0\kq\to {}_{\eta}\bigl(\pi_1\f_1\kq\bigr) \to \pi_1\f_1\kq\]
  where the first map is a $\KMil$-homomorphism.
  The image of the $\KMW$-module  ${}_{\eta}\bigl(\pi_1\f_1\kq\bigr)\iso
  \kmil(1)/\rho^2\kmil(-1)\directsum \rho^2(\kmil(-2))$ in
  $\pi_1\s_1\kq$ coincides with $\rho^2(\kmil(-2))$ generated
  by the element $[-1]^2\eta_\Top$, because
  $\eta\eta_\Top$ maps trivially to $\pi_{1+(1)}\s_1\kq$.
  The image of the map $\pi_{2}\MZ/2 \to \pi_{1}\s_1\kq$ given by the first slice
  differential $\Sq^2\pr^\infty_2\colon \s_0\kq \to \Sigma\s_1\kq$ is
  (strictly) contained in $\rho^2(\kmil(-2))$.
  It remains to observe that the map
  $\pi_2\MZ\to \kmil(1)/\rho^2\kmil(-1)$ to the first summand
  of~(\ref{eq:eta-torsion-pi1kq}) is zero.
  Since the target is $2$-torsion, $\psi$ factors over $\bigl(\pi_2\MZ\bigr)/2$.
  Consider the short exact sequence
  \[ 0\to \bigl(\pi_{2}\MZ\bigr)/2 \to \pi_{2}\MZ/2 \to {}_{2}\pi_{1}\MZ \to 0 \]
  of $\KMW$-modules, where the first map is induced by
  $\pr^\infty_2\colon \MZ/\to \MZ/2$. This map coincides with
  $\s_0\bigl(\kq\xrightarrow{\mathrm{can.}} \kq/2\bigr)$ up to equivalence.
  Since $\pi_{1}\f_1\kq$ is $2$-torsion, the map
  $\pi_{1}\f_1\kq \to \f_1\kq/2$ is injective.
  As a consequence of the commutative diagram
  \begin{center}
    \begin{tikzcd}
      \pi_{2}\MZ\ar[r]\ar[d] & \bigl(\pi_{2}\MZ\bigr)/2 \ar[r]\ar[d] & \pi_{2}\MZ/2 \ar[d] \\
      \pi_{1}\f_1\kq \ar[r,"\id"] & \pi_{1}\f_1\kq \ar[r] & \pi_{1}\f_1\kq/2
    \end{tikzcd}
  \end{center}
  it suffices to prove that the map $\pi_{2}\MZ/2\iso \kmil(-2) \to \pi_{1}\f_1\kq/2$
  hits the summand $\kmil(1)/\rho^2\kmil(-1)$ trivially. The latter map is determined
  by the image of the (unique) $\KMW$-module generator $g=g_F\in \pi_{2+(-2)}\MZ/2$,
  which is decorated with notation for the base field, because a base change argument
  is about to happen. This generator $g_F$ is the image of the generator
  $g_{F_0}\in \pi_{2+(-2)}\MZ/2$, where $F_0\subset F$ is the prime field of $F$.
  The commutative diagram
  \begin{center}
    \begin{tikzcd}
      \pi_{2+(-2)}(\MZ/2)_{F_0}\ar[r,"\id"]\ar[d] &       \pi_{2+(-2)}(\MZ/2)_{F} \ar[d] \\
      \kmil_3(F_0)/\rho^2\kmil_1(F_0) \ar[r] &       \kmil_3(F)/\rho^2\kmil_1(F)
    \end{tikzcd}
  \end{center}
  implies that the vertical map on the right hand side is zero, because
  $\kmil_3(F_0)/\rho^2\kmil_1(F_0)$ is zero \cite[Example 1.5, Appendix]{milnor.k-quadratic}.
\end{proof}

A consequence of Lemma~\ref{lem:d2-pi1kq-zero} is that the $\KMW$-module
$\pi_1\kq$ is given by the short exact sequence
\[ 0 \to \pi_{1}\f_1\kq/\Sq^2\pr^\infty_2\pi_2\MZ \to \pi_{1}\kq \to \pi_1\MZ \to 0\]
which does not split in general, neither as an extension of $\KMW$-modules, nor
degreewise as an extension of abelian groups.

\begin{theorem}\label{thm:pi1kq-piKQ}
  The canonical map $\kq\to \KQ$ induces an isomorphism
  $\pi_{1-(w)}\kq \to \pi_{1-(w)}\KQ$ for all integers $w<4$.
  In particular, there result isomorphisms
  \[ \pi_{1-(w)}\KQ\iso \bigl(\KMW{\eta_\Top}/(2,\eta^2)\bigr)_{w}\]
  for $w<2$ (including a vanishing for $w<-1$), and short exact sequences
  \begin{align*}
    0 \to \kmil_3/\rho^2\kmil_1\directsum \kmil_2  \to &\pi_{1-(2)}\KQ \to H^{1,2} \to 0\\ 
    0 \to \kmil_4/\rho^2\kmil_2\directsum \kmil_3/\tau^{-1}\Sq^2\pr^\infty_2H^{1,3} \to &\pi_{1-(3)}\KQ \to H^{2,3} \to 0.
  \end{align*}
\end{theorem}

\begin{proof}
  The first statement holds by definition for $w$ nonpositive and follows in the
  remaining cases by inspection of the slice spectral sequence for
  $\KQ$ \cite{roendigs-oestvaer.hermitian}. The remaining assertions
  follow from Lemma~\ref{lem:d2-pi1kq-zero} and Lemma~\ref{lem:pi-f1kq}.
\end{proof}

Having determined $\pi_1\kq$, the next challenge is $\pi_{2}\kq$.

\begin{lemma}\label{lem:pi2-f2kq}
  There is a short exact sequence
  \[ 0 \to \KMil(2)/\{-1\}^3\KMil(-1) \to \pi_{2}\f_2\kq \to \kmil \to 0 \]
  of $\KMW$-modules, where $\kmil$ is generated in $\pi_{2+(0)}\f_2\kq$ by
  the image of $\eta_\Top^2$, and $\KMil$ is generated in $\pi_{2+(2)}\f_2\kq$
  by the unit in symplectic $K$-theory. The extension is uniquely
  determined by the fact that $\eta\eta_\Top^2=\{-1\}\in \KMil_1\iso \pi_{1+(1)}\f_2\kq$. The image of the map $\pi_{2}\f_3\kq\to \pi_{2}\f_2\kq$ coincides
  with ${}_{2}\KMil(2)/\{-1\}^3\KMil(-1)$.
\end{lemma}

\begin{proof}
  Lemma~\ref{lem:pi-f3kq} provides $\pi_{2}\f_3\kq\iso \kmil(1)$, a sub-$\KMW$-module
  of $\pi_{2}\s_3\kq$, generated in
  $\pi_{2+(1)}\f_3\kq$. 
  The slice $d^1$ differential $\pi_{3}\s_2\kq\to \pi_{2}\s_3\kq$
  -- see for comparison Lemma~\ref{lem:e2-kq-23} --
  and the short exact sequence
  \[ 0\to \bigl(\pi_{1}\MZ\bigr)/2 \to \pi_{1}\MZ/2 \to {}_{2}\pi_{0}\MZ \to 0 \]
  imply that the image of $\pi_{2}\f_3\kq$ in $\pi_{2}\f_2\kq$
  coincides with ${}_{2}\KMil(2)/\{-1\}^3\KMil(-1)$.
  The proof of Lemma~\ref{lem:e2-kq-22} shows that the kernel
  of the slice $d^1$ differential $\pi_{2}\s_2\kq \to \pi_{1}\s_3\kq$
  is isomorphic to $2\KMil(2)\directsum \kmil$. There results a
  short exact sequence
  \begin{equation}\label{eq:pi2-f2kq}
    0 \to {}_{2}\KMil(2)/\{-1\}^3\KMil (-1) \to \pi_{2}\f_2\kq\to 2\KMil(2)\directsum \kmil \to 0
  \end{equation}
  of $\KMW$-modules. The image of $\eta_\Top^2$ in $\pi_{2+(0)}\kq$ lifts
  to a unique element $\eta_\Top^2\in \pi_{2+(0)}\f_2 \kq$ generating $\kmil$ in the
  extension~(\ref{eq:pi2-f2kq}). The multiplicative structure of the slices,
  or rather, its behaviour with respect to the Hopf map, implies that
  $\eta\eta_\Top^2$ is the unique nonzero element $\{-1\}\in {}_{2}\KMil_1$.
  Alternatively, one may use (\'etale or complex) realization to conclude this
  fact. To specify the extension~(\ref{eq:pi2-f2kq}) with respect to the
  summand $2\KMil(2)$, observe that
  \[ \Ext_{\KMW}(2\KMil,{}_{2}\KMil/\{-1\}^3\KMil) = \Hom_{\KMW}({}_{2}\KMil,{}_{2}\KMil/\{-1\}^3\KMil) \]
  has two elements by \cite[Theorem A.1 and Lemma A.3]{rondigs.endo},
  namely the zero map and the projection. These correspond to the trivial
  and a unique nontrivial extension. The Wood cofiber sequence
  from Theorem~\ref{thm:kq-wood} supplies a long exact sequence
  \[ \dotsm \to \pi_{3+(2)}\kq \xrightarrow{\forget} \pi_{3+(2)} \kgl
    \xrightarrow{\mathrm{hyper}} \pi_{2+(1)}\kq \xrightarrow{\eta}
    \pi_{2+(2)}\kq \xrightarrow{\forget}\pi_{2+(2)}\kgl \to \dotsm\]
  in which $\pi_{3+(2)}\kq\iso \pi_{3+(2)}\KQ\iso \mathbf{K}^{\mathsf{Sp}}_1\iso 0$
  (also deducible from the slice spectral sequence), whence
  $\pi_{3+(2)}\kgl\iso \pi_{3+(2)}\KGL\iso \KMil_1\xrightarrow{\mathrm{hyper}}
  \pi_{2+(1)}\kq$ is injective. Indeed the latter map is bijective, because
  $\mathbf{K}^{\mathrm{Sp}}_0\iso \pi_{2+(2)}\kq \xrightarrow{\forget}\pi_{2+(2)}\kgl
  \iso \KMil_0$ is the inclusion of the even integers, thus injective.
  Hence the extension in question is nontrivial, which concludes the proof.
\end{proof}

\begin{lemma}\label{lem:d2-pi2kq-zero}
  The slice $d^2$ differentials
  $E^2_{-n+3,j,-n}(\kq)\to E^2_{-n+2,j+2,-n}(\kq)$ are zero for all $j$.
\end{lemma}

\begin{proof}
  The statement is true for $j>1$ by Lemma~\ref{lem:e2-kq-2>3}.
  Consider first the canonical map $\pi_{3}\s_1\kq\to \pi_{2}\f_2\kq$.
  Its  source is isomorphic to $\kmil(-2)$ generated by a unique
  element in $\pi_{3-(2)}\s_1\kq$, because $\s_1\kq\simeq \Sigma^{(1)}\MZ/2$.
  Lemma~\ref{lem:pi2-f2kq} identifies its target. However, since
  the source is a (free) $\kmil$-module, its image is contained in
  ${}_{2,\eta}\pi_{2}\f_2\kq$. Hence the map $\pi_{3}\s_1\kq\to \pi_{2}\f_2\kq$
  is determined by sending the generator $g=\tau^3\in \pi_{3-(2)}\s_1\kq$
  to an element in ${}_{2,\eta}\pi_{2-(2)}\f_2\kq$. The generator $g$ is already
  defined over the prime field $F_0\subset F$ of $F$, whence its image has
  to lie in (the image of) ${}_{2,\eta}\pi_{2-(2)}\f_2\kq_{F_0}$.
  Lemma~\ref{lem:pi2-f2kq} provides in particular  that the kernel of the canonical
  map $\pi_{2}\f_2\kq\to \pi_{2}\s_2\kq$ is ${}_{2}\KMil(2)/\{-1\}^3\KMil(-2)$.
  In the specific degree $2-(2)$, this kernel is ${}_{2}\KMil_4/\{-1\}^3\KMil_3$,
  which is the zero group for prime fields, as \cite{milnor.k-quadratic} implies. Hence the map
  $\pi_{3}\s_1\kq\to \pi_{2}\f_2\kq$ is determined by the slice $d^1$ differential
  $\pi_{3}\s_1\kq \to \pi_{2}\s_2\kq$. In other words, the slice $d^2$
  differential  $E^2_{-n+3,1,-n}(\kq)\to E^2_{-n+2,3,-n}(\kq)$ is zero. 
  The statement for the slice $d^2$ differential
  $E^2_{-n+3,0,-n}(\kq)\to E^2_{-n+2,2,-n}(\kq)$ follows from
  Corollary~\ref{cor:unit-higher-diff-from-30}, the corresponding statement
  for the slice $d^2$ differential
  $E^2_{-n+3,0,-n}(\unit)\to E^2_{-n+2,2,-n}(\unit)$,
  because of the commutative diagram
  \begin{center}
    \begin{tikzcd}
      E^2_{-n+3,0,-n}(\unit)\ar[r,"\id"] \ar[d] &E^2_{-n+3,0,-n}(\kq) \ar[d]\\
      E^2_{-n+2,2,-n}(\unit) \ar[r] & E^2_{-n+2,2,-n}(\kq)
    \end{tikzcd}
  \end{center}
  in which the top horizontal map is the identity.
\end{proof}

\begin{lemma}\label{lem:d3-pi2kq-zero}
  The slice $d^3$ differentials
  $E^3_{-n+3,j,-n}(\kq)\to E^3_{-n+2,j+3,-n}(\kq)$ are zero for all $j$.
\end{lemma}

\begin{proof}
  The statement is true for $j>0$ by Lemma~\ref{lem:e2-kq-2>3}.
  The statement for the slice $d^3$ differential
  $E^3_{-n+3,0,-n}(\kq)\to E^3_{-n+2,3,-n}(\kq)$ follows from
  Corollary~\ref{cor:unit-higher-diff-from-30}, 
  the corresponding statement for the slice $d^3$ differential
  $E^3_{-n+3,0,-n}(\unit)\to E^3_{-n+2,3,-n}(\unit)$,
  because of the commutative diagram
  \begin{center}
    \begin{tikzcd}
      E^3_{-n+3,0,-n}(\unit)\ar[r,"\id"] \ar[d] &E^3_{-n+3,0,-n}(\kq) \ar[d]\\
      E^3_{-n+2,3,-n}(\unit) \ar[r] & E^3_{-n+2,3,-n}(\kq)
    \end{tikzcd}
  \end{center}
  in which the top horizontal map is the identity.
\end{proof}

\begin{lemma}\label{lem:pi2-f1kq}
  The image of $\pi_{3}\s_1\kq \to \pi_{2}\f_2\kq$ coincides with
  $\rho^2\kmil$. The kernel of $\pi_{2}\s_1\kq\to \pi_{1}\f_2\kq$
  coincides with ${}_{\rho^2}\kmil$. There results a short exact sequence
  \[ 0 \to \pi_{2}\f_2\kq/\rho^2\kmil \to \pi_{2}\f_1\kq \to {}_{\rho^2}\kmil(-1) \to 0 \]
  of $\KMW$-modules.
\end{lemma}

\begin{proof}
  This follows from the identification of the slice $d^1$ differential, 
  Lemma~\ref{lem:d2-pi2kq-zero} and Lemma~\ref{lem:d3-pi2kq-zero}.
\end{proof}

\begin{lemma}\label{lem:pi2-kq}
  The map $\pi_3\s_0\kq \to \pi_{2}\f_1\kq$ is the zero map.
  Hence there is a short exact sequence 
  \[ 0 \to \pi_{2}\f_1\kq \to \pi_{2}\kq \to \ker\bigl(\Sq^2\pr^\infty_2\colon \pi_{2}\MZ\to \pi_{1}\Sigma^{(1)}\MZ/2\bigr) \to 0 \]
  of $\KMW$-modules, where $\eta$ operates on
  $\ker\bigl(\Sq^2\pr^\infty_2\colon \pi_{2}\MZ\to \pi_{1}\Sigma^{(1)}\MZ/2\bigr)$
  via the projection $\pr^\infty_2$ to ${}_{\rho^2}\kmil$.
\end{lemma}

\begin{proof}
  The first statement follows from Lemma~\ref{lem:pi3MZmapstrivially}.
  The second statement is then a consequence of the slice $d^1$ differential
  on $\pi_{2}\s_0\kq$, 
  Lemma~\ref{lem:d2-pi2kq-zero} and Lemma~\ref{lem:d3-pi2kq-zero}.
  The last statement follows from the multiplicative
  structure on the slices of $\kq$.
\end{proof}

\begin{theorem}\label{thm:pi2kq-piKQ}
  The canonical map $\kq\to \KQ$ induces an isomorphism
  $\pi_{2-(w)}\kq \to \pi_{2-(w)}\KQ$ for all integers $w<4$.
  In particular, there results a vanishing $\pi_{2-(w)}\KQ\iso 0$ for $w<-2$,
  isomorphisms $\pi_{2+(2)}\KQ\iso \KMil_0$, $\pi_{2+(1)}\KQ\iso \KMil_1$,
  $\pi_{2+(0)}\KQ\iso \KMil_2\directsum \kmil_0$ of abelian groups,
  and short exact sequences
  \begin{align*}
    0 \to \KMil_3/\{-1\}^3\KMil_0\directsum \kmil_1  \to &\pi_{2-(1)}\KQ \to {}_{\rho^2}\kmil_0 \to 0\\ 
    0 \to \KMil_4/\{-1\}^3\KMil_1\directsum \kmil_2/\rho^2\kmil_0  \to &\pi_{2-(2)}\KQ \to {}_{\rho^2}\kmil_1 \to 0\\ 
    0 \to \KMil_5/\{-1\}^3\KMil_2\directsum \kmil_3/\rho^2\kmil_1 \to &\pi_{2-(3)}\KQ \to {}_{\rho^2}\kmil_2 \rtimes \ker\bigl(\Sq^2\pr^\infty_2\colon H^{1,3}\to h^{3,4}\bigr) \to 0
  \end{align*}
  of abelian groups.
\end{theorem}

\begin{proof}
  The first statement holds by definition for $w$ nonpositive and follows in the
  remaining cases by inspection of the slice spectral sequence for
  $\KQ$ \cite{roendigs-oestvaer.hermitian}. The remaining assertions
  follow from Lemma~\ref{lem:pi2-f2kq}, Lemma~\ref{lem:pi2-f1kq} and
  Lemma~\ref{lem:pi2-kq}. 
\end{proof}

\begin{remark}\label{rem:KSP4}
Theorem~\ref{thm:pi2kq-piKQ} contains in particular a "symbolic"
description of the fourth symplectic $K$-group $\mathbf{K}^{\mathrm{Sp}}_4(F)\iso \pi_{2-(2)}\KQ_F$ 
of a field of characteristic not $2$, in the
sense that only symbols in Milnor $K$-theory are involved. It relies
on the vanishing $H^{0,2}=0$, which can be deduced from
\cite[Theorem 6.18]{MR2811603} and 
\cite[Equation (4.2)]{merkurjev.weighttwo}.
\end{remark}

\section{The slice filtration on the sphere spectrum} 
\label{sec:slice-spectr-sequ}

The group $\Ext^{1,2i}_{\MU_{\ast}\MU}(\MU_{\ast},\MU_{\ast})$ in the slices of $\unit_{\Lambda}$, 
see \eqref{equation:sphereslices}, 
is finite cyclic of order $a_{2i}$;
let $\overline{\alpha}_{i}$ denote the generator of its $p$-primary component $\Ext^{1,2i}_{\BP_{\ast}\BP}(\BP_{\ast},\BP_{\ast})$ \cite{ravenel.green}.
Still $\tau$ denotes the unique nonzero element in  $h^{0,1}\cong\mu_{2}$ and $\rho$ denote the class of $-1$ in $h^{1,1}$.
We denote by $\partial^{a_{2q}}_{2}$ the unique nontrivial map from $\M \Lambda/a_{2q}$ to $\Sigma^{1,0}\M \Lambda/2$ in $\SH$.
For reference we recall \cite[Theorem 4.1]{rso.oneline}.
\begin{lemma}
\label{lem:first-diff-unit-1}
In slice degrees $0\leq q\leq 3$, 
$\dd^{\unit_{\Lambda}}_{1}(q)\colon \s_{q}(\unit_{\Lambda})\to \Sigma^{1,0}\s_{q+1}(\unit_{\Lambda})$ is given by:
\begin{align*}
\dd^{\unit_{\Lambda}}_{1}(0)
& 
= \Sq^{2} \mathrm{pr}\colon \M \Lambda\to \M \Lambda/2\to \Sigma^{2,1}  \M \Lambda/2 \\
\dd^{\unit_{\Lambda}}_{1}(1)
& 
= (\Sq^{2}, \inc_{2} \Sq^{2}\Sq^1)\colon \Sigma^{1,1}  \M \Lambda/2\to \Sigma^{3,2}  \M \Lambda/2 \vee \Sigma^{4,2}  \M \Lambda/12 \\
\dd^{\unit_{\Lambda}}_{1}(2) 
& 
= 
\begin{pmatrix} 
\Sq^{2} &   \tau \partial^{12}_{2} \\ 
\Sq^3\Sq^1   &  \Sq^{2} \partial^{12}_{2}
\end{pmatrix}
\colon 
\Sigma^{2,2}  \M \Lambda/2 \vee \Sigma^{3,2}  \M \Lambda/12 
\to
\Sigma^{4,3}  \M \Lambda/2 \vee \Sigma^{6,3}  \M \Lambda/2 \\
\dd^{\unit_{\Lambda}}_{1}(3)
& 
=
\begin{pmatrix} 
\Sq^{2} & \tau  \\ 
\Sq^3\Sq^1 & \Sq^{2}+\rho\Sq^1  
\end{pmatrix}
\colon
\Sigma^{3,3}\M \Lambda/2\vee\Sigma^{5,3}\M \Lambda/2
\to 
\Sigma^{5,4}  \M \Lambda/2 
\vee  
\Sigma^{7,4} \M \Lambda/2 
\end{align*}
For $q\geq 4$,
$\dd^{\unit_{\Lambda}}_{1}(q)$ restricts to the direct summand $\Sigma^{q,q}\M \Lambda/2\vee\Sigma^{q+2,q}\M \Lambda/2$ of $\s_{q}(\unit_{\Lambda})$ by 
\begin{align*}
\begin{pmatrix} 
\Sq^{2} & \tau  \\ 
\Sq^3\Sq^1 & \Sq^{2}+\rho\Sq^1  
\end{pmatrix}
\colon
\Sigma^{q,q}\M \Lambda/2\vee\Sigma^{q+2,q}\M \Lambda/2\to\Sigma^{q+2,q+1}  \M \Lambda/2 \vee  \Sigma^{q+4,q+1} \M \Lambda/2.
\end{align*}
Here $\Sigma^{q+2,q}\M \Lambda/2$ is generated by $\alpha_{1}^{q-3}\alpha_{3}\in\Ext^{q-2,2q}_{\BP_{\ast}\BP}(\BP_{\ast},\BP_{\ast})$.

Moreover, 
$\dd^{\unit_{\Lambda}}_{1}(q)$ restricts as follows on direct summands of $\s_{q}(\unit_{\Lambda})$:
\begin{align*}
\inc_{2q}\Sq^{2}\Sq^1
& 
\colon \Sigma^{4q-3,2q-1}   \M \Lambda/2
\to 
\Sigma^{4q,2q}  \M \Lambda/{a_{2q}} \\
\Sq^{2} \partial^{a_{2q}}_{2} 
& 
\colon \Sigma^{4q-1,2q}  \M \Lambda/{a_{2q}} 
\to 
\Sigma^{4q+2,2q+1}   \M \Lambda/2\\
\tau \partial^{a_{2q}}_{2} 
& 
\colon \Sigma^{4q-1,2q}  \M \Lambda/{a_{2q}} 
\to 
\Sigma^{4q,2q+1} \M \Lambda/2 
& 
q \mathrm{\ odd} \\
0 
& 
\colon \Sigma^{4q-1,2q}  \M \Lambda/{a_{2q}} 
\to 
\Sigma^{4q,2q+1} \M \Lambda/2 
& 
q \mathrm{\ even} 
\end{align*}
Here $\Sigma^{4q-3,2q-1}\M \Lambda/2$ and $\Sigma^{4q-1,2q}\M \Lambda/{a_{2q}}$ are generated by $\alpha_{2q-1}$ and $\alpha_{q/n}$, 
respectively.
\end{lemma}

Figure~\ref{fig:d1-sphere} summarizes these calculations by pairing the 
slices of $\unit_{\Lambda}$ with their corresponding generators. 
Each direct summand of a fixed slice is labeled by the difference between the simplicial suspension degree and the slice degree.
The colors of the $\dd^{1}$-differentials correspond to elements of the motivic Steenrod algebra, 
here ordered by simplicial degree.
An open square refers to integral (or rather $\Lambda$-) coefficients and solid dots to $\Lambda/2$-coefficients.
The open circle in the second slice indicates $\Lambda/12$-coefficients, 
and similarly for $\Lambda/240$-, $\Lambda/6$-, and $\Lambda/504$-coefficients.
Lemma~\ref{lem:e2-unit-24} discusses the $\dd^{1}$-differentials exiting and entering the direct summand $\Sigma^{6,4}\M\Lambda/2\{\beta_{2/2}\}$ corresponding to the generator $\beta_{2/2}=\alpha_{2}^{2}$.

\begin{figure}
\begin{center}
  \pgfsetshortenend{3pt}
  \pgfsetshortenstart{3pt}
  \begin{tikzpicture}[scale=1.2,line width=1pt]
    \draw[help lines] (-.3,-.3) grid (7.3,7.3);

    \node[rectangle, draw] at (0,0) {};
    \draw[fill] (0,1) circle (2pt) node[below left]{$\alpha_{1}$};
    \foreach \i in {2,...,7} {\draw[fill] (0,\i) circle (2pt) node[below left]{$\alpha_{1}^{\i}$};}
    
    \draw[] (1,2) circle (4pt) node {\tiny{12}} node[below right] {$\alpha_{2}$};

    \draw[fill] (2,3) circle (2pt) node[below left]{$\alpha_{3}$};
    \draw[fill] (2,4) circle (2pt) node[below left]{$\alpha_{1}\alpha_{3}$};
    \foreach \i in {2,3,4} {\draw[fill] (2,\i+3) circle (2pt) node[below left]{$\alpha_{1}^{\i}\alpha_{3}$};}
    \draw[fill] (2.2,4) circle (2pt) node[below right] {$\alpha_{2}^{2}$};

    \draw[] (3,4) circle (5pt) node {\tiny{240}} node[below right] {$\alpha_{4}$};
    \draw[fill] (3,5) circle (2pt) node[below left] {$\alpha_{1}\alpha_{4}$};
    \foreach \i in {2,3} {\draw[fill] (3,\i+4) circle (2pt) node[below left] {$\alpha_{1}^{\i}\alpha_{4}$};}
    \draw[fill] (3.2,5) circle (2pt) node[below right] {$\beta_{2}$};
    \draw[fill] (3.2,6) circle (2pt) node[below right] {$\alpha_{2}^{3}$};

    \draw[fill] (4,5) circle (2pt) node[below right] {$\alpha_{5}$};
    \draw[] (4,6) circle (3pt) node {\tiny{6}} node[below right] {$\alpha_{1}\alpha_{5}$};
    \draw[fill] (4,7) circle (2pt) node[below right] {$\alpha_{1}^{2}\alpha_{5}$};
    \draw[] (5,6) circle (5pt) node {\tiny{504}} node[below right] {$\alpha_{6}$};
    \draw[fill] (5,7) circle (2pt) node[below right] {$\alpha_{1}\alpha_{6}$};
    \draw[fill] (6,7) circle (2pt) node[below right] {$\alpha_{7}$};

    \foreach \i in {0,...,6} {\node[below] at (\i,-0.4) {\i};}
    \foreach \i in {0,...,7} {\node[left] at (-0.4,\i) {\i};}

    {\draw (8.3,-0.2) circle (0pt) node[below=-1pt] {\sc \scriptsize Simplicial degree}
    (8.3,-0.6) circle (0pt) node[below=-1pt] {\sc \scriptsize minus slice degree}
    ;}

  {\draw[incsq2color,->] 
    (4,5) -- (4,6)
    ;}
  
  {\draw[incsq2sq1color,->] 
    (0,1) -- (1,2)
    ;}
  {\draw[incsq2sq1color,->] 
    (2,3) -- (2.95,3.95)
    ;}
  {\draw[incsq2sq1color,->] 
    (4,5) -- (4.95,5.95)
    ;}
  {\draw[incsq2sq1color,->] 
    (6,7) -- (6.5,7.5)
    ;}

  {\draw[sq2partialcolor,->] 
    (1,2.05) -- (2,3.05)
    ;}
  {\draw[sq2partialcolor,->] 
    (3.05,4.05) -- (4,5)
    ;}
  {\draw[sq2partialcolor,->] 
    (5,6.05) -- (6,7.05)
    ;}

  {\draw[taupartialcolor,->] 
    (1,2) -- (0,3)
    ;}

  \foreach \i in {2,...,6} {\draw[sq3sq1color,->] (0,\i) -- (2,\i+1);}
  \foreach \i in {4,...,6} {\draw[sq3sq1color,->] (2,\i) -- (4,\i+1);}
  {\draw[sq3sq1color,->] 
    (4,6) -- (6,7)
    ;}
  {\draw[sq3sq1color,->] 
    (0,7) -- (1,7.5)
    ;}
  {\draw[sq3sq1color,->] 
    (2,7) -- (3,7.5)
    ;}
  {\draw[sq3sq1color,->] 
    (4,7) -- (5,7.5)
    ;}

  \foreach \i in {3,...,6} {\draw[taucolor,->] (2,\i) -- (0,\i+1);}
  {\draw[taucolor,->] 
    (2,7) -- (1,7.5)
    ;}
  {\draw[tauprcolor,->] 
    (5,6) -- (3,7)
    ;}
  {\draw[taucolor,->] 
    (5,7) -- (4,7.5)
    ;}
  {\draw[taucolor,->] 
    (6,7) -- (5,7.5)
    ;}

  \foreach \i in {1,...,6} {\draw[sq2color,->] (0,\i) -- (0,\i+1);}
  \foreach \i in {5,6} {\draw[sq2color,->] (4,\i) -- (4,\i+1);}
  {\draw[sq2prcolor,->] 
    (0,0) -- (0,1)
    ;}
  {\draw[sq2prcolor,->] 
    (4,6) -- (4,7)
    ;}
  {\draw[sq2color,->] 
    (0,7) -- (-0,7.5)
    ;}
  {\draw[sq2color,->] 
    (4,7) -- (4,7.5)
    ;}

  \foreach \i in {3,...,6} {\draw[sq2rhosq1color,->] (2,\i) -- (2,\i+1);}
  {\draw[sq2rhosq1color,->] 
    (2,7) -- (2,7.5)
    ;}
  {\draw[sq2rhosq1color,->] 
    (6,7) -- (6,7.5)
    ;}

  {\draw[taupartialcolor,->] 
    (5,6) -- (4,7)
    ;}

  {\draw[sq3sq1color] 
    (7.4,7.4) circle (0pt) node[right] {$\Sq^3\Sq^1$}
    ;}
  {\draw[incsq2sq1color,->] 
    (7.4,6.6) circle (0pt) node[right] {$\inc^2_{?} \Sq^2\Sq^1$}
    ;}
  {\draw[sq2partialcolor,->] 
    (7.4,5.8) circle (0pt) node[right] {$\Sq^2 \partial^{?}_{2}$}
    ;}
  {\draw[incsq2color,->] 
    (7.4,5) circle (0pt) node[right] {$\inc^{2}_{6} \Sq^2$}
    ;}
  {\draw[sq2prcolor] 
    (7.4,4.2) circle (0pt) node[right] {$\Sq^2 \pr^{?}_{2}$}
    ;}
  {\draw[sq2rhosq1color] 
    (7.4,3.4) circle (0pt) node[right] {$\Sq^2+\rho\Sq^1$}
    ;}
  {\draw[sq2color] 
    (7.4,2.6) circle (0pt) node[right] {$\Sq^2$}
    ;}
  {\draw[taupartialcolor] 
    (7.4,1.8) circle (0pt) node[right] {$\tau \partial^{?}_{2}$}
    ;}
  {\draw[tauprcolor] 
    (7.4,1) circle (0pt) node[right] {$\tau \pr$}
    ;}
  {\draw[taucolor] 
    (7.4,0.2) circle (0pt) node[right] {$\tau$}
    ;} 
\end{tikzpicture}
\end{center}
\caption{The first slice differential for $\unit_{\Lambda}$.}
\label{fig:d1-sphere}
\end{figure}

Fix a compatible pair $(F,\Lambda=\mathbb{Z}[\tfrac{1}{c}])$, 
where $F$ is a field of exponential characteristic $c\neq 2$.
For legibility we leave $\Lambda$ out of the notation, 
and write $h^{\ast,\ast}_{2,2}$ for motivic cohomology groups with $(\Lambda/2\times\Lambda/2)$-coefficients.
Figure \ref{fig:e1-sphere} displays the $E^{1}$-page of the weight $w=-n$th slice spectral sequence, 
$n\in\ZZ$.
Here $E^{1}_{p,q,-n}(\unit)=0$ for $p<-n$ or $q<0$, and $H^{p,n}=0$ for $n<0$.

\begin{figure}
  \pgfsetshortenend{3pt}
  \pgfsetshortenstart{1pt}
  \begin{tikzpicture}[scale=2.3,line width=1pt]
    \draw[help lines,shift={(-.3,-.1)}] (-2.3,0) grid (1.95,7.9);
    \foreach \i in {0,...,7} {\node[label=left:$\i$] at (-2.4,\i+.3) {};}
    {\node[label=below:$-n$] at (-2,-0.3) {};}
    {\node[label=below:$-n+1$] at (-1,-0.3) {};}
    {\node[label=below:$-n+2$] at (0,-0.3) {};}
    {\node[label=below:$-n+3$] at (1,-0.3) {};}

    {\draw[fill]     
      (1.5,-.35) circle (0pt) node[below] {$s$}
      ;}
    
    {\draw[fill]     
      (-2.7,4.4) circle (0pt) node[left] {$q$}
      ;}
    
    \draw[->,sq2prcolor]
    (0,0) -- (-1,1);
    \draw[->,red]
    (0,1) -- (-1,2);
    \draw[->,red]
    (0,2) -- (-1,3);
    \draw[sq2color,->]
    (0,3) -- (-1,4);
    \draw[sq2color,->]
    (0,4) -- (-1,5);
    \draw[sq2color,->]
    (0,5) -- (-1,6);
    \draw[sq2color,->]
    (0,6) -- (-1,7);
    \draw[sq2color,->]
    (0,7) -- (-1,8);

    \draw[->,sq2prcolor]
    (1,0) -- (0,1);
    \draw[->,red]
    (1,1) -- (0,2);
    \draw[->,red]
    (1,2) -- (0,3);
    \draw[sq2color,->]
    (1,3) -- (0,4);
    \draw[sq2color,->]
    (1,4) -- (0,5);
    \draw[sq2color,->]
    (1,5) -- (0,6);
    \draw[sq2color,->]
    (1,6) -- (0,7);
    \draw[sq2color,->]
    (1,7) -- (0,8);

    \draw[sq2rhosq1color,->]
    (1,3.5) -- (0,4.5);
    \draw[sq2rhosq1color,->]
    (1,4.5) -- (0,5.5);
    \draw[sq2rhosq1color,->]
    (1,5.5) -- (0,6.5);
    \draw[sq2rhosq1color,->]
    (1,6.5) -- (0,7.5);
    \draw[sq2rhosq1color,->]
    (1,7.5) -- (0.5,8);

    \draw[sq3sq1color,->]
    (1,2) -- (0,3.5);
    \draw[sq3sq1color,->]
    (1,3) -- (0,4.5);
    \draw[sq3sq1color,->]
    (1,4) -- (0,5.5);
    \draw[sq3sq1color,->]
    (1,5) -- (0,6.5);
    \draw[sq3sq1color,->]
    (1,6) -- (0,7.5);
    \draw[sq3sq1color,->]
    (1,7) -- (0.5,8);
    
    \draw[taucolor,->]
    (0,3.5) -- (-1,4);
    \draw[taucolor,->]
    (0,4.5) -- (-1,5);
    \draw[taucolor,->]
    (0,5.5) -- (-1,6);
    \draw[taucolor,->]
    (0,6.5) -- (-1,7);
    \draw[taucolor,->]
    (0,7.5) -- (-1,8);

    \draw[taucolor,->]
    (1,3.5) -- (0,4);
    \draw[taucolor,->]
    (1,4.5) -- (0,5);
    \draw[taucolor,->]
    (1,5.5) -- (0,6);
    \draw[taucolor,->]
    (1,6.5) -- (0,7);
    \draw[taucolor,->]
    (1,7.5) -- (0,8);

    \draw[taupartialcolor,->]
    (0,2.5) -- (-1,3);
    \draw[taupartialcolor,->]
    (1,2.5) -- (0,3);
   
    \draw[sq2partialcolor,->]
    (1,2.5) -- (0,3.5);
 
    \draw[incsq2sq1color,->]
    (1,1) -- (0,2.5);
    
    \node at (-2,0) [shape=rectangle,draw] {};
    \node at (-2,0) [above right=3pt] {$H^{n,n}$};
    \node at (-1,0) [shape=rectangle,draw] {};
    \node at (-1,0) [above right=3pt] {$H^{n-1,n}$};
    \node at (0,0) [shape=rectangle,draw] {};
    \node at (0,0) [above right=3pt] {$H^{n-2,n}$};
    \node at (1,0) [shape=rectangle,draw] {};
    \node at (1,0) [above right=3pt] {$H^{n-3,n}$};
    
    \node[star, star points=12, draw, fill] at (-1,2.5)  {};
    \node at (-1,2.5) [right=5pt] {$h_{12}^{n+2,n+2}$};
    \node[star, star points=12, draw, fill] at (0,2.5)  {};
    \node at (0,2.5) [right=5pt] {$h_{12}^{n+1,n+2}$};
    \node[star, star points=12, draw, fill] at (1,2.5)  {};
    \node at (1,2.5) [right=5pt] {$h_{12}^{n,n+2}$};

    \node[star, star points=15, draw, fill] at (1,4.75)  {};
    \node at (1,4.75) [right=5pt] {$h_{240}^{n+4,n+4}$};

    {\draw[fill]     
      (-2,1) circle (1pt) node[above right=3pt] {{$h^{n+1,n+1}$}}
      (-2,2) circle (1pt) node[above right=3pt] {{$h^{n+2,n+2}$}}
      (-2,3) circle (1pt) node[above right=3pt] {{$h^{n+3,n+3}$}}
      (-2,4) circle (1pt) node[above right=3pt] {{$h^{n+4,n+4}$}}
      (-2,5) circle (1pt) node[above right=3pt] {{$h^{n+5,n+5}$}}
      (-2,6) circle (1pt) node[above right=3pt] {{$h^{n+6,n+6}$}}
      (-2,7) circle (1pt) node[above right=3pt] {{$h^{n+7,n+7}$}}
      ;}

    {\draw[fill]     
      (-1,1) circle (1pt) node[above right=3pt] {{$h^{n,n+1}$}}
      (-1,2) circle (1pt) node[above right=3pt] {{$h^{n+1,n+2}$}}
      (-1,3) circle (1pt) node[above right=3pt] {{$h^{n+2,n+3}$}}
      (-1,4) circle (1pt) node[above right=3pt] {{$h^{n+3,n+4}$}}
      (-1,5) circle (1pt) node[above right=3pt] {{$h^{n+4,n+5}$}}
      (-1,6) circle (1pt) node[above right=3pt] {{$h^{n+5,n+6}$}}
      (-1,7) circle (1pt) node[above right=3pt] {{$h^{n+6,n+7}$}}
      ;}
    
    {\draw[fill]      
      (0,1) circle (1pt) node[above right=3pt] {$h^{n-1,n+1}$}
      (0,2) circle (1pt) node[above right=3pt] {$h^{n,n+2}$}
      (0,3) circle (1pt) node[above right=3pt] {$h^{n+1,n+3}$}
      (0,3.5) circle (1pt) node[right=3pt] {$h^{n+3,n+3}$}
      (0,4) circle (1pt) node[above right=3pt] {$h^{n+2,n+4}$}
      (0,4.5) circle (1pt) node[right=3pt] {$h_{2,2}^{n+4,n+4}$}
      (0,5) circle (1pt) node[above right=3pt] {$h^{n+3,n+5}$}
      (0,5.5) circle (1pt) node[right=3pt] {$h^{n+5,n+5}$}
      (0,6) circle (1pt) node[above right=3pt] {$h^{n+4,n+6}$}
      (0,6.5) circle (1pt) node[right=3pt] {$h^{n+6,n+6}$}
      (0,7) circle (1pt) node[above right=3pt] {$h^{n+5,n+7}$}
      (0,7.5) circle (1pt) node[right=3pt] {$h^{n+7,n+7}$}
      ;}

    {\draw[fill]      
      (1,1) circle (1pt) node[above right=3pt] {$h^{n-2,n+1}$}
      (1,2) circle (1pt) node[above right=3pt] {$h^{n-1,n+2}$}
      (1,3) circle (1pt) node[above right=3pt] {$h^{n,n+3}$}
      (1,3.5) circle (1pt) node[right=3pt] {$h^{n+2,n+3}$}
      (1,4) circle (1pt) node[above right=3pt] {$h^{n+1,n+4}$}
      (1,4.5) circle (1pt) node[right=3pt] {$h_{2,2}^{n+3,n+4}$}
      (1,5) circle (1pt) node[above right=3pt] {$h^{n+2,n+5}$}
      (1,5.5) circle (1pt) node[right=3pt] {$h^{n+4,n+5}$}
      (1,5.75) circle (1pt) node[right=3pt] {$h_{2,2}^{n+5,n+5}$}
      (1,6) circle (1pt) node[above right=3pt] {$h^{n+3,n+6}$}
      (1,6.5) circle (1pt) node[right=3pt] {$h^{n+5,n+6}$}
      (1,6.75) circle (1pt) node[right=3pt] {$h_{2,2}^{n+6,n+6}$}
      (1,7) circle (1pt) node[above right=3pt] {$h^{n+4,n+7}$}
      (1,7.5) circle (1pt) node[right=3pt] {$h^{n+6,n+7}$}
      (1,7.75) circle (1pt) node[right=3pt] {$h^{n+7,n+7}$}
      ;}
    {\draw (1,4.5) circle (2pt);} 
    {\draw (0,4.5) circle (2pt);} 
    {\draw (1,5.75) circle (2pt);} 
    {\draw (1,6.75) circle (2pt);} 

     {\draw[sq3sq1color] 
        (2.1,5.6) circle (0pt) node[right] {$\Sq^3\Sq^1$}
        ;}
      {\draw[incsq2sq1color,->] 
        (2.1,4.8) circle (0pt) node[right] {$\inc \Sq^{2}\Sq^1$}
        ;}
      {\draw[sq2partialcolor] 
        (2.1,4) circle (0pt) node[right] {$\Sq^{2} \partial$}
        ;}
      {\draw[sq2prcolor] 
        (2.1,3.2) circle (0pt) node[right] {$\Sq^{2} \pr$}
        ;}
      {\draw[sq2rhosq1color] 
        (2.1,2.4) circle (0pt) node[right] {$\Sq^{2}+\rho\Sq^1$}
        ;}
      {\draw[sq2color] 
        (2.1,1.6) circle (0pt) node[right] {$\Sq^{2}$}
        ;}
      {\draw[taupartialcolor,->] 
        (2.1,0.8) circle (0pt) node[right] {$\tau \partial$}
        ;}
      {\draw[taucolor] 
        (2.1,0) circle (0pt) node[right] {$\tau$}
        ;}
\end{tikzpicture}
\caption{The $E^{1}$-page for $\unit$.}
\label{fig:e1-sphere}
\end{figure}

Next we show vanishing in a certain range of the terms contributing to the $2$-line.
\begin{lemma}
\label{lem:e2-unit-2>4}
The group $E^{2}_{-n+2,m,-n}(\unit)$ is trivial for $m\geq 5$.
\end{lemma}
\begin{proof}
For $m\geq 5$ the kernel of the differential
\begin{align*}
E^1_{-n+2,m,-n}(\unit) = h^{n+m-2,n+m}\oplus h^{n+m,n+m} 
& \to 
h^{n+m,n+m+1} = E^1_{-n+1,m+1,-n}(\unit)\\
(b_{2},b_{0}) 
&
\mapsto 
\Sq^{2}b_{2}+\tau b_{0}
\end{align*}
is isomorphic to $h^{n+m-2,n+m}$ via the map sending $c_{2}$ to $(c_{2},\tau^{-1}\Sq^{2}c_{2})$.
The image of the differential
\[ 
E^1_{-n+3,m-1,-n}(\unit) 
\to
E^1_{-n+2,m,-n}(\unit) = h^{n+m-2,n+m}\oplus h^{n+m,n+m} 
\]
hits this group, 
since it restricts to multiplication with $\tau$ on the summand $h^{n+m-2,n+m-1}$ generated by $\alpha_{1}^{m-1}$.
\end{proof}

\begin{lemma}
\label{lem:e2-unit-24}
The group $E^{2}_{-n+2,4,-n}(\unit)$ is isomorphic to $h^{n+4,n+4}$ generated by $\beta_{2/2}$.
\end{lemma}
\begin{proof}
The $d^{1}$-differential on $h^{n+4,n+4}\{\beta_{2/2}\}$ is not $\tau$ by comparison with complex points; 
the corresponding $d^{3}$-differential in the Adams-Novikov spectral sequence on $\beta_{2/2}$ vanishes. 
Hence the $\dd^{1}$-differential restricted to $\Sigma^{6,4}\MZ/2\{\beta_{2/2}\}$ maps trivially to $\Sigma^{6,5}\MZ/\{\alpha_{1}^{5}\}$. 
This implies that on $\s_{4}$ the unit map $\unit \to \KQ$ restricts trivially on $\Sigma^{6,4}\MZ/2\{\beta_{2/2}\}$.
In turn, the $d^{1}$-differential restricted to $\Sigma^{6,4}\MZ/2\{\beta_{2/2}\}$ maps trivially to $\Sigma^{8,5}\MZ/\{\alpha_{1}^{2}\alpha_{3}\}$ as well.
Since there is no other possible $d^{1}$-differential on $h^{n+4,n+4}\{\beta_{2/2}\}$ for weight reasons, 
this group survives to the $E^{2}$-page. 

The $d^{1}$-differentials on the other two direct summands are the same as in the proof of Lemma~\ref{lem:e2-unit-2>4};
whence the kernel of 
\[ 
E^1_{-n+2,4,-n}(\unit) 
\to 
E^1_{-n+1,5,-n}(\unit) 
\]
is isomorphic to $h^{n+2,n+4}\directsum h^{n+4,n+4}$.
The entering $d^{1}$-differential has the form 
\begin{align*}
h^{n+2,n+3}\oplus h^{n,n+3} & = E^1_{-n+3,3,-n}(\unit) 
\to
h^{n+4,n+4}\{\beta_{2/2}\}\oplus h^{n+4,n+4}\oplus h^{n+2,n+4} \\ 
(a_{1},a_{3})& 
\mapsto 
\bigl(\phi\Sq^1(a_{1})+b\Sq^3\Sq^1(a_{3}),\rho\Sq^1(a_{1})+\Sq^3\Sq^1(a_{3}),\tau(a_{1})+\Sq^{2}(a_{3})\bigr), 
\end{align*}
for some $\phi\in h^{1,1}$ and $b\in h^{0,0}$. 
By naturality, 
$\phi\in h^{1,1}(F_{0})$ and $b\in h^{0,0}(F_{0})$, 
where $F_{0}$ is the prime field of $F$. 
The homomorphism $a_{1}=\tau a_{0}\mapsto \phi\Sq^1(\tau a_{0})=\phi\rho \tau a_{0}$ thus amounts to multiplication by $\phi\rho\in h^{2,2}(F_{0})$ on $h^{n+2,n+3}(F)$. 
If $c>2$, it is thus trivial,
since $\phi\rho\in h^{2,2}(F_{0})=0$ for $F_{0}$ a finite field. 
If $c=1$, 
the square class $\phi$ can be determined via comparison with $\ZZ[\tfrac{1}{2}]$;
the slice computation for $2$-primary torsion summands holds for that base scheme, 
as well as the identification of the motivic Steenrod algebra at $2$. 
Hence we can conclude $\phi\in \{\pm 1, \pm{2}\}$. 
Considering the real numbers rules out the cases $\phi<0$. 
To this end, 
observe that real realization sends $\nu$ to $\eta_{\mathrm{top}}$ and $\rho$ to $1$.
Hence $[-1]^{-n}\nu^{2}\in\pi_{n+6,n+4}\unit$ is a nontrivial element for every subfield of $\mathbb{R}$. 
Setting $n=-2$ implies that $\phi=1$ over $\mathbb{R}$. 
To exclude the case $\phi=2$ observe that the symbol $\{-1,2\}$ is the zero element in $K^{\M}_{2}$. 
In particular, 
the homomorphism given by multiplication by $\phi\rho$ is the zero homomorphism.

Similarly, 
but easier, 
the homomorphism $a_{3}=\tau^3a_{0}\mapsto b\Sq^3\Sq^1(a_{3})=b\rho^4a_{0}$ amounts to multiplication by $b\rho^4$. 
It is thus trivial over fields of odd characteristic.
To conclude for fields of characteristic zero we consider the real numbers.
Recall that $[-1]^{-n}\nu^{2}\in\pi_{n+6,n+4}\unit$ is a nontrivial element for every subfield of $\mathbb{R}$. 
Setting $n=-4$ implies $b=0$ over $\mathbb{R}$, and hence over any field of characteristic zero. 
Comparison with $\ZZ[\tfrac{1}{2}]$ implies $b=0$ over any field with $c\neq 2$.
Hence the entering $d^{1}$-differential has the form 
\begin{align*}
h^{n+2,n+3}\oplus h^{n,n+3} & = E^1_{-n+3,3,-n}(\unit) 
\to
h^{n+4,n+4}\{\beta_{2/2}\}\oplus h^{n+4,n+4}\oplus h^{n+2,n+4} \\ 
(a_{1},a_{3})& \mapsto \bigl(0,\rho\Sq^1(a_{1})+\Sq^3\Sq^1(a_{3}),\tau(a_{1})+\Sq^{2}(a_{3})\bigr). 
\end{align*}
We may identify this map's image with the direct summand $h^{n+2,n+4}$.
\end{proof}

\begin{corollary}
\label{cor:first-diff-to-nusquared}
The $\dd^{1}$-differential $\s_{3}(\unit) \to \Sigma^{1,0}\s_{4}(\unit)$ maps trivially to the direct summand $\Sigma^{6,4}\MZ/2$ generated by $\beta_{2/2}$.
\end{corollary}
\begin{proof}
The proof of Lemma~\ref{lem:e2-unit-24} shows the map $\Sigma^{3,3}\MZ/2 \to \Sigma^{7,4}\MZ/2$ (given by $b\Sq^3\Sq^1$) is zero, 
and that $\Sigma^{5,3}\MZ/2 \to \Sigma^{7,4}\MZ/2$ (given by $a\Sq^{2}+\phi\Sq^1$, where $a\in h^{0,0}$, $\phi\in h^{1,1}$) 
is of the form $a\Sq^{2}+\phi\Sq^1$, where $a\in h^{0,0}$, $\phi\in \{1,2\}$. 
The composite of $d^{1}$-differentials   
\[ 
\s_{2}(\unit)
\to 
\Sigma^{1,0}\s_{3}(\unit) 
\to 
\Sigma^{2,0}\s_{4}(\unit)
\]
is trivial.  
Since $a\Sq^{5}\Sq^{1} = \phi\Sq^1\Sq^{2}\partial^{12}_{2} = 0$,
it follows that $a=0$ and $\phi=1$.
\end{proof}

\begin{lemma}
\label{lem:e2-unit-23}
The group $E^{2}_{-n+2,3,-n}(\unit)$ is isomorphic to 
\[
h^{n+1,n+3}/\tau^{2}\rho^{2}h^{n-1,n-1}
\] 
generated by $\alpha_{1}^{3}$.
\end{lemma}

\begin{proof}
As before, the kernel of the $d^{1}$-differential
\[ 
h^{n+1,n+3}\oplus h^{n+3,n+3} 
\to 
h^{n+3,n+4}, 
\quad   
(b_{2},b_{0}) 
\mapsto 
\Sq^{2}b_{2}+\tau b_{0} 
\]
is isomorphic to $h^{n+1,n+3}$ via the map $c_{2}\mapsto (c_{2},\tau^{-1}\Sq^{2}c_{2})$. 
The image of the $d^{1}$-differential
\[  
h^{n-1,n+2}\oplus h^{n,n+2}_{12} 
\to 
h^{n+1,n+3}\oplus h^{n+3,n+3}, 
\,
(a_{3},a_{2}) 
\mapsto 
(\Sq^{2}a_{3}+\tau\partial^{12}_{2} a_{2}, \Sq^3\Sq^1a_{3}+\Sq^{2}\partial^{12}_{2}a_{2})
\]
is the subgroup $\Sq^{2}h^{n-1,n+2}$ by 
Corollary~\ref{cor:partial42zerooneven}.
It can be identified with $\tau^{2}\rho^{2}h^{n-1,n-1}$.
\end{proof}

\begin{lemma}
\label{lem:e2-unit-22}
The group $E^{2}_{-n+2,2,-n}(\unit)$ is isomorphic to the direct sum 
\[
\bigl(h^{n,n+2}/\Sq^{2}h^{n-2,n+1}\bigr)
\directsum 
\im(h^{n+1,n+2}_{24}
\to 
h^{n+1,n+2}_{12})
\] 
generated by $\alpha_{1}^{2}$ and $\alpha_{2/2}$, respectively.
\end{lemma}
\begin{proof}
If $\rho^{2}=0$, the kernel of the $d^{1}$-differential
\[
h^{n,n+2}\oplus h^{n+1,n+2}_{12} 
\to 
h^{n+2,n+3},
\quad 
(b_{2},b_{1}) 
\mapsto 
\Sq^{2}b_{2}+\tau \partial^{12}_{2} b_{1}
\]
is isomorphic to $h^{n,n+2}\directsum \ker(h^{n+1,n+2}_{12}\xrightarrow{\partial^{12}_{2}} h^{n+2,n+2})=h^{n,n+2}\directsum \im(h^{n+1,n+2}_{24}\to h^{n+1,n+2}_{12})$,
and the image of 
\[ 
h^{n-2,n+1}
\to 
h^{n,n+2}\oplus h^{n+1,n+2}_{12}, 
\quad 
a_{3} 
\mapsto  
(\Sq^{2}a_{3}, \inc\Sq^{2}\Sq^1a_{3})\]
is trivial. 
In general, 
consider the map
\begin{align*}
\phi_n\colon \{(b_{2},b_{1})\in h^{n,n+2}\oplus h^{n+1,n+2}_{12}  
\colon 
\Sq^{2}b_{2}=\tau\partial^{12}_{2} b_{1}\} & 
\to 
\ker(h^{n+1,n+2}_{12}\xrightarrow{\partial^{12}_{2}} h^{n+2,n+2}) \\
(b_{2},b_{1}) 
& 
\mapsto 
\partial^{2}_{12}(\tau^{-1}b_{2})\cdot \tau_{12}+ b_{1}, 
\end{align*}
where $\tau_{12}$ denotes $-1\in F$ viewed as an element in $h^{0,1}_{12}$ (i.e., the kernel of $H^{1,1}\xrightarrow{12} H^{1,1}$). 
Observe that $\partial^{12}_{\infty}(\tau_{12})=-1$ and $\inc^{2}_{12}(\tau) = \tau_{12}$. 
Moreover,
$\partial^{2}_{12}\colon h^{n,n+1}\to h^{n+2,n+2}_{12}$ is the connecting map for the short exact coefficient sequence
\[ 
0\to \ZZ/12\to \ZZ/24 \to \ZZ/2 \to 0 
\]
and $\cdot$ denotes the product in $h^{\ast,\ast}_{12}$.
Note that $\phi_n$ is well-defined because 
\begin{align*}
\partial^{12}_{2}\bigl(\partial^{2}_{12}(\tau^{-1}b_{2})\cdot \tau_{12}+ b_{1}\bigr) 
& 
= 
\partial^{12}_{2}\bigl(\partial^{2}_{12}(\tau^{-1}b_{2})\cdot \tau_{12}\bigr)+ \partial^{12}_{2} b_{1} \\
& 
= 
\rho^{2}\tau^{-1}b_{2}+\partial^{12}_{2}b_{1} 
= 
0.
\end{align*}
It is surjective since  $\phi(0,c_{1})=c_{1}$ for every $c_{1}\in \ker(\partial^{12}_{2}\colon h^{n+1,n+2}_{12}\rightarrow h^{n+2,n+2})$.
The kernel of $\phi_n$ is precisely $h^{n,n+2}$, 
mapping via the split injection $a_{2}\mapsto (a_{2},\partial^{2}_{12}(\tau^{-1}a_{2})\cdot \tau_{12})$. 
The image of the $d^{1}$-differential in $\ker(\phi_n)$ coincides with $\Sq^{2} (h^{n-2,n+1})$, 
which can be identified with $\tau^{2}\rho^{2}h^{n-2,n-2}$.
The image of the $d^{1}$-differential in the target of $\phi_n$ is trivial since
\begin{align*}
\phi_n\bigl(\Sq^{2}(\tau^3\pr^{\infty}_{2}C_{0}),\inc^{2}_{12}\Sq^{2}\Sq^1(\tau^3\pr^{\infty}_{2}C_{0})\bigr) 
& 
=
\phi_n \bigl(\tau^{2}\rho^{2} \pr^{\infty}_{2}C_{0}, \inc^{2}_{12}(\tau \rho^3 \pr^{\infty}_{2}C_{0})\bigr) \\
& 
= 
\partial^{2}_{12}(\tau \rho^{2} \pr^{\infty}_{2}C_{0})\cdot \tau_{12} + \inc^{2}_{12}(\tau \rho^3 \pr^{\infty}_{2}C_{0}) \\
& 
= 
\rho_{12}^3 \pr^{\infty}_{12} C_{0} \cdot \tau_{12} +\rho_{12}^3\pr^{\infty}_{12}C_{0}\cdot \tau_{12} = 0.
\end{align*}
Here $\rho_{12} = \pr^{\infty}_{12}(\{-1\})$, 
which implies that $\pr^{12}_{2}(\rho_{12})=\rho$.
From the above we obtain a split short exact sequence
\[ 
0
\to 
h^{n,n+2}/\Sq^{2}(h^{n-2,n+1}) 
\to 
E^{2}_{-n+2,2,-n}(\unit) 
\to 
\im(h^{n+1,n+2}_{24}\to h^{n+1,n+2}_{12})
\to 
0, 
\]
which completes the argument.
\end{proof}

\begin{lemma}
\label{lem:e2-unit-21}
The group $E^{2}_{-n+2,1,-n}(\unit)$ is isomorphic to 
\[ 
\{a\in h^{n-1,n+1} \colon \Sq^{2} a=0\} \iso {}_{\rho^2}h^{n-1,n-1}
\] 
generated by $\alpha_{1}$.
\end{lemma}

\begin{proof}
  This follows from the description of the $\dd^{1}$-differential
  and Lemma~\ref{lem:Sq2pr42iszero} below, since
  $\Sq^2\pr^\infty_2=\Sq^2\pr^4_2\pr^\infty_4$.
\end{proof}

\begin{lemma}
\label{lem:e2-unit-20}
The group $E^{2}_{-n+2,0,-n}(\unit)$ is isomorphic to 
\[ 
\ker \bigl( H^{n-2,n} \xrightarrow{\Sq^{2}\pr} h^{n,n+1}\bigr) 
\] 
generated by $\alpha_{1}^{0}=1$.
\end{lemma}

\begin{proof}
  This follows directly from the description of the $\dd^{1}$-differential.  
\end{proof}

We summarize our calculations by displaying the groups contributing to the $1$- and $2$-line:
\begin{theorem}
\label{theoremE2sphere}
The nontrivial rows in the $1$st and $2$nd columns of the $E^{2}$-page of the $-n$th slice spectral sequence for $\unit$ are given as:
\begin{center}
\begin{tabular}{lll}
\hline
$q$ & $E^{2}_{-n+1,q,-n}(\unit)$    & $E^{2}_{-n+2,q,-n}(\unit)$ \\ \hline
$4$ & $0$ & $h^{n+4,n+4}$ \\
$3$ & $h^{n+2,n+3}/\tau\partial^{12}_{2} h^{n+1,n+2}_{12}$ \quad\quad  & $h^{n+1,n+3}/\Sq^{2}h^{n-1,n+2}$ \\
$2$ & $h^{n+2,n+2}_{12} \directsum $    & $\ker(h_{12}^{n+1,n+2}\xrightarrow{\partial^{12}_{2}} h^{n+2,n+2}) \directsum $ \\
    & $h^{n+1,n+2}/\Sq^{2}h^{n-1,n+1})$   & $\quad\quad h^{n,n+2}/\Sq^{2}h^{n-2,n+1}$ \\
$1$ & $h^{n,n+1}/\Sq^{2}\pr H^{n-2,n}$    & $\ker (h^{n-1,n+1}\xrightarrow{\Sq^{2}}h^{n+1,n+2})$ \\
$0$ & $H^{n-1,n}$     & $\ker (H^{n-2,n}\xrightarrow{\Sq^{2}\pr}h^{n,n+1})$ 
\end{tabular}
\end{center}
\end{theorem}

\begin{proof}
  For the first column, see the table in the proof of \cite[Theorem 5.5]{rso.oneline}. Note that $E^{2}_{-n+1,3,-n}(\unit)$ coincides with the group
  $h^{n+2,n+3}\!/\tau\partial^{12}_{2}h^{n+1,n+2}_{12}$, because of the 
  inclusion $\Sq^2(h^{n,n+2})\subset \tau\partial^{12}_{2}(h^{n+1,n+2}_{12})$.
  In fact, the first subgroup can be identified with
  $\tau\rho^2 h^{n,n}$. Since $\partial^{12}_{2}\circ \inc^{2}_{12}=\Sq^1$,
  the second subgroup contains $\tau\Sq^1(h^{n+1,n+2})=\tau\rho h^{n+1,n+1}$.
  The second column follows from Lemmas~\ref{lem:e2-unit-2>4},~\ref{lem:e2-unit-24},~\ref{lem:e2-unit-23},~\ref{lem:e2-unit-22},~\ref{lem:e2-unit-21}, and~\ref{lem:e2-unit-20}.
\end{proof}

We claim the higher slice differentials entering the second column vanish.
Due to Lemma~\ref{lem:e2-unit-2>4}, it suffices to identify three terms on the second page of $\unit$'s slice spectral sequence. 

\begin{lemma}
\label{lem:e2-unit-32}
The group $E^{2}_{-n+3,2,-n}(\unit)$ coincides with the direct sum 
\[ h^{n,n+2}_{12}\directsum \ker\bigl(\Sq^2\colon h^{n-1,n+2}\to h^{n+1,n+3}\bigr).\]
\end{lemma}

\begin{proof}
  The incoming differential $\Sq^2\colon h^{n-3,n+1}\to 
  h^{n,n+2}_{12}\directsum h^{n-1,n+2}$ is zero. The connecting
  map $\partial^{12}_{2}\colon h^{n,n+2}_{12}\to h^{n+1,n+2}$ is zero as,
  well, by Corollary~\ref{cor:partial42zerooneven}, which implies the statement.
\end{proof}

\begin{lemma}
  \label{lem:e2-unit-31}
  The group $E^{2}_{-n+3,1,-n}(\unit)$ coincides with 
  $\ker\bigl(\Sq^2\colon h^{n-2,n+1}\to h^{n,n+2}\bigr)$.
\end{lemma}

\begin{proof}
  The incoming differential $\Sq^2\pr^\infty_2\colon H^{n-4,n}\to 
  h^{n-2,n+1}$ is zero.   The statement then follows, because the kernel of
  $\Sq^2\colon h^{n-2,n+1}\to h^{n,n+2}$ is contained in the kernel of
  $\Sq^2\Sq^1\colon h^{n-2,n+1}\to h^{n+1,n+2}$.
\end{proof}

\begin{lemma}
  \label{lem:e2-unit-30}
  The group $E^{2}_{-n+3,0,-n}(\unit)$ coincides with 
  $H^{n-3,n}$.
\end{lemma}

\begin{proof}
  This follows from Corollary~\ref{cor:sq2prinfty2iszero}.
\end{proof}

\begin{lemma}\label{lem:unit-2nd-diff-to-24}
The second differential
\[ 
E^2_{-n+3,2,-n} (\unit)
\to 
E^2_{-n+2,4,-n} (\unit)
\]
in the $-n$th slice spectral sequence for $\unit$ is trivial.
\end{lemma}
\begin{proof}
  The target is computed as 
  $E^2_{-n+2,4,-n}(\unit)\iso h^{n+4,n+4}\{\alpha_{2/2}^2\}$
  in Lemma~\ref{lem:e2-unit-24}. 
  For $n=-4$, the element $1\cdot \{\alpha_{2/2}^2\}\in h^{0,0}$
  detects $\nu^2\in \pi_{6,4}\unit$. If $F$ is a subfield of the
  real numbers, the complex realization of $\nu^2$ is $\nu_{\Top}^2\neq 0$,
  and the real realization of $\nu^2$ is $\eta_{\Top}^2\neq 0$.
  Lemma \ref{lem:e2-unit-32}
  determines the source of this differential as
  \begin{align*}
    E^{2}_{-n+3,2,-n}(\unit) & = \{(a_2,a_3)\in h^{n,n+2}_{12}
                               \oplus h^{n-1,n+2} \colon \tau\partial^{12}_{2} (a_{2}) = \Sq^{2}(a_3)\} \\
    & = h^{n,n+2}_{12}
      \oplus \{a_3\in h^{n-1,n+2} \colon \Sq^{2}(a_3)=0\}
  \end{align*}
  because of Corollary~\ref{cor:partial42zerooneven}.
  Since the target of this differential
  is 2-torsion, it will be trivial on the direct summand
  $h^{n,n+2}_{3}$ of $E^{2}_{-n+3,2,-n}(\unit)$.
  In order to prove that it is trivial on the
  direct summand 
  $h^{n,n+2}_{4}$ of $E^{2}_{-n+3,2,-n}(\unit)$, it suffices
  to prove that a generator $g\in h^{0,2}_{4}$ provided
  by Lemma~\ref{lem:generators-motcohommod4} is mapped 
  to zero. 
  As remarked already in Lemma~\ref{lem:generators-motcohommod4},
  a generator of $h^{0,2}_{4}$ already exists
  over the prime field. For prime fields of odd characteristic, the
  target of the differential is the trivial group $h^{4,4}$.
  For the prime field $\mathbb{Q}$, the unique nontrivial element
  in $h^{4,4}$ indexed by $\beta_{2,2}$
  is $\rho^4$. Since the real realization of $\rho^4\nu^{2}$ is the
  nontrivial element
  $1^4\eta_{\mathrm{top}}^{2}$, the element $\rho^4$ cannot be hit
  by any differential.

  It remains to prove that the differential is zero on the
  direct summand 
  \[ \{a_3\in h^{n-1,n+2} \colon \Sq^{2}(a_3)=0\} \subset E^{2}_{-n+3,2,-n}(\unit).\]
  Considering the direct sum 
  over all $n$, one obtains a graded module over
  Milnor $K$-theory by Lemma~\ref{lem:differential-module-hom},
  which 
  is isomorphic to 
  $\{a\in h^{n-1,n-1}\colon \rho^{2}a=0\}$. Hence it is
  generated by elements 
  in degree at most $1$ (corresponding to $1\leq n\leq 2$)
  by \cite[Theorem 3.3]{ovv}. Consider these cases.
  \begin{description}
  \item[$n=2$] In this case, a generator $(0,\overline{a})$ is represented
    by a (square class of a) unit $a\in F$ such that
    $\rho^{2}\overline{a} =0\in h^{3,3}(F)$. By \cite[Theorem 3.2]{ovv}
    such an element is in the image of the sum of finitely many
    transfer maps for quadratic field extensions $F\subset L_i$,
    each of them satisfying $\rho^{2}=0\in h^{2,2}(L_i)$.
    By naturality of slice differentials with respect to transfer
    maps, it suffices to consider such fields, bringing us to the case
  \item[$n=1$] Here a generator can only have the form
    $(0,\tau^3)$, which is possible if and only if $\rho^{2}=0$
    in $h^{2,2}(F)$.
    On the other hand $\rho^{2}=0\in h^{2,2}(F)$ 
    if and only if there exist
    two elements $a,b\in F$ such that $-1=a^{2}+b^{2}$ 
    \cite[Corollary 3.5]{elman-lam.pfister}. Then $(0,\tau^3)$
    is the image of $(0,\tau^3)$ over the field $F_{0}(a,b)$,
    where $F_{0}\subset F$ is the prime field of $F$. This field
    has 2-cohomological dimension at most 3. In that case, the target
    of the second differential under consideration, $h^{5,5}$, vanishes.
  \end{description}
  This case by case analysis shows that also the last
  direct summand of $E^{2}_{-n+3,2,-n}(\unit)$
  maps trivially under this differential, implying that the
  whole differential is zero.
\end{proof}

Similar to the treatment of very effective hermitian $K$-theory $\kq$,
the homotopy groups of the successive slice filtrations $\f_q\unit$
will be discussed as $\KMW$-modules, starting with $\pi_1\f_4\unit$
whose image in $\pi_1\unit$ will be zero. The treatment for $\pi_1\unit$ given here can also be found in \cite{rondigs.moore}.

\begin{lemma}\label{lem:pi1f4sphere}
  There is an isomorphism
  of $\KMW$-modules $\pi_{1}\f_4\unit\iso \kmil(3)$
  generated by $\eta^3\eta_\Top \in \pi_{1+(3)}\f_4\unit$.
\end{lemma}

\begin{proof}
  Since $\eta_\Top\in \pi_{1+(0)}\unit$ lifts uniquely to
  $\eta_\Top\in \pi_{1+(0)}\f_1\unit$, the multiplicativity
  of the slice filtration (see \cite{grso}) provides that
  $\eta^3\eta_\Top\in \pi_{1+(3)}\f_4\unit$. The multiplicative
  structure on the slices (with respect to multiplication
  with $\alpha_1$ only) shows that the image
  of $\eta^3\eta_\Top$ in $\pi_{1+(3)}\s_4\unit\iso h^{0,1}\iso \kmil_0$
  is the unique nonzero element. Since the map
  $\pi_{1}\s_4\unit \to \pi_{0}\f_5\unit$ is zero, and the map
  $\pi_{2}\s_4\unit \to \pi_{1}\f_5\unit$ is surjective, there
  results an identification
  $\pi_{1}\f_4\unit\iso \pi_{1}\s_4\unit\iso \kmil(3)$.
\end{proof}

\begin{lemma}\label{lem:pi1f3sphere}
  There is an isomorphism of $\KMW$-modules $\pi_{1}\f_3\unit\iso \kmil(2)$
  generated by $\eta^2\eta_\Top \in \pi_{1+(2)}\f_3\unit$.
  The image of the canonical map $\pi_{1}\f_4\unit\to \pi_{1}\f_3\unit$
  is zero.
\end{lemma}

\begin{proof}
A proof follows by inspection of the slice spectral sequence for $\f_3\unit$, similar to the argument in Lemma~\ref{lem:pi1f4sphere}.
\end{proof}

\begin{lemma}\label{lem:pi1f2sphere}
  There is a short exact sequence
  \[ 0\to \KMil(2)/24 \to \pi_{1}\f_2\unit \to \kmil(1) \to 0 \]
  of $\KMW$-modules
  generated by $\nu \in \pi_{1+(2)}\f_2\unit$ and
  $\eta\eta_\Top \in \pi_{1+(1)}\f_2\unit$, respectively. 
  It is uniquely determined by the fact that $\eta^2\eta_\Top=12\nu$.
  The image of the canonical map $\pi_{1}\f_3\unit\to \pi_{1}\f_2\unit$
  coincides with $\kmil(2)$ generated by $12\nu\in \pi_{1+(2)}\unit$.
\end{lemma}

\begin{proof}
The crucial part is to identify the extension as a $\KMW$-module.
As \cite[Lemma A.3]{rondigs.endo} implies, 
a unique nontrivial extension with the property that $\eta^2\eta_\Top=12\nu$ exists. 
See, for example, the proof of 
\cite[Theorem 5.5]{rso.oneline} for this identification.
\end{proof}

For the discussion of higher slice differentials, considering
the second slice filtration suffices. Therefore the presentation
\begin{equation}\label{eq:pi1f1unit}
  \frac{\KMW(2)\{\nu\} \directsum \KMW\{\eta_\Top\} }{(2\eta_\Top,\eta\nu,
    \eta^2\eta_\Top-12\nu)} \iso \pi_{1}\f_1\unit
\end{equation}
given in \cite[Theorem 2.5]{rondigs.moore} will not
be directly relevant. Instead the path to
$\pi_2\unit$ will be followed, starting with $\pi_2\f_4\unit$, as
higher filtrations do not contribute -- see Lemma~\ref{lem:e2-unit-2>4}.

\begin{lemma}\label{lem:pi2f4sphere}
  There is an 
  isomorphism of $\KMW$-modules 
  $\pi_{2}\f_4\unit\iso \kmil(4)\directsum \kmil(2)$
  generated by $\nu^2\in \pi_{2+(4)}\f_4\unit$ and
  $\eta^2\eta^2_\Top\in \pi_{2+(2)}\f_4\unit$, respectively.
\end{lemma}
\begin{proof}
  Since $\nu\in \pi_{1+(2)}\unit$ lifts (uniquely)
  to $\pi_{1+(2)}\f_2\unit$, the multiplicativity of the
  slice filtration provides that $\nu^2\in \pi_{2+(4)}\f_4\unit$.
  Lemma~\ref{lem:nusquared-twotorsion} shows that $2\nu^2=0$.
  The multiplicative
  structure on the slices implies that $\nu^2$ maps to the
  unique nonzero element in the group $E^2_{6,4,4}(\unit)$
  computed in Lemma~\ref{lem:e2-unit-24}. Similarly, the element
  $\eta^2\eta_\Top^2\in \pi_{2+(2)}\f_4\unit$ hits the element
  $(\tau^2,\rho^2)\in h^{0,2}\directsum h^{2,2}$ of the kernel
  of the slice $d^1$ differential. There results an
  identification as stated. 
\end{proof}

\begin{lemma}\label{lem:pi2f3sphere}
  There is an isomorphism of $\KMW$-modules 
  $\pi_{2}\f_3\unit\iso \kmil(4)\directsum \kmil(1)$
  generated by $\nu^2\in \pi_{2+(4)}\f_3\unit$ and
  $\eta\eta^2_\Top\in \pi_{2+(2)}\f_3\unit$, respectively.
  The image of the canonical map $\pi_{2}\f_4\unit\to \pi_{2}\f_3\unit$
  coincides with $\kmil(4)$ generated by $\nu^2\in \pi_{2+(4)}\unit$.
\end{lemma}

\begin{proof}
  The slice filtration provides a short exact sequence
  \[ 0 \to \kmil(4)\to \pi_{2}\f_3\unit \to \kmil(1)\to 0\]
  where $\kmil(4)$ is the cokernel of $\pi_{3}\s_3\unit\to \pi_{2}\f_4\unit$.
  The element $\eta\eta_\Top^2\in \pi_{2+(2)}\f_3\unit$ lifts
  the uniqe generator of $\kmil(1)$.
  The equation $\eta^2\eta_\Top=12\nu$ implies that
  $\eta^2\eta_\Top^2=12\nu\eta_\Top=0$, since already
  $2\eta_\Top=0$. It follows that $\eta$ operates as zero
  on $\pi_{2}\f_3\unit$, whence it is in fact a $\KMil$-module.
  The equation $2\eta\eta^2_\Top=0$ implies that
  $\pi_{2}\f_3\unit$ splits as described.
\end{proof}

\begin{lemma}\label{lem:pi2f2sphere}
  There is a short exact sequence
  \[ 0 \to \kmil(4)\directsum \bigl(\pi_{1}\Sigma^{(2)}\MZ/24\bigr)/\inc^{2}_{24}\rho^2\tau\kmil \to
  \pi_{2}\f_2\unit \to \kmil \to 0 \]
  of $\KMW$-modules, where $\kmil$ is generated by
  $\eta_\Top^2\in \pi_{2+(0)}\f_2\unit$.
  This extension is uniquely determined by the fact that
  $\eta \eta_\Top^2 = -1 \in h^{0,1}_{24} \iso\pi_{1+(1)}\Sigma^{(2)}\MZ/24$.
  The image of the canonical map $\pi_{2}\f_3\unit\to \pi_{2}\f_2\unit$
  coincides with $\kmil(4)\directsum \kmil(1)/\rho^2\kmil(-1)$
  generated by
  $\nu^2\in \pi_{2+(4)}\f_2\unit$ and $\eta\eta_\Top^2\in \pi_{2+(1)}\f_2\unit$,
  respectively.
\end{lemma}

\begin{proof}
  The slice filtration provides a short exact sequence
  \[ 0 \to \kmil(4)\directsum \kmil(1)/\rho^2\kmil(-1)\to
  \pi_{2}\f_2\unit \to \kerpart \directsum \kmil \to 0 \]
  of $\KMW$-modules, where $\kerpart$ is the $\KMil$-module
  with
  $\kerpart_n= \ker\bigl(\partial^{12}_{2}\colon h^{n+1,n+2}_{12}\to h^{n+2,n+2}\bigr)$.
  Moreover, the $\KMW$-module $\kmil(1)/\rho^2\kmil(-1)$
  coming from $\pi_{2}\f_3\unit$ is generated by $\eta\eta_\Top^2$, and
  the $\KMW$-module $\kmil$ is generated by (the image of) $\eta_\Top^2$.
  Already $2\eta_\Top=0$, whence the extension with respect to
  $\kmil$ is uniquely determined. It remains to determine
  the extension with respect to $\kerpart$. For this purpose,
  consider the commutative diagram
  \begin{equation}\label{eq:extf2sphere}
    \begin{tikzcd}
      \kmil(4)\directsum \kmil(1)/\rho^2\kmil(-1)\ar[r] \ar[d] &
      \pi_{2}\f_2\unit \ar[r] \ar[d] &
      \kerpart \directsum \kmil \ar[d] \\
      {}_{2}\KMil(2)/\{-1\}^3\KMil(-1) \ar[r] &
      \pi_{2}\f_2\kq \ar[r] &
      2\KMil(2)\directsum \kmil
    \end{tikzcd}
  \end{equation}
  induced by the unit map $\unit\to \kq$ on the slice filtration.
  Lemma~\ref{lem:unit-kq-sq} implies that
  the vertical map on the left hand side in diagram~(\ref{eq:extf2sphere})
  projects away from
  $\kmil(4)$ and is induced by the boundary map
  \[ \kmil(1)\iso \directsum_{n\in \NN}h^{n,n+1} \xrightarrow{\partial^2_\infty}
    \directsum_{n\in \NN} {}_{2}H^{n+1,n+1} \iso {}_{2}\KMil(2); \]
  in particular it is surjective.
  The vertical map on the right hand side
  in diagram~(\ref{eq:extf2sphere}) is the identity on the
  summand $\kmil$ and induced by the boundary
  $\partial^{12}_{\infty}\colon \MZ/12\to \Sigma \MZ$ by
  Lemma~\ref{lem:unit-kq-s2}; in particular, it is injective.
  As the proof of Lemma~\ref{lem:pi2-f2kq} shows, there is
  a unique nontrivial extension
  \[ 0\to {}_{2}\KMil(2)/\{-1\}^3\KMil(-1) \to
    \pi_{2}\f_2\kq \to
    2\KMil(2)\to 0\]
  corresponding to the unique nonzero element in
  \[ \Hom_{\KMW}\!\bigl({}_{2}\KMil(2)/\!\{-1\}^3\KMil),{}_{2}\KMil(2)/\!\{-1\}^3\KMil\bigr)
    \!\iso
    \Hom_{\KMW} \!\bigl(\kmil(1)/\!\rho^2\kmil,\kmil(1)/\!\rho^2\kmil\bigr).\]
  Hence there is a unique nontrivial extension
  \[ 0\to \kmil(1)/\rho^2\kmil(1) \to \bigl(\pi_{1}\Sigma^{(1)}\MZ/24\bigr)/\inc^{2}_{24}\rho^2\tau\kmil \to \kerpart \to 0\]
  mapping to the extension
  \[ 0\to {}_{2}\KMil(2)/\{-1\}^3\KMil(-1) \to
    \pi_{2}\f_2\kq \to
    2\KMil(2)\to 0\]
  as prescribed by the diagram~(\ref{eq:extf2sphere}). 
\end{proof}

As a consequence of Lemma~\ref{lem:pi2f2sphere}, every element
of $\pi_{2}\f_{2}\unit$ is 24-torsion, as well as $12\hyper$-torsion.

\begin{lemma}\label{lem:unit-2nd-diff-to-23}
  The second differential
  \[ 
    E^2_{-n+3,1,-n} (\unit)
    \to 
    E^2_{-n+2,3,-n} (\unit)
  \]
  in the $-n$th slice spectral sequence for $\unit$ is trivial.
\end{lemma}

\begin{proof}
  The cofiber sequence
  \[ \f_2\unit\to \f_1\unit\to \s_1\unit \to \Sigma\f_2\unit \]
  induces a long exact sequence on $\KMW$-modules containing
  the portion
  \[ \dotsm \to \pi_{3}\s_1\unit\to \pi_{2}\f_2\unit \to \pi_{2}\f_1\unit \to
    \pi_{2}\s_1\unit\to \dotsm. \]
  The map $\kmil(-2)\iso \pi_{3}\s_1\unit \to \pi_{2}\f_1\unit$ has as
  source a $\kmil$-module on a (unique) generator $g$ in degree
  $\pi_{3-(2)}\s_1\unit$, which comes from the prime field $F_0\subset F$.
  Hence the map is determined by the image of $g$ in
  ${}_{2,\eta}\pi_{2-(2)}\f_2\unit_{F_0}$.
  Lemma~\ref{lem:pi2f2sphere}
  provides that this group is zero in odd characteristic. For $F_0=\QQ$,
  the group ${}_{2}\bigl(h^{3,4}_{24}/\inc^{2}_{24}\rho^2h^{1,2}\bigr)(\QQ) =0$
  is zero, and the group $\kmil_6(\QQ)$ contains a unique nonzero element
  $\rho^6\nu^2$,
  which is mapped to $\eta_\Top^2$ under real realization. It follows that
  the image of $g$ is given by $\rho^2\eta_\Top^2$, as
  determined by the slice $d^1$ differential. The result follows.
\end{proof}

\begin{lemma}\label{lem:unit-3rd-diff-to-24}
  The slice $d^3$ differential
  \[ 
    E^3_{-n+3,1,-n} (\unit)
    \to 
    E^3_{-n+2,4,-n} (\unit)
  \]
  in the $-n$th slice spectral sequence for $\unit$ is trivial.
\end{lemma}

\begin{proof}
The proof of Lemma~\ref{lem:unit-2nd-diff-to-23} verifies this claim.
\end{proof}

\begin{lemma}\label{lem:pi3MZmapstrivially}
  The canonical map $\pi_{3}\s_0\unit\to \pi_{2}\f_1\unit$
  is the zero map.
\end{lemma}

\begin{proof}
  The composition $\pi_{3}\s_0\unit\to \pi_{2}\f_1\unit \to \pi_{2}\s_1\unit$
  is the slice $d^1$ differential, and zero by Corollary~\ref{cor:sq2prinfty2iszero}.
  Hence the map $\pi_{3}\s_0\unit\to \pi_{2}\f_1\unit$ lifts to
  a map $\pi_{3}\s_0\unit \to \pi_{2}\f_2\unit/\pi_{3}\s_1\unit$,
  where the target denotes the cokernel of the (not necessarily injective)
  canonical map
  $\pi_{3}\s_1\unit\to \pi_{2}\f_2\unit$.
  The slice $d^1$ differential $\pi_{3}\s_1\unit \to \pi_{2}\f_2\unit \to \pi_{2}\s_2\unit$ sends the (unique) generator
  $g\in \pi_{3-(2)}\s_1\unit \iso \kmil_0$ to $\bigl(\tau^2\rho^2,\inc^{2}_{12}(\tau\rho^3)\bigr)\in h^{2,4}\directsum h^{3,4}_{12}
  $.
  Lemma~\ref{lem:pi2f2sphere} then provides a short exact sequence
  \[ 0 \to \kmil(4)\directsum \bigl(\pi_{1}\Sigma^{(2)}\MZ/24\bigr)/\inc^{2}_{24}\rho^2\tau\kmil \to
    \pi_{1}\f_2\unit/\pi_{3}\s_1\unit \to \kmil/\rho^2\kmil(-2) \to 0 \]
  of $\KMW$-modules which is determined by the equality
  \[\eta\eta^2_\Top = -1\in h^{0,1}_{24}\iso \pi_{1+(1)}\Sigma^{(2)}\MZ/24.\]
  Since every element in the $\KMW$-module $\pi_{2}\f_2\unit/\pi_{3}\s_1\unit$
  is $24$-torsion,\footnote{The $\KMil$-module $\pi_{3}\MZ/24$ is more accessible than the quite mysterious $\KMil$-module $\pi_{3}\MZ$.}
  the map $\pi_{3}\s_0\unit \to \pi_{2}\f_2\unit/\pi_{3}\s_1\unit$
  factors as
  \[ \pi_{3}\s_0\unit \to \bigl(\pi_{3}\MZ\bigr)/24  \to \pi_{2}\f_2\unit/\pi_{3}\s_1\unit.\]
  The second map fits into a commutative diagram
  \begin{center}
    \begin{tikzcd}
      \pi_{3}\s_0\unit \ar[r] \ar[d,"\pr^\infty_{24}"] &  \pi_{2}\f_2\unit/\pi_{3}\s_1\unit \ar[r]  &
      \pi_{2}\f_1\unit \ar[dd] \\
      \bigl(\pi_{3}\MZ\bigr)/24 \ar[d] \ar[ru] & &  \\
      \pi_{3}\s_0\unit/12\hyper \ar[rr] && \pi_{2}\f_1\unit/12\hyper
    \end{tikzcd}
  \end{center}
  in which the vertical maps are induced by the canonical map
  $\unit\to \unit/12\hyper$ of motivic spectra, factored as a surjection followed
  by an injection on the left hand side. By Lemma~\ref{lem:pi2f2sphere} and
  Lemma~\ref{lem:pi1f2sphere}, every element in
  $\pi_{2}\f_2\unit$ and in $\pi_{1}\f_2\unit$ is $12\hyper$-torsion.
  There results a short exact sequence
  \[ 0 \to \bigl(\pi_{2}\f_{2}\unit\bigr)/12\hyper =  \pi_{2}\f_{2}\unit \to \pi_{2}\f_{2}\unit/12\hyper
    \to {}_{12\hyper}\pi_{1}\f_{2}\unit =\pi_{1}\f_2\unit \to 0 \]
  of $\KMW$-modules. Hence it suffices to show that the image of
  $\pi_{3}\s_0\unit/12\hyper \to \pi_{2}\f_1\unit/12\hyper$ and
  of $\pi_{2}\f_2\unit \to \pi_{2}\f_1\unit \to \pi_{2}\f_1\unit/12\hyper$
  intersect in $\{0\}$.
  Consider the short exact sequence
  \[ 0 \to \pi_{3}\s_0\unit/6\hyper \to \pi_{3}\s_0\unit/12\hyper
    \to \kerpart \to 0 \]
  of $\KMil$-modules, where $\kerpart_n:=
  \ker\bigl(\partial^{2}_{12}\colon h^{n-3,n}\to h^{n-2,n}_{12}\bigr)$
  is isomorphic to the $\KMil$-module bearing the same notation used
  already in the proof of Lemma~\ref{lem:pi2f2sphere}.
  First the image of
  $\pi_{3}\s_0\unit/6\hyper \to \pi_{3}\s_012\hyper \to \pi_{2}\f_1\unit/12\hyper$
  will be shown to intersect the image
  of $\pi_{2}\f_2\unit \to \pi_{2}\f_1\unit/12\hyper$ in $\{0\}$.

  The $\KMil$-module
  $\pi_{3}\s_0\unit/6\hyper\iso \directsum_{n\in \NN} h^{n-3,n}_{12}$
  splits into a $3$-torsion part and a $2$-primary torsion part.
  Consider the $3$-torsion first. If the base field $F$ in
  question contains a primitive third root of unity 
  $\zeta_3\in h^{0,1}_3$, the $\KMil$-module
  $h^{\star-3,\star}_3$ is generated in bidegree $h^{0,3}_3$ by
  the third power of $\zeta_3$. Since the generator exists over the field $F_0(\zeta_3)$ of 3-cohomological dimension at most one,
  it maps trivially, by naturality of the slice differentials for field extensions. Suppose now that $F$ does not contain a primitive third root of unity. The degree two field extension
  $F\hookrightarrow F(\zeta_3)$ induces an injection on the $3$-components of the groups occurring in the differential in question. Thus the $3$-component maps trivially.
  Consider now the $2$-primary torsion part
  $\directsum_{n\in \NN} h^{n-3,n}_{4}$. As a $\KMil$-module, it
  is generated by elements in $h^{0,3}_{4}$ and $h^{1,4}_{4}$.
  If $\sqrt{-1}\notin F$, then $\inc^2_4\colon h^{0,3}=h^{0,3}_{4}\iso \{1,-1\}$.
  If $\sqrt{-1}\in F$, then $h^{0,3}_{4}\iso \{1,-1,\sqrt{-1},-\sqrt{-1}\}$.
  In any case, an element in $h^{0,3}_{4}(F)$ already lives over a
  subfield $E$ of $F$ which is at most a quadratic extension of
  the prime field $F_0$. In this case, the group
  \[ \pi_{2-(3)}\f_2\unit/\pi_{3-(3)}\s_1\unit (E)
    \iso \kmil_7 \directsum h^{4,5}_{24}/\inc^{2}_{24}\rho^2h^{2,3}
    \directsum \kmil_3/\rho^2\kmil_1 \iso \kmil_7(E) \]
  simplifies to a group with either one element, or the two elements
  $\{0,\rho^7\nu^2\}$.
  In the latter case, real realization provides that the element
  $\rho^7\nu^2$, realizing to $\eta^2_\Top$, cannot be hit. 
  An element in $h^{1,4}_{4}(F)$ already lives over a subfield
  $F_0(u)$, where $u$ is a unit in $F$. In particular, $F_0(u)$ has
  virtual cohomological dimension at the prime 2 strictly less than
  four. The relevant target group
  \[ \pi_{2-(4)}\f_2\unit/\pi_{3-(4)}\s_1\unit(F_0(u))
    \iso \kmil_8 \directsum h^{5,6}_{24}/\inc^{2}_{24}\rho^2h^{3,4}
    \kmil_4/\rho^2\kmil_2 \iso \kmil_8(F_0(u)) \]
  again simplifies. If $F_0(u)$ admits a real embedding, real realization
  provides that the single nonzero element $\rho^8\nu^2$, which
  realizes to $\eta_\Top^2$, cannot be hit. If $F_0(u)$ does not
  admit a real embedding, then already $\kmil_8(F_0(u))=0$.
  It follows that the image of 
  $\pi_{3}\s_0\unit/6\hyper \to \pi_{3}\s_012\hyper \to \pi_{2}\f_1\unit/12\hyper$
  intersects the image
  of $\pi_{2}\f_2\unit \to \pi_{2}\f_1\unit/12\hyper$ in $\{0\}$.
  Consider now an element
  \[ x \in \image\bigl(\pi_{3}\s_012\hyper \to \pi_{2}\f_1\unit/12\hyper\bigr)
    \cap \image\bigl(\pi_{2}\f_2\unit \to \pi_{2}\f_1\unit/12\hyper\bigr). \]
  By the preceding argument, $2x=0$. Lemma~\ref{lem:pi2f2sphere} implies
  that the $2$-torsion in $\pi_{2}\f_2\unit/\pi_{3}\s_1\unit$, which injects
  into the $2$-torsion in $\pi_{2}\f_1\unit/12\hyper$, is generated by
  $\nu^2$ and $\eta_\Top^2$ as a $\KMW$-module. In order to show that $x=0$, let
  $y$ be its image in $\pi_{2}\f_1\unit/2$ under the canonical map
  $\unit/12\hyper\to \unit/2$. It has the property that the composition
  $\unit\to \unit/12\hyper\to \unit/2$ of canonical maps is the canonical
  map. This canonical map induces an injection on the sub-$\KMW$-module
  of $\pi_{2}\f_2\unit/\pi_{3}\s_1\unit$ generated by $\nu^2$ and
  $\eta_\Top^2$, because these elements are $2$-torsion.
  Hence it suffices to prove that $y=0\in \pi_{2}\f_1\unit/2$.
  The canonical maps $\unit\to \unit/12\hyper\to \unit/2$
  induce a commutative diagram
  \begin{center}
    \begin{tikzcd}
      \pi_{3}\s_0\unit \ar[r,"\pr^\infty_{24}"] \ar[d] &
      \pi_{3}\s_0\unit/12\hyper \ar[r,"\pr^{24}_2"]
      \ar[d] & \pi_{3}\s_0\unit/2 \ar[d]\\
      \pi_{2}\f_1\unit\ar[r]  &
      \pi_{2}\f_1\unit/12\hyper \ar[r] & \pi_{2}\f_1\unit/2.
    \end{tikzcd}
  \end{center}
  Naturality implies that $y$ is in the image of
  $\pi_{3}\s_0\unit/2\to \pi_{2}\f_1\unit/2$, the vertical map on the right hand side.
  It is determined by the image of the
  (unique) generator in $\pi_{3-(3)}\MZ/2$, which already lives over
  the prime field $F_0\subset F$. In this degree, 
  $ \pi_{2-(3)}\f_2\unit/\pi_{3-(3)}\s_1\unit (F_0)
  \iso \kmil_7(F_0)=\{0,\rho^7\nu^2\}$, as mentioned before. Since
  $2\nu^2=0$, the element $\rho^7\nu^2$ maps to a nonzero element
  in $\pi_{2-(3)}\f_1\unit_{\QQ}/2$; real realization may be invoked as
  well to see this. As a consequence, $y=0$.
  It follows that $x$ is already zero. This concludes the proof that 
  \[\image\bigl(\pi_{3}\s_012\hyper \to \pi_{2}\f_1\unit/12\hyper\bigr)
    \cap \image\bigl(\pi_{2}\f_2\unit \to \pi_{2}\f_1\unit/12\hyper\bigr) = \{0\}. \]
  In particular, the map $\pi_{3}\s_0\unit\to \pi_{2}\f_1\unit$ is zero.
\end{proof}

\begin{corollary}\label{cor:unit-higher-diff-from-30}
  The slice $d^j$ differential
  \[ 
    E^j_{-n+3,0,-n} (\unit)
    \to 
    E^j_{-n+2,j,-n} (\unit)
  \]
  in the $-n$th slice spectral sequence for $\unit$ is zero.
\end{corollary}

\begin{proof}
  This follows from Corollary~\ref{cor:sq2prinfty2iszero} and
  Lemma~\ref{lem:pi3MZmapstrivially}.
\end{proof}

As a consequence of Corollary~\ref{cor:unit-higher-diff-from-30},
the remaining statements of this section 
address actually the $\E^{\infty}$-page of the
slice spectral sequence for $\unit$.

\begin{corollary}
\label{cor:kernel-24}
The kernel of $E^{2}_{-n+2,4,-n}(\unit\to \kq)$ is $h^{n+4,n+4}$ generated by $\beta_{2/2}$.
\end{corollary}
\begin{proof}
This follows from Lemmas~\ref{lem:e2-unit-24} and~\ref{lem:e2-kq-2>3}.
\end{proof} 

\begin{corollary}
\label{cor:kernel-23}
The kernel of the induced map $E^{2}_{-n+2,3,-n}(\unit\to \kq)$ is isomorphic to 
the image of $\pr^{\infty}_{2} H^{n+1,n+2}=H^{n+1,n+2}/2$ in 
$h^{n+1,n+2}/\Sq^2h^{n-1,n+1}$.
\end{corollary}
\begin{proof}
This follows from 
Lemmas~\ref{lem:e2-unit-23},~\ref{lem:unit-kq-sq}, and~\ref{lem:e2-kq-23}.
\end{proof}

\begin{corollary}
\label{cor:kernel-22}
The kernel of $E^{2}_{-n+2,2,-n}(\unit\to \kq)$
is isomorphic to $H^{n+1,n+2}/12$.
\end{corollary}
\begin{proof}
This follows from Lemmas~\ref{lem:e2-unit-22},~\ref{lem:unit-kq-s2},~\ref{lem:e2-kq-22}, and \ref{lem:kernel-aux-1}.
\end{proof}

\begin{lemma}
\label{lem:kernel-aux-1}
The kernel $H^{n+1,n+2}/12$ of $\partial^{12}_{\infty}\colon h^{n+1,n+2}_{12}\to H^{n+2,n+2}$ is contained in the kernel of $\partial^{12}_{2}\colon h^{n+1,n+2}_{12}\to h^{n+2,n+2}$.
\end{lemma}
\begin{proof}
If $x\in h^{n+1,n+2}_{12}$ satisfies $\partial^{12}_{\infty}x=0$ then $\partial^{12}_{2}x=\pr^{\infty}_{2}\partial^{12}_{\infty}x=0$.
\end{proof}

\begin{lemma}
\label{lem:unit-kq-e2-iso}
For $n,k \in \ZZ$ the unit map $\unit\to \kq$ induces an isomorphism on
\begin{align*}
& E^{2}_{-n,m,-n}     &  \\
& E^{2}_{-n+1,m,-n} & m\neq 2,3 \\
& E^{2}_{-n+2,m,-n} & m\neq 2,3,4 \\
& E^{2}_{-n+k,m,-n} & m\leq 1,
\end{align*}
and a surjection on $E^{2}_{-n+1,m,-n}$ for $2\leq m\leq 3$.
\end{lemma}
\begin{proof}
This follows by inspection of $d^{1}$ in Figure~\ref{fig:d1-kq} and Corollary~\ref{cor:unit-kq-e1-iso}, 
with the exception of $E^{2}_{-n+k,1,-n}$ for $k\equiv 3\bmod 4$. 
In the latter case, 
if an element of $h^{n+1-k,n+1}$ is in $\ker(\Sq^{2})$ then it is also in $\ker(\Sq^{2}\Sq^1)$. 
The result for the direct summand $\Sigma^{2,2}\MZ/2$ of $\s_{2}(\unit)$ in Lemma~\ref{lem:unit-kq-s2} implies now the 
isomorphism on $E^{2}_{-n+k,1,-n}$ for all $k\in\ZZ$. 
\end{proof}

\section{The $1$-line}
\label{section:1line}

Let $F$ is a field of exponential characteristic $c\neq 2$, and
set $\Lambda:=\mathbb{Z}[\tfrac{1}{c}])$. 
Then the slice spectral sequence for $\unit_F$ converges conditionally to the homotopy groups of the $\eta$-completion ${\unit}^\wedge_\eta$ \cite[Theorem 3.50]{rso.oneline}.
The calculation of the $E^{2}$-page given in Theorem \ref{theoremE2sphere} shows that $\pi_{1+(n)}{\unit}^\wedge_\eta = 0$ for $n\geq 3$, 
and $\pi_{2+(n)}{\unit}^\wedge_\eta= 0$ for $n\geq 5$. 
As a consequence, 
\[ 
\pi_{1}{\unit}^\wedge_\eta[\tfrac{1}{\eta}]
= 
\pi_{2}{\unit}^\wedge_\eta[\tfrac{1}{\eta}] 
=
0. 
\]
The vanishing $\pi_{1}{\unit}[\tfrac{1}{\eta}]=\pi_{2}{\unit}[\tfrac{1}{\eta}]=0$ from \cite[Theorem 8.3]{roendigs.etainv} and the $\eta$-arithmetic square imply there is a short exact sequence
\[ 
0 
\to \pi_{0}\unit 
\to \pi_{0} \unit^\wedge_\eta\directsum \pi_{0}\unit[\tfrac{1}{\eta}] 
\to \pi_{0}\unit^\wedge_\eta[\tfrac{1}{\eta}]\to 0,
\]
an isomorphism
\[ 
\pi_{1}\unit 
\xrightarrow{\cong} \pi_{1} \unit^\wedge_\eta, 
\]
and an exact sequence
\[
\dotsm \to \pi_{3}{\unit}^\wedge_\eta[\tfrac{1}{\eta}] 
\to \pi_{2}\unit 
\to \pi_{2} \unit^\wedge_\eta 
\to 0
\]
of $\KMW$-modules.

\begin{theorem}
\label{theorem:1line-kq}
  Let $F$ be a field of exponential characteristic $c\neq 2$. The
  unit map $\unit\to \kq$ induces an isomorphism
  $\pi_{0+\bideg}\unit \to \pi_{0+\bideg}\kq$, and a surjection
  $\pi_{1+\bideg}\unit\to \pi_{1+\bideg}\kq$ whose kernel coincides
  with $\KMil_{2-\star}/24$ after inverting $c$. 
  In particular, 
  since $\pi_{1+(n)}\kq=0$ for $n\geq 2$,
  $\pi_{1+(2)}\unit_{\Lambda}\cong \Lambda/24$ and
  $\pi_{1+(n)}\unit_{\Lambda}=0$ for $n\geq 3$.
\end{theorem}

\begin{proof}
  Surjectivity of $\pi_{1} \unit\rightarrow \pi_{1}\kq$ follows from
  Lemma~\ref{lem:unit-kq-e2-iso} for the $\pi_1$-column of
  the $E^{2}=E^\infty$-page; 
  see Lemma~\ref{lem:d2-pi1kq-zero} for
  $E^{2}_{-n+1,j,-n}(\kq)=E^\infty_{-n+1,h,-n}(\kq)$ and
  \cite[Lemma 4.16 and 4.17]{rso.oneline} for
  $E^{2}_{-n+1,j,-n}(\unit)=E^\infty_{-n+1,h,-n}(\unit)$.
  The statement on the kernel of $\pi_{1+(n)} \unit\rightarrow \pi_{1+(n)}\kq$ follows
  from \cite[Theorem 5.5]{rso.oneline}, which applies to $\f_0\KQ$ instead of
  $\kq$.
  Finally, for $n>1$, 
  $\pi_{1+(n)}\kq=\pi_{1+(n)}\KQ
  =W^{2n-1}(F)=0$, although it follows as well from the presentation
  of $\pi_{1}\f_1\kq$ given in Lemma~\ref{lem:pi-f1kq}.
\end{proof}
 
At least after inverting the exponential characteristic $c$
of the base field,
the kernel of $\pi_{1+\bideg}\unit\to \pi_{1+\bideg}\kq$
is generated as a $\KMW$-module by the element $\nu\in \pi_{1+(2)}\unit$.
An unstable representative of $\nu$, already defined over
the integers, is obtained as the Hopf construction on
$\mathrm{SL}_2\simeq \A^2\minus \{0\}$.
Another description of this representative is the attaching map
$\nu \colon S^{7,4}\simeq \A^{4}\minus \{0\}\to \mathbf{HP}^1 \simeq S^{4,2}$
for the quaternionic plane, again defined over any base scheme. 
The equation $\eta\nu=0$ from \cite[Theorem 1.4]{dugger-isaksen.hopf}
(which follows from Theorem~\ref{theorem:1line-kq} as well)
implies that the kernel of $\pi_{1+\bideg}\unit\to \pi_{1+\bideg}\kq$
is in fact a $\KMil$-module.

\section{The $2$-line}
\label{section:2line}

The obvious three elements in $\pi_{2}\unit$ are
\[ \nu^2\in \pi_{2+(4)}\unit, \quad \nu\eta_\Top \in \pi_{2+(2)}\unit, \quad
  \eta^2_\Top \in \pi_{2+(0)}\unit. \]
Even the first one is $2$-torsion by the following
statement.

\begin{lemma}
\label{lem:nusquared-twotorsion}
Over any base scheme we have $(1-\varepsilon)\nu^{2}=2\nu^{2} = 0$. 
\end{lemma}
\begin{proof}
The first equality follows by $\varepsilon\nu=-\nu$ \cite[Corollary 1.5]{dugger-isaksen.hopf}. 
The commutative diagram
\begin{center}
\begin{tikzcd}
S^{7,4}\smash S^{7,4} \ar[rrr,"{\mathrm{twist}=(-1)^3\epsilon^4=-1}"] \ar[d,"{\nu\smash \nu}"] 
&&& 
S^{7,4}\smash S^{7,4} \ar[d,"{\nu\smash \nu}"] \\
S^{4,2}\smash S^{4,2} \ar[rrr,"{\mathrm{twist}=(-1)^{2}\epsilon^4=1}"] 
&&&
S^{4,2}\smash S^{4,2}
\end{tikzcd}
\end{center}
shows that $\nu^{2}=-\nu^{2}$,
which implies the second equation. 
\end{proof}

The passage to $\pi_{2}\unit$ via the slice spectral sequence
requires a discussion of the exact sequence
\[
\dotsm \to \pi_{3}{\unit}^\wedge_\eta[\tfrac{1}{\eta}] 
\to \pi_{2}\unit 
\to \pi_{2} \unit^\wedge_\eta 
\to 0
\]
of $\KMW$-modules given in Section~\ref{section:1line}.
The equality $\pi_{2}\unit=\pi_{2}\unit^\wedge_\eta$ is equivalent
to the surjectivity of the map
$\pi_{3}\unit^\wedge_\eta\directsum \pi_{3}\unit[\tfrac{1}{\eta}] \to
\pi_{3}\unit^\wedge_\eta[\tfrac{1}{\eta}]$ and
the latter $\KMW$-module can be understood completely via the
slice spectral sequence.

\begin{lemma}
\label{lem:eta-surj-ngeq5}
For $n\geq 5$, multiplication with
the first Hopf map $\eta$ induces a surjection
$\pi_{3+(n)}\unit^\wedge_\eta \to \pi_{3+(n+1)}\unit^\wedge_\eta$.
It is an isomorphism when $n\geq 7$.
\end{lemma}

\begin{proof}
The slice spectral sequence, 
the weak equivalence $\slicecomp(\unit)\simeq \unit^\wedge_\eta$, 
and the data in \cite[Proposition 6.1]{zahler} and \cite[Tables A.3.3, A.3.4]{ravenel.green} to identify the relevant slices, imply this result. Here are the details.

Recall that $\s_{0}(\unit)\cong\MZ$ is generated by $\alpha_{1}^0$.
For $q\geq 1$, 
the $q$-dimensional part of $s_{q}(\unit)$ is $\Sigma^{q,q}\MZ/2$ on the generator $\alpha_{1}^{q}$.
For $q\geq 0$ and $q\neq 2$, 
the $q+1$-dimensional part of $s_{q}(\unit)$ is trivial. 
We note that $\s_{2}(\unit)$ contains $\Sigma^{3,2}\MZ/12$ as a direct summand. 
For $q\geq 3$ and $q\neq 4$, 
the $q+2$-dimensional part of $s_{q}(\unit)$ is $\Sigma^{q+2,q}\MZ/2$ generated by $\alpha_{1}^{q-3}\alpha_{3}$.
We note that $\s_{4}(\unit)$ has an additional direct summand of dimension $4+2$; 
it is $\Sigma^{6,4}\MZ/2$ generated by $\beta_{2/2}$. 
For $q\geq 7$, 
the $q+3$-dimensional part of $s_{q}(\unit)$ is $\Sigma^{q+3,q}\MZ/2$ generated by $\alpha_{1}^{q-4}\alpha_{4}$.
The $7$-dimensional part $\Sigma^{7,4}\MZ/240$ of $\s_{4}(\unit)$ is generated by $\alpha_{1}^0\alpha_{4}$.
The $8$-dimensional part of $\s_5(\unit)$ is $\Sigma^{8,5}\MZ/2\{\alpha_{1}^1\alpha_{4}\}\vee \Sigma^{8,5}\MZ/2\{\beta_{2}\}$.
The $9$-dimensional part of $\s_{6}(\unit)$ is $\Sigma^{9,6}\MZ/2\{\alpha_{1}^{2}\alpha_{4}\}\vee \Sigma^{9,6}\MZ/2\{\alpha_{1}\beta_{2}\}$.
For $q\geq 5$ and $q\neq 6$, 
the $q+4$-dimensional part of $s_{q}(\unit)$ is $\Sigma^{q+4,q}\MZ/2$ generated by $\alpha_{1}^{q-5}\alpha_5$.
Finally, 
the $10$-dimensional part of $\s_{6}(\unit)$ is $\Sigma^{10,6}\MZ/2\{\alpha_{1}^1\alpha_5\}\vee \Sigma^{10,6}\MZ/3\{\beta_{1}\}$.

For $q\geq 5$, 
the second column of the $E^{2}$-page of the $q$-th slice spectral sequence for $\unit$ contains only zeros.
Hence all elements in the third column of the corresponding $E^{2}$-page are permanent cycles. 
The description of the generators of the slice summands shows that multiplication by $\eta$ 
-- which is represented by $\alpha_{1}$ -- 
induces an isomorphism between the third column of the $q$-th and the $q+1$-th slice spectral sequence, 
provided $q\geq 7$. 
For $q\in\{5,6\}$ it is surjective. 
Since the same is true for the fourth column, 
and since the differentials are $\eta$-linear, 
there results in an isomorphism resp. surjection between the third column of the infinite terms in the $q$-th and the $q+1$-th slice spectral sequences. 
\end{proof}

\begin{lemma}
\label{lem:connecting-eta-arithmetic-zero}
If $\Char(F)\neq 2$ then $\pi_{3+(n)}\unit^\wedge_\eta \to \pi_{3+(n)}\unit^\wedge_\eta[\tfrac{1}{\eta}]$ is surjective for $n\geq 4$.
\end{lemma}
\begin{proof}
This follows directly from Lemma~\ref{lem:eta-surj-ngeq5} for $n\geq 5$.
Although multiplication with $\eta$ is not surjective as a map $\pi_{3+(4)}\unit^\wedge_\eta\to \pi_{3+(5)}\unit^\wedge_\eta$, 
an inspection of the slices described in detail in the proof of Lemma~\ref{lem:eta-surj-ngeq5} shows that 
multiplication with $\eta^3$ induces a surjection
\[ 
\pi_{3+(4)}\unit^\wedge_\eta
\to 
\pi_{3+(7)}\unit^\wedge_\eta. 
\]
\end{proof}

The following statement, which generalizes
Lemma~\ref{lem:connecting-eta-arithmetic-zero},
uses the fundamental ideal $\fundid\subset \Witt$ 
in the Witt ring of $F$, and its powers $\fundid^\ell\subset \Witt$
which by definition coincide with $\Witt$ for $\ell \leq 0$.

\begin{lemma}
  \label{lem:cok-eta-arithmetic-leq3}
  For $n\in \ZZ$, the cokernel of 
  \[ \pi_{3+(n)}\unit^\wedge_\eta \to \pi_{3+(n)}\unit^\wedge_\eta[\tfrac{1}{\eta}]\]
  is a quotient of $\Witt/\fundid^{-(n-4)}$. 
\end{lemma}

\begin{proof}
  If $n\geq 7$, then the column on the $E^{2}$-page of the $n$-th slice spectral sequence computing $\pi_{3}\unit^\wedge_\eta$ is independent of $n$ by the proof of Lemma~\ref{lem:eta-surj-ngeq5}. 
  It consists of the groups $h^{0,0},h^{1,1},\dotsc$, where $h^{q,q}$ is generated by $\alpha_{1}^{q+n}\alpha_{3}$. 
  By Milnor's conjecture on quadratic forms these groups form the associated graded of the $\fundid$-adic filtration on $\Witt$. 
  The multiplicative structure on the generators imply that for $n\geq 7$, 
  $\pi_{3+(n)}\unit^\wedge_\eta$ is a quotient of the $\fundid$-adic completion $\Witt^\wedge_\fundid$. The same holds thus for 
  $\pi_{3+(n)}\unit^\wedge_\eta[\tfrac{1}{\eta}]$, which is independent of $n\in \ZZ$,
  by Lemma~\ref{lem:eta-surj-ngeq5}.
  This quotient is the zero quotient for $n\geq 4$ by Lemma~\ref{lem:connecting-eta-arithmetic-zero}.

  If $n\leq 4$ there are no motivic cohomology groups of weight $<4-n$ in the column on the $E^{2}$-page of the $n+4$-th slice spectral sequence computing $\pi_{3}\unit^\wedge_\eta$.
  That is,  by the proof of Lemma~\ref{lem:eta-surj-ngeq5}, 
  multiplication with $\alpha_{1}^{3-n}$  hits $h^{-n+4,-n+4},h^{-n+5,-n+5},\dotsc$, but not $h^{0,0},h^{1,1},\dotsc,h^{-n+3,-n+3}$. 
  Thus the cokernel of
  \[ 
    \pi_{3+(n)}\unit^\wedge_\eta 
    \xrightarrow{\eta^{7-n}}
    \pi_{3+(7)}\unit^\wedge_\eta 
    \cong 
    \pi_{3+(n)}\unit^\wedge_\eta[\tfrac{1}{\eta}] 
  \]
  is a quotient of $\Witt/\fundid^{-(n-4)}$.
\end{proof}

Lemma~\ref{lem:cok-eta-arithmetic-leq3} implies that the cokernel of 
\[ 
  \pi_{3}\unit^\wedge_\eta 
  \to 
  \pi_{3}\unit^\wedge_\eta[\tfrac{1}{\eta}] 
\]
is generated, as a $\KMW$-module (in fact even a $\KMW[\eta^{-1}]=\Witt[\eta,\eta^{-1}]$-module), by a single element. What is it?

\begin{lemma}
  \label{lem:cok-eta-arithmetic-generator}
  The image of the third Hopf map $\sigma\in \pi_{3+(4)}\unit$
  generates the cokernel of
  \[\pi_{3}\unit^\wedge_\eta \to \pi_{3}\unit^\wedge_\eta[\tfrac{1}{\eta}]\] as a
  $\KMW$-module.
\end{lemma}

\begin{proof}
  The third Hopf map $\sigma\in \pi_{3+(4)}\unit$ lifts (uniquely)
  to $\pi_{3+(4)}\f_4\unit$ for weight reasons. This lift maps to a generator in
  the cyclic group $\pi_{3+(4)}\s_4\unit$ and hence by
  constitutes by Lemma~\ref{lem:connecting-eta-arithmetic-zero} a
  generator of $\pi_{3}\unit^\wedge_\eta[\tfrac{1}{\eta}]$, hence also
  of its cokernel.
\end{proof}

As a consequence we can improve on Lemma~\ref{lem:connecting-eta-arithmetic-zero}. 

\begin{lemma}
  \label{lem:connecting-eta-arithmetic-surj}
  The canonical map
  \[ \pi_{3}\unit^\wedge_\eta \directsum \pi_{3}\unit[\tfrac{1}{\eta}]\to \pi_{3}\unit^\wedge_\eta[\tfrac{1}{\eta}]\]
  is surjective for all $n\in\Z$.
\end{lemma}
\begin{proof}
  Lemma~\ref{lem:cok-eta-arithmetic-generator} says that the cokernel of the map
  $\pi_{3}\unit^\wedge_\eta \to \pi_{3}\unit^\wedge_\eta[\tfrac{1}{\eta}]$
  is generated by the image of $\sigma$, which in turn is in the image
  of $\pi_{3}\unit[\tfrac{1}{\eta}]\to \pi_{3}\unit^\wedge_\eta[\tfrac{1}{\eta}]$.
\end{proof}

\begin{theorem}
  \label{thm:twoline-slicecomp}
  If $\Char(F)\neq 2$ the canonical map 
  \[\pi_{n+2,n}\unit 
    \to 
    \pi_{n+2,n}\unit^\wedge_\eta
  \]
  is an isomorphism for all $n\in\Z$.
\end{theorem}
\begin{proof}
Lemma~\ref{lem:connecting-eta-arithmetic-surj} implies the map $\pi_{n+2,n}\unit\to \pi_{n+2,n}\unit^\wedge_\eta\directsum \pi_{n+2,n}\unit[\tfrac{1}{\eta}]$ from the $\eta$-arithmetic square is injective. 
The group $\pi_{n+2,n}\unit[\tfrac{1}{\eta}]$ is trivial by \cite[Theorem 8.3]{roendigs.etainv}. 
\end{proof}  

The Twoline Convergence Theorem~\ref{thm:twoline-slicecomp} provides
that the slice spectral sequence for $\unit$ computes the $2$-line of $\unit$. 
Proposition~\ref{prop:eff-kq-conv} implies the same for $\kq$. 
Hence the kernel of $\pi_{2+(n)}\unit\to\pi_{2+(n)}\kq$ can be read off, 
up to extensions, 
from the kernel of the map
\[ 
E^{2}_{n+2,\star,n}(\unit) 
= 
E^{\infty}_{n+2,\star,n}(\unit) 
\to 
E^{\infty }_{n+2,\star,n}(\kq) 
= E^{2}_{n+2,\star,n}(\kq), 
\]
described in Corollaries~\ref{cor:kernel-24} ($\star=4$),~\ref{cor:kernel-23} ($\star=3$) and~\ref{cor:kernel-22} ($\star=2$).
It remains to clarify extensions, which a priori are $\KMW$-module extensions by varying $n$. The crucial contributions to the kernel
of $\pi_{2+(n)}\unit\to\pi_{2+(n)}\kq$ occur in the fourth, third and second
filtration, as Lemma~\ref{lem:unit-kq-e2-iso} already indicates.

\begin{theorem}
\label{thm:ext-kernel}
If $F$ is a field of exponential characteristic $c\neq 2$, the kernel 
of \[\pi_{2+\bideg}\unit \to \pi_{2+\bideg}\kq\] 
is isomorphic to the $\KMil$-module
$ \kmil(4)\directsum \bigl(\pi_{1}\Sigma^{(2)}\MZ\bigr)/24$
after inverting $c$. The (unique) generator for $\kmil(4)$ is
$\nu^2\in \pi_{2+(4)}\unit$.
\end{theorem}

\begin{proof}
  Lemma~\ref{lem:pi2-f2kq} and Lemma~\ref{lem:pi2f2sphere} describe
  the $\KMW$-modules $\pi_{2}\f_2\kq$ and $\pi_{2}\f_2\unit$, respectively.
  The behaviour of the unit map on slices, as given in
  Lemma~\ref{lem:unit-kq-s2} and Lemma~\ref{lem:unit-kq-sq}, provides the
  commutative diagram
  \begin{center}
    \begin{tikzcd}
     0 \ar[r] & \kmil(4)\directsum \bigl(\pi_{1}\Sigma^{(2)}\MZ/24\bigr)/\inc^{2}_{24}\rho^2\tau\kmil \ar[r] \ar[d,"\delta"] & \pi_{2}\f_2\unit \ar[r] \ar[d] & \kmil \ar[r] \ar[d,"\id"]&  0\\
     0 \ar[r] & \bigl(\pi_{0}\Sigma^{(2)}\MZ\bigr)/\partial^2_\infty\rho^2\tau\kmil \ar[r]  & \pi_{2}\f_2\kq \ar[r]  & \kmil \ar[r] &  0
    \end{tikzcd}
  \end{center}
  where the map $\delta$ is the composition
  \begin{center}
    \begin{tikzcd}
      \kmil(4)\directsum \bigl(\pi_{1}\Sigma^{(2)}\MZ/24\bigr)/\inc^{2}_{24}\rho^2\tau\kmil \ar[d,"{\pr}"] \\ \bigl(\pi_{1}\Sigma^{(2)}\MZ/24\bigr)/\inc^{2}_{24}\rho^2\tau\kmil \ar[d,"{\partial^{24}_{\infty}}"] \\
      \bigl(\pi_{0}\Sigma^{(2)}\MZ\bigr)/\partial^2_\infty\rho^2\tau\kmil.
    \end{tikzcd}
  \end{center}
  The snake lemma implies that the kernel of $\pi_{2}\f_2\unit\to \pi_{2}\f_2\kq$
  coincides with the kernel of $\delta$.
  Since the sequence
  \[ 0 \to H^{n+1,n+2}/24\to h^{n+1,n+2}_{24}\to H^{n+2,n+2} \]
  is exact, another application of the snake lemma
  shows that the kernel of $\delta$ is thus the $\KMil$-module
  $\kmil(4)\directsum \bigl(\pi_{1}\Sigma^{(2)}\MZ\bigr)/24$.
  The kernel of $\pi_{2}\f_2\unit\to \pi_{2}\f_2\kq$ coincides with
  the kernel of $\pi_{2}\unit\to \pi_{2}\kq$ because of Lemma~\ref{lem:unit-kq-s0}
  and Lemma~\ref{lem:unit-kq-s1}.
\end{proof}

It is worthwhile to revisit the three obvious elements listed at the
beginning of this section. The element $\nu^2\in \pi_{2+(4)}\unit$
figured prominently in Theorem~\ref{thm:ext-kernel} already.
The element $\eta_\Top\nu\in \pi_{2+(2)}\unit$ lifts by construction to
$\pi_{2+(2)}\f_3\unit\iso \kmil_2$, whence there exists a unique
element $\varphi\in \kmil_2$ with $\eta_\Top\nu=\varphi\nu^2$.
Since these elements are defined over $\Spec(\ZZ)$, base change
implies that $\varphi=\rho^2$, giving us the equation
\begin{equation}\label{eq:etatopnu}
  \eta_\Top\nu = \rho^2\nu^2\in \pi_{2+(2)}\unit
\end{equation}
over every field of characteristic not two. The element
$\eta_\Top^2\in \pi_{2+(0)}\unit$
lifts to an element in $\pi_{2+(0)}\f_2\unit$, as described
in Lemma~\ref{lem:pi2f2sphere}, by construction.
This element has the property that $\eta\eta_\Top^2\in \pi_{2+(1)}\unit$ is the unique
nonzero element of order $2$ in $h^{0,1}_{24}$. Moreover,
$\eta^2\eta_\Top^2=0$. The group $h^{0,1}_{24}=\mu_{24}$ of $24$-th roots
of unity contributing to $\pi_{2+(1)}\unit$ has the following property.

\begin{lemma}
\label{lem:pi31-complex}
Complex realization induces an isomorphism $\pi_{2+(1)}\unit_{\CC} \overset{\iso}{\to} \pi_{3}\sphere$.
\end{lemma}
\begin{proof}
By \cite{levine.sliceconvergence} the slice spectral sequence for $\unit_{\CC}=\unit$ is strongly convergent. 
The nontrivial groups on the $E^1$-page contributing to $\pi_{2+(1)}\unit$ are
$\pi_{2+(1)}\s_{2}(\unit) = h^{0,1}_{12}$ and $\pi_{2+(1)}\s_{3}(\unit)=h^{0,2}=\pi_{2+(1)}\f_{3}(\unit)$ generated by $\eta\eta^{2}_{\Top}$.
Since $\pi_{3+(1)}\s_{2}(\unit) = \pi_{1+(1)}\f_{3}(\unit) = 0$, 
there is a short exact sequence
\[ 
0 
\to 
\pi_{2+(1)}\f_{3}(\unit) 
\to 
\pi_{2+(1)}\f_{2}(\unit) 
\to 
\pi_{2+(1)}\s_{2}(\unit) 
\to 
0.
\]
The main result in \cite{levine.an} shows its complex Betti realization is the short exact sequence 
\[ 
0 
\to 
\ZZ/2 
\to 
\pi_{3}\sphere 
\to 
\ZZ/12 
\to 
0
\]
obtained from the Adams-Novikov filtration. 
The map $\pi_{2+(1)}\f_{3}(\unit) \to\ZZ/2$ is an isomorphism since the complex Betti realization of $\eta\eta^{2}_{\Top}$ is $\eta^{3}_{\Top}$. 
It remains to show that $\pi_{2+(1)}\s_{2}(\unit) \to\ZZ/12$ is an isomorphism. 
This follows from Lemma~\ref{lem:betti-roots} below. 
\end{proof}

\begin{remark}
\label{rem:betti-twoline}
More generally than Lemma~\ref{lem:pi31-complex}, 
the complex Betti realization induces an isomorphism $\pi_{2+(n)}\unit_{\CC} \overset{\cong}{\to} \pi_{2+(n)}\sphere$ for $-1\leq n\leq 4$.
For $n=4$, the generator $\nu^{2}$ maps to $\nu_{\Top}^{2}$.
The groups $\pi_{2+(3)}\unit_{\CC}$ and $\pi_{2+(2)}\unit_{\CC}$ are
trivial by Theorem~\ref{theorem:2line}, 
the vanishing of $\pi_{2+(3)}\kq=\pi_{2+(3)}\KQ$, 
and since $\pi_{2+(2)}\kq=\pi_{2+(2)}\KQ$ is the group of even integers. 
For $n=0$, 
$\pi_{2+(0)}\unit_{\CC}\to \pi_{2+(0)}\kq_{\CC}$ is injective (both $H^{1,2}(\CC)$ and $K^\M_{4}(\CC)$ are divisible), 
and the image is $h^{0,2}$ generated by $\eta_{\Top}^{2}$.
The case $n=-1$ is quite interesting. The slice spectral sequence
for $\unit_{\CC}$ contains a single nonzero element $\eta_\Top^2/\eta$ on the column
of the $E^1$-page contributing to $\pi_{2-(1)}\unit_\CC$. It comes
from $\pi_{2-(1)}\s_1\unit\iso h^{0,2}$, which is isomorphic
to $\pi_{2-(1)}\s_1\kq$. Since $\rho=0$, and in particular $\rho^2=0$,
the element $\eta_\Top^2/\eta$
is a permanent cycle, both for $\unit$ and $\kq$.
The multiplicative structure on the slices of $\unit$ or $\kq$
shows that $\eta\cdot \eta_\Top^2/\eta$ coincides with the unique element
in $\pi_{2+(0)}\s_2\unit_\CC$ detecting $\eta_\Top^2$, whence the
awkward name for this element. 
Consider the commutative diagram
\begin{equation}\label{eq:betti-kq}
  \begin{tikzcd}
    \dotsm \ar[r,"\forget"] &
    \pi_{3-(0)} \kgl_{\CC} \ar[r,"\mathrm{hyper}"] \ar[d,"b_4"] &
    \pi_{2-(1)}\kq_{\CC} \ar[r,"\eta"] \ar[d,"b_3"] &
    \pi_{2-(0)}\kq_{\CC} \ar[r,"\forget"] \ar[d,"b_2"] &
    \pi_{2-(0)}\kgl_{\CC} \ar[r] \ar[d,"b_1"] & \dotsm \\
    \dotsm \ar[r] &
    \pi_{3}\ku \ar[r] &
    \pi_{1}\ko \ar[r,"\eta_\Top"] &
    \pi_{2}\ko \ar[r] &
    \pi_{2}\ku \ar[r] &\dotsm
  \end{tikzcd}
\end{equation}
of long exact sequences, where Theorem~\ref{thm:kq-wood}
provides the top sequence, and the vertical maps are given by
complex Betti realization. The map labelled ``$\eta_\Top$'' is an
isomorphism. The map
$b_2\colon \pi_{2-(0)}\kq_{\CC}\iso \KMil_2(\CC)\directsum \kmil_0 \to \pi_{2}\ko$
is surjective, with the nonzero element in $\kmil_0$ given by
the image of $\eta_\Top^2$ realizing to $\eta_\Top^2$. Since this element is
mapped to $0\in \pi_{2+(0)}\kgl_{\CC}$ via the forgetful map, it is in the
image of multiplication with $\eta$. Hence also
the map labelled ``$b_3$'' is surjective, implying the same for
$\pi_{2-(1)}\unit_\CC\to \pi_{1}\sphere\iso \pi_{1}\ko$.
Complex Betti realization $\pi_{2-(2)}\unit_{\CC}\to \pi_{0}\sphere$ is the
zero map.
\end{remark}

\begin{lemma}
  \label{lem:betti-roots}
  Complex Betti realization induces an isomorphism
  \[ \pi_{1-(1)}\MZ/n \overset{\cong}{\to} \pi_{0}\HZ/n\iso \ZZ/n\ZZ \]
  of groups.
\end{lemma}
\begin{proof}
  The complex Betti realization of $\MZ/n$ is $\HZ/n$ by
  \cite[Proposition 5.10]{levine.comparison}. 
A primitive $n$-th root of unity $\xi\in\CC^{\times}$ provides a stable map $\xi\colon S^{1-(1)}\to \MZ/n$ or alternatively an unstable map $\xi\colon \Spec(\CC)_+ \to \mu_n$,
with target the fiber of multiplication by $n$ on $B \mathbb{G}_m =\PP^\infty$ (the first space in the $\PP^1$-spectrum $\MZ$).
The complex Betti realization of $\Spec(\CC)_+\to \mu_n$ is the choice of a generator of $\ZZ/n$. 
Hence it sends $\xi\colon S^{1+(-1)}\to \MZ/n$ to the unit map $\sphere \to \HZ/n$. 
The result follows now from \cite[Corollary 5.12]{levine.an} after multiplication with $\xi$.
\end{proof}

\appendix 

\section{Modules over Milnor $K$-Theory}
\label{sec:motiv-cohom}

The following result from \cite{rso.oneline} permeates the differential calculations.
\begin{lemma}
\label{lem:differential-module-hom}
The $r$th slice differentials induce a $\KMil=\bigoplus_{n\in\NN}H^{n,n}$-module map
\[ 
\bigoplus_{n\in \ZZ}
d^{r}_{p+n,q,n}(\E)
\colon 
\bigoplus_{n\in \ZZ} E^r_{p+n,q,n}(\E) 
\rightarrow 
\bigoplus_{n\in \ZZ} E^r_{p-1+n,q+r,n}(\E). 
\]
\end{lemma}
The triviality of the higher slice differentials for the relevant
column follows from arguing with $\KMil$-module generators.
An important example is the $\KMil$-module of mod-$2$ Milnor $K$-theory
\[ 
\kmil
= 
\bigoplus_{n\in \ZZ} h^{n,n}. 
\]
It is generated by the nontrivial element $1\in h^{0,0}$.
More generally, 
Voevodsky's solution of the Milnor conjecture implies that for a fixed integer $k\geq 0$, 
the $\KMil$-module
\[ 
\bigoplus_{n\in \ZZ} h^{n-k,n} 
\]
is generated by the unique nontrivial element $\tau^k\in h^{0,k}$ in bidegree $(0,k)$. 
Other examples of generators for $\KMil$-modules are given by \cite[Theorem 3.3]{ovv} for fields of characteristic zero and \cite[Theorem 2.1]{merkurjev-suslin.rost} for fields of odd characteristic. 
A central example is the kernel of $\Sq^{2}\colon h^{n-2,n}\to h^{n,n+1}$, 
since it corresponds 
to multiplication with the pure symbol $\{-1,-1\}$).

\begin{lemma}\label{lem:cone-alpha}
  Let $\alpha\in \pi_{s+(w)}\unit$, and $\E$ a motivic spectrum.
  The cofiber sequence
  \[ \Sigma^{s+(w)}\E \xrightarrow{\alpha} \E \to \E/\alpha \to \Sigma^{s+1+(w)}\E\]
  provides
  a short exact sequence
  \[ 0 \to \bigl(\pi_{m+\bideg}\E\bigr)/\alpha \to \pi_{m+\bideg}\E/\alpha
    \to {}_{\alpha}\pi_{m-s-1+\bideg}\E \]
  of $\KMW$-modules for every $m\in \ZZ$.
\end{lemma}  

\begin{proof}
  This claim follows directly from the definition of $\E/\alpha$.
\end{proof}

\begin{lemma}
  Let $n>0$ be a natural number. The canonical maps induce
  a short exact sequence
  \[ 0\to \bigl(\pi_{1+\bideg}\MZ\bigr)/{2^n} \to \pi_{1+\bideg} \MZ/2^n
    \to {}_{2^n} \KMil \to 0 \]
  of $\KMil$-modules.
\end{lemma}

\begin{proof}
 This result is a particular case of Lemma~\ref{lem:cone-alpha}.
\end{proof}

Specializing even further, let us consider mod-$4$ coefficients.  
Let $F_{0}$ denote the prime field of $F$.

\begin{lemma}
\label{lem:generators-motcohommod4}
For $k\geq 0$ a fixed natural number, consider the $\KMil$-module
\[ 
\bigoplus_{n\in \ZZ} h^{n-k,n}_{4}. 
\]
If $k=2n$ is even, 
it is generated by a single element $\theta^n\in h^{0,k}_{4}=\ZZ/4$ 
where $\theta \in h^{0,2}_{4}$. The element $\theta$
is defined over $F_{0}$. 
If $k$ is odd,
it is generated by elements in bidegrees $(0,k)$ and $(1,k+1)$; 
every generator in bidegree $(0,k)$ is defined over $F_{0}$ or 
$F_{0}[\sqrt{-1}]$, 
whereas for every generator $\overline{g}\in h^{1,k+1}_{4}(F)$ 
there exists a unit $g\in F^{\times}$ such that $\overline{g}$ is 
defined over the field $F_{0}(g)$.
\end{lemma}

\begin{proof}
The boundary map in motivic cohomology for the change of coefficients exact sequence 
\[ 
0 \to \ZZ/2\to \ZZ/4\to\ZZ/2\to 0
\]
is the motivic Steenrod operation $\Sq^1\colon h^{n-k,n}\to h^{n-k+1,n}$. 
This map is trivial if $k$ is even, 
and if $k$ is odd it exchanges one factor of $\tau$ for $\rho$.

Assuming $k$ is even there is a short exact sequence
\[ 
0\to h^{n-k,n}/\Sq^1h^{n-k-1,n} \to h^{n-k,n}_{4} \to h^{n-k,n}\to 0. 
\]
Both of the outer modules are generated by a single element in bidegree $(0,k)$, 
and there is a short exact sequence
\[ 
0 \to \ZZ/2 = h^{0,k} \to h^{0,k}_{4} \to h^{0,k} =\ZZ/2 \to 0.
\]
Comparison with the case $k=2$ shows this is a non-split extension over any field. 
The above proves the statement for $k$ even.
  
Assuming $k$ is odd, there is a short exact sequence  
\[ 
0\to h^{n-k,n} \to h^{n-k,n}_{4} \to \ker(h^{n-k,n}\xrightarrow{\Sq^1}h^{n-k+1,n})\to 0. 
\]
As modules,  
$h^{n-k,n}$ is generated in bidegree $(0,k)$ and $\ker(h^{n-k,n}\xrightarrow{\Sq^1}h^{n-k+1,n})$ is generated in $(0,k)$ or in $(1,k+1)$. 
More precisely, 
if $\rho=0$ over $F$,
it is generated in bidegree $(0,k)$, and the generator already
lives over the smallest subfield of $F$ containing $\sqrt{-1}$. 
If $\rho\neq 0$ over $F$, 
there exists a set of units $G\subset F$ such that for $g\in G$,
$\Sq^{1}(\tau^{k}\overline{g}) = \tau^{k-1}\rho\overline{g}=0$ and every element in $\ker(\Sq^1\colon h^{n-k,n}\rightarrow h^{n-k+1,n})$ is a finite $\KMil$-linear combination of elements $\tau^{k}\overline{g}$,  
for $g\in G$. 
Finally,   
in the $\kmil$-theory long exact sequence for the field extension $F\subset F[\sqrt{-1}]$ every $\overline{g}$ is the image of the transfer map of some $\overline{h}$,  
$h\in F[\sqrt{-1}]$.
The field $F_{0}(g^\prime)$, 
where $g^\prime$ is the transfer of $h$,
has the property that $\tau^{k}\overline{g}=\tau^{k}\overline{g^\prime}$ over $F$ and $\Sq^1(\tau^{k}\overline{g})=0$.
\end{proof} 

Next, we consider mod-$8$ coefficients.  
Let $F_{0}$ denote the prime field of $F$.

\begin{lemma}
\label{lem:generators-motcohommod8}
For $k\geq 0$ a fixed natural number, consider the $\KMil$-module
\[ 
\bigoplus_{n\in \ZZ} h^{n-k,n}_{8}. 
 \]
 If $k=2n$ is even, 
 it is generated by a single element $\omega^n\in h^{0,k}_{8}=\ZZ/8$, where
 $\omega\in h^{0,2}_{8}$ is defined over $F_{0}$. 
\end{lemma}

\begin{proof}
The change of coefficients exact sequence 
\[ 0 \to \ZZ/4\to \ZZ/8\to\ZZ/2\to 0 \]
induces a long exact sequence of motivic cohomology groups
\[ 
0 
\to 
h^{0,k}_{4}
\to 
h^{0,k}_{8}
\to 
h^{0,k} 
\to 
h^{1,k}_{4}
\to 
\dotsm. 
\]
Assuming $k$ is even, 
Lemma~\ref{lem:generators-motcohommod4} implies $h^{0,k}_{4}\iso \ZZ/4$. 
Hence $h^{0,k}_{8}$ has order at most $8$, and order precisely $8$ if $h^{0,k}\to h^{1,k}_{4}$ is trivial. 
It suffices to prove this map is trivial over $F_{0}$. 
For $k=2$ note that $h^{0,2}_{8}$, 
the kernel of multiplication by $8$ on $H^{1,2}$, 
is $\ZZ/8$ when $\Char(F_{0})\neq 2$. 
The group $h^{0,2}_{4}$, 
i.e., 
the kernel of multiplication by $4$ on $H^{1,2}$, 
is $\ZZ/4$ when $\Char(F_{0})\neq 2$. 
The projection map $h^{0,2}_{8}\to h^{0,2}_{4}$ is induced by multiplication by $2$ on $H^{1,2}$, 
and hence surjective over prime fields. 
For $k>2$ even, 
using the commutative diagram
\begin{center}
\begin{tikzcd}
h^{0,2}_{8}\times h^{0,k-2}_{8} \ar[d,"{\mathrm{mult}}"] \ar[r,"{\pr^{8}_{4}}"] & h^{0,2}_{4}\times h^{0,k-2}_{4} \ar[d,"{\mathrm{mult}}"] \\
h^{0,k}_{8} \ar[r,"{\pr^{8}_{4}}"] &h^{0,2}_{4} 
\end{tikzcd}
\end{center}
we conclude the claimed surjection by induction and 
Lemma~\ref{lem:generators-motcohommod4}.
It follows that $\partial^{2}_{4}h^{0,k}=\{0\}$. 
To prove $\partial^{2}_{4}h^{n-k,n}=\{0\}$ for all $n>k$, let 
$\tau^kx\in h^{n-k,n}$, with $x\in h^{n-k,n-k}$.
We have
\[ 
\partial^{2}_{4}(\tau^kx)
=
\partial^{2}_{4}(\tau^k)\cdot x 
= 
0\cdot x 
=
0, 
\]
since $\partial^{2}_{4}$ is a $\KMil$-module map.
Thus there is a short exact sequence of $\KMil$-modules
\[ 
0 
\to 
\bigoplus_{n\in \ZZ}h^{n -k,n}_{4}/\partial^{2}_{4}h^{n-k-1,n}
\to 
\bigoplus_{n\in \ZZ}h^{n -k,n}_{8}
\to 
\bigoplus_{n\in \ZZ} h^{n -k,n}_{}
\to 
0 
\]
with outer terms generated in bidegree $(0,k)$. 
Hence so is the middle term, concluding the case where
$k$ is even.
\end{proof} 

\begin{corollary}\label{cor:partial42zerooneven}
  Let $k$ be an even natural number. Then
  $\partial^{4}_{2}\colon h^{n-k,n}_{4}\to h^{n-k+1,n}$ is 
  the zero map.
\end{corollary}

\begin{proof}
  This follows as in the proof of Lemma~\ref{lem:generators-motcohommod8},
  using a generator in $h^{0,k}_{4}$
  from~\ref{lem:generators-motcohommod4}.
\end{proof}

\begin{lemma}
\label{lem:modcohom8-surjective}
Multiplication with a generator $\omega\in h^{0,2}_{8}$ induces an isomorphism
\[ 
\bigoplus_{n\in \ZZ}h^{n-k,n}_{8}
\overset{\cong}{\to}
\bigoplus_{n\in \ZZ}h^{n-k,n+2}_{8} 
\]
of $\KMil$-modules for every natural number $k$.
\end{lemma}
\begin{proof}
First, we show that multiplication with a generator $\theta\in h^{0,2}_{4}$ induces an isomorphism
\[ 
\bigoplus_{n\in \ZZ}h^{n-k,n}_{4}
\overset{\cong}{\to}
\bigoplus_{n\in \ZZ}h^{n-k,n+2}_{4}.
\]
In effect, 
for $\tau^2\in h^{0,2}$ and $\theta\in h^{0,2}_{4}$, 
consider the diagram of Bockstein sequences
\begin{center}
    \begin{tikzcd}

\dotsm \ar[r,"{\Sq^1}"] & 
h^{m,n+1} \ar[r]\ar[d,"\tau^2"] &
h^{m,n+1}_{4} \ar[r] \ar[d,"\theta"] &
h^{m,n+1} \ar[r,"{\Sq^1}"]\ar[d,"\tau^2"] &
h^{m+1,n+1} \ar[r]\ar[d,"\tau^2"] &
\dotsm \\ 
\dotsm \ar[r,"{\Sq^1}"] & 
h^{m,n+3} \ar[r] &
h^{m,n+3}_{4} \ar[r]  &
h^{m,n+3} \ar[r,"{\Sq^1}"] &
h^{m+1,n+3} \ar[r] &
\dotsm. 
\end{tikzcd}
\end{center}
Recall $\MZ/4\to \MZ/2$ is a map of motivic ring spectra, 
and the image of $\theta$ is $\tau^2$.
Whence the middle square commutes. 
Moreover, 
for $\Sq^1\colon h^{n-k,n} \to h^{n-k+1,n}$ we have
\[ 
\Sq^1(\tau^kx) 
= 
\begin{cases} 
\tau^{k-1}\rho x & k \equiv 1\bmod 2 \\
0 & k\equiv 0 \bmod 2,
\end{cases}
\]
which implies the right-hand square commutes.
For commutativity of the left-hand square, we use that $\MZ/2\to \MZ/4$ is a $\MZ/4$-module map.
Since multiplication with $\tau^2$ is an isomorphism, so is multiplication with $\theta$ according to the five lemma.

The above feeds into the case of $\Z/8$-coefficients via the diagram of Bockstein sequences
\begin{center}\begin{tikzcd}
\dotsm \ar[r,"{\partial^{2}_{4}}"] & 
h^{m,n+1}_{4} \ar[r]\ar[d,"{\theta}"] &
h^{m,n+1}_{8} \ar[r] \ar[d,"{\omega}"] &
h^{m,n+1}_{2} \ar[r,"{\partial^{2}_{4}}"]\ar[d,"{\tau^2}"] &
h^{m+1,n+1}_{4} \ar[r]\ar[d,"{\theta}"] &
\dotsm \\ 
\dotsm \ar[r,"{\partial^{2}_{4}}"] & 
h^{m,n+3}_{4} \ar[r] &
h^{m,n+3}_{8} \ar[r]  &
h^{m,n+3} \ar[r,"{\partial^{2}_{4}}"] &
h^{m+1,n+3}_{4} \ar[r] &
\dotsm. 
\end{tikzcd}\end{center}
Commutativity of this diagram follows for the same reasons as above since the $\KMil$-module map $\partial^{2}_{4}\colon h^{n-k,n} \to h^{n-k+1,n}_{4}$ satisfies
\[ 
\theta\cdot \partial^{2}_{4}(\tau^kx) 
= 
\partial^{2}_{4}(\tau^{k+2}x).
\]
Using the five lemma we conclude that multiplication with $\theta\in h^{0,2}_{4}$ is an isomorphism.
\end{proof}

{\bf Acknowledgements.}
The authors gratefully acknowledge support by the DFG priority program 1786 "Homotopy theory and algebraic geometry" 
and the RCN Frontier Research Group Project no.~250399 "Motivic Hopf Equations."
Work on this paper took place at the Institut Mittag-Leffler in Djursholm, 
the Hausdorff Research Institute for Mathematics in Bonn;
we thank both institutions for providing excellent working conditions, hospitality, and support.
We also acknowledge the support of the Centre for Advanced Study at the Norwegian Academy of Science and Letters in Oslo, 
which funded and hosted our research project ``Motivic Geometry" during the 2020/21 academic year.
The third author acknowledges support from the Humboldt Foundation and the Radboud Excellence Initiative.

\begin{small}

\begin{center}
Institut f\"ur Mathematik, Universit\"at Osnabr\"uck, Germany.\\
e-mail: oliver.roendigs@uni-osnabrueck.de
\end{center}
\begin{center}
Institut f\"ur Mathematik, Universit\"at Osnabr\"uck, Germany.\\
e-mail: markus.spitzweck@uni-osnabrueck.de
\end{center}
\begin{center}
Department of Mathematics, University of Oslo, Norway.\\
e-mail: paularne@math.uio.no
\end{center}
\end{small}
\end{document}